\documentclass[11pt,reqno]{amsart}
\usepackage{amsmath,amsfonts,amssymb,amscd,amsthm,amsbsy,bbm, epsf,calc,  pdfsync}
\usepackage[pdftex]{graphicx}
\usepackage{color, float}
\usepackage{datetime}
\usepackage{bm}
\numberwithin{equation}{section}

\newcounter{hours}\newcounter{minutes}

\theoremstyle{definition}
\newtheorem{thm}{Theorem}[section]

\newtheorem{lem}[thm]{Lemma}
\newtheorem{cor}[thm]{Corollary}

\theoremstyle{remark}                  %% For unnumbered Remarks, etc.
\newtheorem{rem}[thm]{Remark}

\theoremstyle{definition}
\newtheorem{DEF}[thm]{Definition}

\def\S{{\mathbb S}}

\def\real{{\mathbb R}}

\def\diam{\textnormal{diam}}

\def\det{\textnormal{det}}

\def\union{\bigcup}

\newcommand{\norm}[1]{\lvert#1\rvert}

\newcommand{\inner}[2]{\langle #1,#2\rangle}

\def\xbar{\bar{x}}
\def\Dbar{\bar{D}}
\def\Omegabar{\bar{\Omega}}
\def\pbar{\bar{p}}
\def\xhat{\widehat{x}}

\def\xbarhat{\widehat{\bar{x}}}

\def\Mbar{\bar{M}}
\DeclareMathOperator{\cl}{cl}

\def\calL{\mathcal{L}}
\newcommand{\Leb}[1]{\left\vert{#1}\right\vert_{\calL}}

\def\qcheck{\widehat{q}}
\def\pcheck{\widehat{p}}
\newcommand\g[1]{g_{#1}}
\def\qbar{\bar{q}}
\newcommand\gbar[1]{\bar{g}_{#1}}

\DeclareMathOperator{\dist}{dist}
\DeclareMathOperator{\MTW}{MTW}
\def\pbardot{\dot{\bar{p}}}

\newcommand{\MTWcoord}[4]{-(c_{#1#2, \bar{r}\bar{s}}-c_{#1#2, \bar{t}}c^{\bar{t}, s}c_{s, \bar{r}\bar{s}})c^{\bar{r}, #3}c^{\bar{s}, #4}}
\DeclareMathOperator{\ch}{conv}
\DeclareMathOperator{\spt}{spt}
\DeclareMathOperator{\aff}{aff}
\def\gammabar{\bar{\gamma}}
\DeclareMathOperator{\interior}{int}
\def\halfspace{H}
\def\arbitrary{A}
\def\sublevelset{S}
\def\ellipsoid{E}
\def\normal{N}

\DeclareMathOperator{\dVol}{dVol}
\def\mubar{\bar{\mu}}
\def\innerdom{\spt{\rho}}
\def\outerdom{\Omega}
\def\innertarget{\spt{\rhobar}}
\def\outertarget{\Omegabar}
\def\arbitrarybar{\bar{\arbitrary}}
\def\rhobar{\bar{\rho}}

\def\xbardot{\dot{\xbar}}

\def\vectn{q}

\newcommand{\cone}[2]{I_{\pbar_0}(#1, #2)}
\newcommand{\proj}[2]{\pi_{#1}{\left(#2\right)}}

\def\pmax{p_{\max}}
\newcommand{\coord}[2]{\left[#1\right]_{#2}}
\newcommand{\outerdomcoord}[1]{\coord{\outerdom}{#1}}
\newcommand{\innerdomcoord}[1]{\coord{\innerdom}{#1}}
\newcommand{\outertargetcoord}[1]{\coord{\outertarget}{#1}}

\newcommand{\cExp}[2]{exp^c_{#1}({#2})}

\def\subzero{\sublevelset_0}
\def\subzerocoord{\coord{\subzero}{\xbar_0}}

\def\contactcoord{\coord{\contact}{\xbar_0}}
\def\mountain{m}
\def\mountainhat{{m}}
\def\mountaintilde{\tilde{m}}

\def\contact{\sublevelset_0}
\newcommand\transform[2]{\mathcal{M}_{#1, #2}}
\newcommand\transforminv[2]{\mathcal{M}_{#1, #2}^{-1}}
\newcommand\transformpm[2]{\mathcal{M}^{\pm 1}_{#1, #2}}
\newcommand\transformadj[2]{\mathcal{M}^*_{#1, #2}}

\newcommand{\gnorm}[2][]{\lvert #2\rvert_{\g{#1}}}
\newcommand{\gbarnorm}[2][]{\lvert #2\rvert_{\gbar{#1}}}
\newcommand{\innerg}[3][x_0, \xbar_0]{\g{#1}\left(#2, #3\right)}
\newcommand{\innergbar}[3][\xbar_0]{\gbar{#1}\left(#2, #3\right)}
\newcommand{\linear}[3][]{l^{#2}_{#1}(#3)}
\def\xmax{x_{\max}}
\newcommand{\e}[2][]{v^{#1}_{#2}}

\newcommand{\ehat}[2][]{\widehat{v}{}^{#1}_{#2}}
\newcommand{\vectsp}[1][]{V_{#1}}
\newcommand{\DDbar}[2]{-D \Dbar c(#1, #2)}
\newcommand{\DDbarinv}[2]{\left[ -D\Dbar c(#1, #2)\right]^{-1}}

\newcommand{\DbarDinv}[2]{\left[ -\Dbar Dc(#1, #2)\right]^{-1}}
\newcommand{\tansp}[2]{T_{#1}{#2}}
\newcommand{\cotansp}[2]{T^*_{#1}{#2}}
\newcommand{\cotanspM}[1]{\cotansp{#1}{M}}
\newcommand{\tanspM}[1]{\tansp{#1}{M}}
\newcommand{\cotanspMbar}[1]{\cotansp{#1}{\Mbar}}
\newcommand{\tanspMbar}[1]{\tansp{#1}{\Mbar}}
\def\mountainhat{\widehat{\mountain}}

\newcommand{\extremalhalfspace}[2][x_e, \xbar_0,\e{0}]{\halfspace^{#2}_{#1}}
\newcommand{\w}[2][]{w^{#1}_{#2}}
\newcommand{\wbar}[2][]{\bar{w}{}^{#1}_{#2}}
\newcommand{\raisebar}[1]{{#1}^{\sharp}}
\newcommand{\plane}[2]{\Pi^{#1}_{#2}}
\newcommand{\subcoord}[1][\xbar]{\coord{\sublevelset}{#1}}

\newcommand{\supsegment}[2]{ l(#1,#2) } % 1 = domain 2 = direction of segment
\def\bigsublevelset{\sublevelset^{\textrm{big}}}
\def\ttil{\tilde{t}}
\begin{document}

\title{On the local geometry of Maps with c-convex potentials}

\author{Nestor Guillen}
\address[Nestor Guillen]{Department of Mathematics, University of California at Los Angeles Department of Mathematics}
\thanks{N. Guillen is partially supported by a National Science Foundation grant DMS-1201413.}
\email{nestor@math.ucla.edu} 

\author{Jun Kitagawa}
\address[Jun Kitagawa]{University of British Columbia and Pacific Institute for Mathematical Sciences, Vancouver, Canada}
\thanks{J. Kitagawa is partially supported by a Pacific Institute for the Mathematical Sciences Postdoctoral Fellowship.}
\email{kitagawa@math.ubc.ca}

\begin{abstract}
    We identify a condition for regularity of optimal transport maps that requires only three derivatives of the cost function, for measures given by densities that are only bounded above and below. This new condition is equivalent to the weak Ma-Trudinger-Wang condition when the cost is $C^4$. Moreover, we only require (non-strict) $c$-convexity of the support of the target measure, removing the hypothesis of strong $c$-convexity in a previous result of Figalli, Kim, and McCann, but at the added cost of assuming compact containment of the supports of both the source and target measures. %A key tool is a geometric consequence of the weak Ma-Trudinger-Wang condition derived from the pseudo-Riemannian framework of Kim and McCann, and an extension of the localization property of solutions discovered by Caffarelli for the Euclidean distance squared.
\end{abstract}

\date{\today}
%\keywords{Keywords}
%\subjclass[2000]{35J99, %pde other
%45J05, %PIDE
%47G20, %integro-differential operators
%49L25, %differential games
%49N70, %optimal control viscosity solutions
%60J75, %jump processes
%93E20 %optimal stoch. control
%}

\date{}
\maketitle
\markboth{N. Guillen and J. Kitagawa}{The local geometry of maps with c-convex potentials}
%%%%%%%%%%%%%%%%%%%%%%%%%%%%%%%%%%%%%%%%%%%%%%
%%%%%%%%%%%%%%%%%%%%%%%%%%%%%%%%%%%%%%%%%%%%%%

\section{Introduction}

% From optimal transport to c-convex potentials, need to cite: Brenier, Gangbo-McCann... First: talk about the ``assumptions'', state main theorem, then go into role of potentials
    \subsection{Statement of main result}\label{subsection: main result}
    This paper is concerned with the regularity of solutions to the optimal transport problem under certain conditions. Namely, suppose that $\outerdom$ and $\outertarget$ are subsets of Riemannian manifolds $M$ and $\Mbar$, $\mu=\rho\dVol_M$, $\mubar=\rhobar\dVol_{\Mbar}$ are probability measures, and $c:\outerdom^{\cl}\times \outertarget^{\cl}\to \real$ is a \emph{cost function} satisfying conditions~\eqref{Twist},~\eqref{Nondeg},~\eqref{Convdm}, and~\eqref{QConv} (described in Subsection~\ref{subsection: setup}). Then, we wish to discuss the regularity and injectivity properties of solutions to the optimal transport problem, i.e. measurable maps $T:\outerdom\to\outertarget$ such that $T_\#\mu=\mubar$, and which satisfy
    \begin{equation}\label{OT problem}
         \int_{\outerdom} c(x,T(x)) d\mu(x)=\inf_{S_\# \mu =\mubar}\int_{\outerdom} c(x,S(x)) d\mu(x).
    \end{equation}
    Since $\mu$ is absolutely continuous with respect to $\dVol_M$ and $c$ satisfies~\eqref{Twist} and~\eqref{Nondeg}, it can be shown that~\eqref{OT problem} admits a solution, and moreover, this solution can be characterized as 
    \begin{equation}\label{eqn: solution characterization}
         T(x)=\cExp{x}{Du(x)}
    \end{equation}
    defined $\dVol_M$ almost everywhere for some $c$-convex real valued function $u$, known as a $c$-convex potential of $T$ (for details see~\cite[Chapter 9]{Vil09}, see Subsection~\ref{subsection: setup} for relevant definitions). This was first discovered by Brenier for $c(x,y)=-x\cdot y$ in $\mathbb{R}^n$ (see~\cite{Bre87}), and then later extended by Gangbo and McCann to more general cost functions (see~\cite{GM96}). A similar existence result and characterization on Riemannian manifolds was discovered by McCann (\cite{Mc01}). The issue of regularity of such solutions is markedly more difficult.

\bigskip

    The main result of this paper is the following theorem. We note that we only require $C^3$ regularity of the cost function instead of the usual $C^4$, and even then the slightly weaker notion that mixed derivatives of the cost should be continuously differentiable. This along with the relevant hypotheses~\eqref{Twist},~\eqref{Nondeg},~\eqref{Convdm}, and~\eqref{QConv} are described in Subsection~\ref{subsection: setup}.

    \begin{thm}\label{thm: OT main theorem}
         Suppose $\mu=\rho\dVol_{M}$ and $\mubar=\rhobar\dVol_{\bar{M}}$ for non-negative measurable functions $\rho$ and $\rhobar$, and let $c:\outerdom^{\cl}\times\outertarget^{\cl}\to\real$ be $C^3$ in the sense of Subsection~\ref{subsection: setup}. Suppose $c$, $\innerdom$, $\innertarget$, $\outerdom$, and $\outertarget$ satisfy the properties~\eqref{Twist},~\eqref{Nondeg},~\eqref{Convdm}, and~\eqref{QConv}. Additionally, suppose  there are constants $0<\alpha_1\leq \alpha_2<+\infty$ such that
         \begin{equation*}
              \alpha_1\leq \frac{\rho(x)}{\rhobar(\xbar)}\leq \alpha_2,\qquad a.e.\; (x,\xbar) \in \innerdom\times \innertarget.
         \end{equation*} 
         Then, the optimal transport map $T$  in~\eqref{OT problem} is injective and continuous in the interior of $\innerdom$.
\end{thm}

    It is well known that strict $c$-convexity implies differentiability for Aleksandrov solutions (by a standard compactness argument), this goes back to \cite{C92} in the case of the Monge-Ampere equation and convexity. See also \cite{FKM11} for the details in the case of a general cost and $c$-convexity. Thus, most of the paper will be devoted to the following result which implies Theorem~\ref{thm: OT main theorem} after combining with Theorem~\ref{thm: Brenier implies Aleksandrov} and the characterization~\eqref{eqn: solution characterization}.    

    \begin{thm}\label{thm: main}
         Consider $\mu$ and $\mubar$ and a cost $c: \outerdom^{\cl}\times\outertarget^{\cl} \to \mathbb{R}$, $\innerdom$, $\innertarget$, $\outerdom$, and $\outertarget$ satisfying the assumptions of Theorem~\ref{thm: OT main theorem}. Then any $c$-convex $u$ that is an Aleksandrov solution of~\eqref{OT problem} (see Definition~\ref{def: Aleksandrov solution}) is strictly $c$-convex in the interior of $\innerdom$.
    \end{thm}

% Previous results: Caffarelli, Urbas, Gutierrez, Ma-Trudinger-Wang, Loeper, Figalli-Loeper, Liu, Kim-McCann, Figalli-Kim-McCann.

    \subsection{History of regularity theory for the optimal transport problem}\label{subsection: history}
    We give a brief overview of the history of regularity theory of this problem here. Details about all relevant conditions can be found in Section~\ref{section: basic definitions} below. Due to the characterization~\eqref{eqn: solution characterization}, it can be seen that a potential function solving the optimal transport problem satisfies the Monge-Amp\'ere type equation, 
    \begin{equation}\label{eqn: main classical}
         \det\left (D^2u(x)+D^2_{x,x}c(x, T(x))\right ) =\lvert \det\;{D^2_{x, \xbar}c(x, T(x))}\rvert \rho(x)/\rhobar(T(x))
    \end{equation}
    for $\dVol_M$ almost-every $x$. In particular, regularity techniques of elliptic partial differential equations can be applied towards deducing regularity of the optimal transport problem.

    There are two branches in  the regularity theory for the optimal transport problem, each of which originated with the regularity theory for the classical Monge-Amp\'ere equation. This is not surprising, for when $\outerdom$ and $\outertarget$ are convex subsets of $\mathbb{R}^n$ and $c(x,y)=-x\cdot y$ is the inner product,~\eqref{eqn: main classical} reduces to the classical Monge-Amp\'ere equation for a convex function $u$
    \begin{equation*}
         \det(D^2u) = \rho(x)/\rhobar(D u(x)).
    \end{equation*}

    The first branch utilizes the continuity method and the existence of \emph{a priori} $C^2$ estimates for solutions, in order to show $C^{2, \alpha}$ regularity. Such estimates originate with the work of Pogorelov (see~\cite{Pog1971}) and were used by Urbas (~\cite{Urb97}) in showing regularity for the so-called second boundary value problem for Monge-Amp\'ere. For a long time regularity was open for all other costs, the first result for a general class of costs was by Ma, Trudinger, and Wang in~\cite{MTW05},  who derived new a priori estimates and showed the existence of solutions which were $C^{2,\alpha}$ in the interior,  followed then by work of Trudinger and Wang in~\cite{TW09} to show global $C^{2, \alpha}$ regularity of solutions to the optimal transport problem for costs and domains satisfying certain conditions (see Section~\ref{subsection: MTW}). See also the following by Liu, Trudinger, and Wang:~\cite{Liu09},~\cite{LT10},~\cite{LTW10}.

    On the other hand, the second method begins with weak solutions and shows they possess $C^{1, \alpha}$ regularity, and takes a more geometric approach. This method was pioneered by Caffarelli (see~\cite{C90,C92}) in the case of the classical Monge-Amp\'ere equation. The method is adapted later to the optimal transport problem  by Figalli, Kim, and McCann first to costs satisfying~\eqref{NNCC} (\cite{FKM09}) then~\eqref{A3w}(\cite{FKM11}), with costs and domains satisfying the same assumptions as Ma, Trudinger, and Wang. This is also the method that we utilize in this paper.

    It should be noted here that the work of Loeper (\cite{L09}) serves to connect the two branches of regularity theory. Loeper showed that the condition~\eqref{A3w} introduced by Ma, Trudinger, and Wang is necessary for regularity of solutions to the optimal transport problem. Additionally, Loeper proves Theorem~\ref{thm: Loeper's lemma} which shows that condition~\eqref{A3w} implies a certain kind of convexity of certain sublevel sets of the solution to the optimal transport problem, which plays a key part in the geometric approach used by Figalli, Kim, and McCann.

    We also mention here that there are a number of examples known to satisfy~\eqref{A3w} and weakened versions,~\eqref{A3s} and~\eqref{NNCC} (see Subsection~\ref{subsection: MTW} for these conditions). The condition \eqref{A3s} is satisfied by many well known costs, including $c(x, \xbar)=\norm{x-\xbar}^p$ on $\real^n\times\real^n$, for $-2<p<1$ (see~\cite{TW09}); geodesic distance squared and the cost related to the far-field antenna problem on the sphere (\cite{L11}); and geodesic distance squared on quotients and perturbations of the sphere, which is a widely studied case (\cite{DG10,DG11,FRV12}). On the other hand, \eqref{NNCC} is satisfied on products of spheres (\cite{FKM10}). For the geodesic distance squared in more general manifolds, Loeper and Villani \cite{ LV10} showed regularity with the extra assumption that the Riemannian manifold is non-focal (thus this is a global result), further discussion on regularity of optimal transport maps and its relation with the cut locus can be found in \cite{FRV11}, see also \cite{FRV10}.

    It should also be noted that a cost function that satisfies~\eqref{A3w} but neither~\eqref{A3s} nor~\eqref{NNCC} is $c(x, \xbar)=\norm{x-\xbar}^{-2}$ on $\real^n\times\real^n$ (see the appendix in~\cite{KK12}). 

    \subsection{Outline of Caffarelli's method for the classical Monge-Amp\'ere equation}\label{subsection: Caffarelli outline}
    Here we provide a brief outline of Caffarelli's method applied to the classical Monge-Amp\'ere equation, for a thorough discussion of this part of the theory, see \cite{C92lectures}. Heuristically, this equation is degenerate elliptic for general convex $u$ but becomes uniformly elliptic when solutions are shown to be uniformly convex, thus one expects for elliptic theory to provide the regularity of a weak solution  $u$ whenever we can show it is strictly convex, understood in the sense that each supporting hyperplane to the graph touches it at a unique point. Caffarelli first observed in~\cite{C90} that if $u$ is convex and satisfies in some weak sense the inequality
    \begin{equation*}
         \alpha_1\leq \det(D^2u)\leq \alpha_2,
    \end{equation*}
    then for any affine function $l(x)$, for any $x$ we have the pointwise bound%\footnote{The second inequality is even stronger but for the sake of the current discussion we have simplified it.}
    \begin{align}
	     l(x)-u(x)  &\leq c_1(\alpha_1, \alpha_2,n)\Leb{ \{u\leq l\}}^{2/n}\frac{d(x,\Pi_+ \cup \Pi_{-})}{d(\Pi_+,\Pi_{-})}\label{eqn: euclidean aleksandrov}
	     \end{align}
	     where $\Pi_+$ and $\Pi_{-}$ are the two supporting hyperplanes to the sublevel set $ \{u\leq l\}$ normal to some fixed direction, along with the sharp growth estimate
\begin{align}
	     \sup \limits_{\{u\leq l\} }\{ l-u \} &\geq c_2(\alpha_1, \alpha_2,n)\Leb{\{u\leq l\}}^{2/n}.\label{eqn: euclidean sharp growth}
    \end{align}

    To give a brief idea of Caffarelli's proof, the above bounds show that if $l(x)$ is a supporting function to $u$ at $x_0$ and the convex sets $\{ u \leq l+r^2\}$ are comparable to a ball of radius $r$ for small $r$, then $u$ grows like a parabola away from the linear function $l(x)$, i.e. $u$ is strictly convex and $C^{1,1}$ at $x=x_0$. If one prevent the sets $\{u\leq l+r^2\}$ from having very high eccentricity, then one can show $u$ is strictly convex and differentiable, this is achieved in \cite{C90} by using the same bounds~\eqref{eqn: euclidean aleksandrov} and~\eqref{eqn: euclidean sharp growth} to show first that any interior singularity (i.e. lack of strict convexity) must propagate to the boundaries.

    The bounds~\eqref{eqn: euclidean aleksandrov} and~\eqref{eqn: euclidean sharp growth} were not originally presented as above  (see  \cite{C90}, Lemmas 1 and 2), rather, they were explicitly stated and used only under the normalization condition
    \begin{equation*}
         B_1(0) \subset \{ u\leq l \} \subset B_n(0)
    \end{equation*}
    where $n$ is the dimension. Then, by virtue of the affine invariance of the Monge-Amp\'ere equation one can obtain the strict convexity and differentiability of $u$ by constantly re-normalizing the sublevel sets to the above situation and applying the pointwise bounds~\eqref{eqn: euclidean aleksandrov} and~\eqref{eqn: euclidean sharp growth} at all scales. Of course, equation~\eqref{eqn: main classical} does not enjoy affine invariance in general, which makes this approach hard to follow in general (however, see \cite{FKM09} where this procedure is effectively applied to NNCC costs). The point of view we take here is to obtain the analogues to estimates~\eqref{eqn: euclidean aleksandrov} and~\eqref{eqn: euclidean sharp growth}  for general costs $c$ without any constraint on the eccentricity of $\{u \leq l\}$, this is where we use~\eqref{QConv} extensively.

    \subsection{The contributions of this paper}\label{subsection: differences}
    In this section, we would like to highlight the three main contributions of this paper. Perhaps the most interesting one is the realization that \eqref{QConv}, a condition representing ``quantitative'' quasiconvexity of a certain collection of functions, is a sufficient condition for the regularity of the optimal transport map. When the cost function $C^4$, the condition~\eqref{A3w} of Ma-Trudinger-Wang is known to be  necessary  for regularity, a fact discovered by Loeper in~\cite{L09}. On the other hand, our condition~\eqref{QConv} only requires the concept of $c$-segment, and the overall proof presented here requires only three derivatives of the cost. Additionally, we show that when the cost function is $C^4$,~\eqref{QConv} is equivalent to~\eqref{A3w} (via Loeper's characterization,~\eqref{eqn: gLp}, see Subsection~\ref{subsection: MTW} for more details). In particular, our approach shows that the property of a cost function being regular is preserved under $C^3$ limits. In this regard, it is worth recalling that Villani has already showed that for the geodesic distance squared, condition \eqref{A3w} is stable under Gromov-Hausdorff limits \cite{Vil08}. 

    Another one of our main contributions is a different geometric condition on the domains $\outerdom$ and $\innertarget$. In~\cite{FKM11}, Figalli, Kim, and, McCann require that the outer domain of definition, $\outerdom$, and the support of the target measure, $\innertarget$, be \emph{strongly} $c$-convex with respect to each other. We are able to show regularity under the condition that $\outerdom$ and $\innertarget$ be only $c$-convex with respect to each other, but at the additional cost of requiring that the supports of the two measures, $\innerdom$ and $\innertarget$, be compactly contained in outer domains $\outerdom$ and $\outertarget$ where the cost function is defined. Of course, we require that $\outerdom$ and $\outertarget$ must be $c$-convex with respect to each other. This restriction is reminiscent of the situation for the Euclidean cost, the boundaries of $\outerdom$ and $\outertarget$ can be thought of as ``infinity'' in the case of Euclidean cost, hence the requirement for compact containment in this case is akin to that of boundedness of the domains in the Euclidean case.
    
    Finally, we point out that our estimates Lemma~\ref{lem: sharp growth estimate} and Theorem~\ref{thm: aleksandrov} differ slightly from those in~\cite{FKM11}. Our proof of Lemma~\ref{lem: sharp growth estimate} is of a different nature. On the other hand, our version of Theorem~\ref{thm: aleksandrov} does not require the sublevel set in question to be trapped inside a small ball (see~\cite[Theorems 6.2 and 6.11]{FKM11} for comparison). A key ingredient in the second proof is the classical result in convex geometry, the Bishop-Phelps Theorem (Theorem~\ref{thm: Bishop-Phelps}).

%Discussion on the nature of~\eqref{QConv}, i.e. seems to be closed under diffeomorphism.

% Plan of the paper
    \subsection{Organization of paper}\label{subsection: organization}
    The remainder of the paper is organized as follows: in Section~\ref{section: basic definitions} we review the basic definitions pertaining to $c$-convex geometry in detail and define most of the notation that will be used later on. In Subsection~\ref{subsection: setup} we describe the four hypotheses on the cost $c$, and in Subsection~\ref{subsection: MTW} we show that the familiar condition~\eqref{A3w} implies our new condition~\eqref{QConv}. In Section~\ref{section: polar duals} we use~\eqref{QConv} to develop some more tools of $c$-convex geometry, at the end of the section we use the tools just developed to prove an analogue of the estimate~\eqref{eqn: euclidean sharp growth}. In Section \ref{section: aleksandrov estimate} we also extend Aleksandrov's theorem (the analogue of the estimate~\eqref{eqn: euclidean aleksandrov}), our proof is somewhat different from the one of a similar estimate in in~\cite{FKM11}, which allows us to prove the bound without assuming that the underlying domain has small diameter. In Section~\ref{section: localization} we use the generalized bounds to reproduce Caffarelli's localization theory for our costs, ruling out extremal points of any contact sets in the interior of $\outerdom$. Finally, in Section~\ref{section: strict convexity} we show that a solution must indeed be strictly $c$-convex, thus proving Theorem~\ref{thm: main}.

%% ELEMENTS OF C-CONVEX GEOMETRY

\bigskip

\bigskip

\section{Elements of $c$-convex geometry}\label{section: basic definitions}

%% SET UP %%
\subsection{Set up and notation}\label{subsection: setup} We describe all the objects under consideration, borrowing from the notation in~\cite{KM10}, we consider two sets $\outerdom$ and $\outertarget$ which are bounded subdomains of $n$-dimensional Riemannian manifolds $(M, \g{})$ and $(\Mbar, \gbar{})$, respectively. Elements of $\outerdom$ will be denoted by $x$ and those of $\outertarget$ by $\xbar$.

Together with these sets we have a \emph{cost function} $c(x, \xbar)$, a function $c: \outerdom \times \outertarget \to \mathbb{R}$. We will assume the cost function is $C^3$ in the sense that, any mixed derivatives which are of order $2$ in one variable and order $1$ in the other are continuous in $\outerdom^{\cl} \times \outertarget^{\cl}$. By an abuse of language, we will say the cost is $C^3$ to denote this. Note that by the canonical splitting
\begin{align*}
\cotansp{(x, \xbar)}{(M \times \Mbar)}&=\cotanspM{x} \oplus \cotanspMbar{\xbar},
\end{align*}
we can write the canonical splitting of the differential of $c$ as
\begin{equation*}
dc=Dc\oplus\Dbar c.
\end{equation*} 
Furthermore, the analogous splitting of $\tansp{(x, \xbar)}{(M \times \Mbar)}$ guarantees that the linear operator $-\Dbar D c(x, \xbar): \tanspMbar{\xbar} \to \tansp{-Dc(x, \xbar)}{\left(\cotanspM{x}\right)}\cong \cotanspM{x}$ and its adjoint $-D \Dbar c(x, \xbar)$ are unambiguously defined for each $(x, \xbar)\in\outerdom^{\cl}\times \outertarget^{\cl}$. First we define a bit of notation that will be used heavily.\\

    \begin{DEF}\label{def: p notation}
         For any pair of points $(x, \xbar)\in\outerdom\times\outertarget$ we define
         \begin{align*}
              p_{(x, \xbar)}:&=-\Dbar c(x, \xbar)\\
              \pbar_{(x, \xbar)}:&=-Dc(x, \xbar).
         \end{align*}
         Moreover, if $A \subset \outerdom$ (resp $\bar A \subset \outertarget$) and $\xbar \in \outertarget$ (resp. $x \in \outerdom$), we will write
         \begin{align*}
         	  \coord{A}{\xbar} &:= -\Dbar c(A,\xbar)\\
              (resp. \coord{\bar A}{x} & := -Dc(x,\bar A) \;).
         \end{align*}
    \end{DEF}
We will often use $p$ for elements of $\outerdomcoord{\xbar}$ and $\pbar$ for elements of $\outertargetcoord{x}$.

We now list our hypotheses on $c$, $\outerdom$, $\outertarget$, and $\innertarget$. \eqref{Twist},~\eqref{Nondeg}, and~\eqref{Convdm} are standard in the literature, but~\eqref{QConv} is introduced here for the first time. Later we explain how~\eqref{QConv} follows from condition~\eqref{A3w} of Ma, Trudinger, and Wang (see~\cite{MTW05} and~\cite{TW09}) when $c$ is $C^4$.\\

\noindent \textbf{\underline{Twist.}} The mappings
\begin{align}
\xbar & \mapsto -Dc(x_0, \xbar),\ \xbar\in\outertarget \notag\\
x & \mapsto -\Dbar c(x, \xbar_0),\ x\in\outerdom \tag{Twist}\label{Twist}
\end{align}
are injective for each $x_0 \in\outerdom$ and $\xbar_0 \in \outertarget$.

\begin{DEF}\label{def: exp^c}
For $x_0 \in \outerdom$ and $\pbar\in\outertargetcoord{x_0}$, (resp. $\xbar_0 \in \outertarget$ and $p\in\outerdomcoord{\xbar_0}$) we write $\cExp{x_0}{\pbar}\in\outertarget$ (resp. $\cExp{\xbar_0}{p}\in\outerdom$) for the unique element such that 
\begin{equation*}
-Dc(x_0, \cExp{x_0}{\pbar})=\pbar \;\;(\text{resp.} -\Dbar c(\cExp{\xbar_0}{p}, \xbar_0)=p)\;.
\end{equation*}
\end{DEF}

\noindent \textbf{\underline{Nondegeneracy.}} For any pair $(x, \xbar)\in\outerdom\times\outertarget$, the following linear mappings are invertible
\begin{align}
-\Dbar D c(x, \xbar)&: \tanspMbar{\xbar} \to \tansp{-Dc(x, \xbar)}{\left(\cotanspM{x}\right)}\cong \cotanspM{x}\notag\\
\DDbar{x}{\xbar}&: \tanspM{x} \to \tansp{-\Dbar c(x, \xbar)}{\left(\cotanspMbar{\xbar}\right)}\cong \cotanspMbar{\xbar}\tag{Nondeg}\label{Nondeg}.
\end{align}
\begin{DEF}\label{def: DDbar}
For each $(x, \xbar)\in \outerdom^{\cl}\times\outertarget^{\cl}$, for brevity of notation we will denote the linear transformation
    \begin{align*}
         \transform{x}{\xbar}:= \DbarDinv{x}{\xbar}: \cotanspM{x}\to \tanspMbar{\xbar}
    \end{align*}
    and its adjoint
        \begin{align*}
         \transformadj{x}{\xbar}:= \DDbarinv{x}{\xbar}: \cotanspMbar{\xbar}\to \tanspM{x}.
    \end{align*}
\end{DEF}
\begin{DEF}\label{def: inner product}
For any fixed $x\in \outerdom^{\cl}$, $\xbar\in \outertarget^{\cl}$ we define the inner product 
\begin{equation*}
\innerg[x, \xbar]{v}{w}:=\innergbar[\xbar]{\transform{x}{\xbar}v}{\transform{x}{\xbar}w}
\end{equation*}
 for $v$,  $w\in \cotanspM{x}$, here $\gbar{}$ is the Riemannian metric on $\Mbar$ (note that this defines an inner product by~\eqref{Nondeg}). 
\end{DEF}
\begin{rem}\label{rem: differential of c-exp}
By~\eqref{Nondeg}, the linear map $\transform{x}{\xbar}$ is an isomorphism of $\cotanspM{x}$ with $\tanspMbar{\xbar}$ for each $(x, \xbar)\in\outerdom^{\cl}\times \outertarget^{\cl}$. Additionally, for a fixed $\xbar\in \outertarget$, from the relation $-\Dbar c(\cExp{\xbar}{p}, \xbar)=p$, we deduce that at any $p\in\outerdomcoord{\xbar}$, the differential of the $c$-exponential map is given by
\begin{equation*}
D_p \cExp{\xbar}{\cdot}=\transformadj{\cExp{\xbar}{p}}{\xbar}.
\end{equation*}
In particular, since this map is bijective by~\eqref{Nondeg}, if 
\begin{equation*}
q\in \partial (u\circ \cExp{\xbar}{\cdot})(p)\subset \cotansp{p}{(\cotanspM{\xbar})}
\end{equation*}
we can see that 
\begin{equation*}
\left([D_p exp^c_{\xbar}]^{-1}\right)^*(q)=\transforminv{\cExp{\xbar}{p}}{\xbar} q\in\partial u(\cExp{\xbar}{p}),
\end{equation*}
where $L^*: W ^*\to V^*$ denotes the transpose of a linear transformation $L: V\to W$ between two vector spaces $V$ and $W$.
\end{rem}
\begin{rem}\label{rem: universal constants}
Again by~\eqref{Nondeg}, the quantities $\lVert{\transformpm{x}{\xbar}}\rVert$ and $\lvert\det{\left(\transform{x}{\xbar}\right)}\rvert$ are uniformly bounded away from $0$ and infinity on $\outerdom^{\cl}\times \outertarget^{\cl}$. Throughout the paper, whenever we refer to a constant as \emph{universal}, this will denote that the constant depends only on the domains $\outerdom$, $\outertarget$, $\innerdom$, and $\innertarget$, the dimension $n$, the constants $\alpha_1$ and $\alpha_2$ in~\eqref{eqn: main}, and the following finite, nonzero quantities: $\sup_{\outerdom^{\cl}\times \outertarget^{\cl}}{\lvert\det{\left(\transform{x}{\xbar}\right)}\rvert^{\pm 1}}$, $\sup_{\outerdom^{\cl}\times \outertarget^{\cl}}{\lVert{\transformpm{x}{\xbar}}\rVert}$, $\lVert c\rVert_{C^3(\outerdom^{\cl}\times \outertarget^{\cl})}$.
\end{rem}

The next hypothesis is a geometric condition on the domains $\outerdom$, $\innertarget$.
\begin{DEF}\label{def: c-segment}
Given points $x_0$, $x_1\in\outerdom$ and $\xbar\in\outertarget$, we define the \emph{(canonical parametrization of the) $c$-segment with respect to $\xbar$ from $x_0$ and $x_1$} by the curve
\begin{equation*}
\cExp{\xbar}{(1-t)p_{(x_0, \xbar)}+tp_{(x_1, \xbar)}},\;\;\;t\in[0, 1].
\end{equation*}
In an analogous manner we define the \emph{(canonical parametrization of the) $c$-segment with respect to $x$ from $\xbar_0$ to $\xbar_1$}, given points $\xbar_0$, $\xbar_1\in\outertarget$ and $x\in\outerdom$. We will write $[x_0, x_1]_{\xbar}$ ($[\xbar_0, \xbar_1]_{x}$) to refer to the actual images of these curves.
\end{DEF}

\begin{DEF}\label{def: c-convex sets}
Given a point $x\in\outerdom$, we will say that $\arbitrarybar\subset\outertarget$ is \emph{$c$-convex with respect to $x$} if for any two points $\xbar_0$ and $\xbar_1\in\arbitrarybar$, the $c$-segment $[\xbar_0, \xbar_1]_{x}$ is entirely contained in $\arbitrarybar$. We will say that $\arbitrarybar$ is \emph{$c$-convex with respect to $\arbitrary\subset\outerdom$} if it is $c$-convex with respect to every $x\in\arbitrary$.

We also define $\arbitrary$ being $c$-\emph{convex with respect to $\xbar\in\outertarget$} or \emph{with respect to $\arbitrarybar\subset\outertarget$} in a similar way. Finally, we define when $\arbitrary\subset\outerdom$ and $\arbitrarybar\subset\outertarget$ are \emph{$c$-convex with respect to each other} in the obvious manner.
\end{DEF}

\noindent \textbf{\underline{$c$-convexity of domains.}} The domains $\outerdom$ and $\outertarget$ are $c$-convex with respect to each other, i.e. 
\begin{equation}\label{Convdm}\tag{DomConv}
x_0,x_1 \in \outerdom \Rightarrow [x_0,x_1]_{\xbar} \subset \outerdom, \qquad\forall\; x_0,x_1 \in \outerdom,\; \xbar\in\innertarget
\end{equation}
(and vice versa), $\innertarget$ is $c$-convex with respect to $\outerdom$, and $\innertarget$ is compactly contained in $\outertarget$. By this we mean $\innertarget\subset \outertarget^{\interior}$.

Conditions~\eqref{Twist},~\eqref{Nondeg}, and~\eqref{Convdm} are well known in the literature of optimal transport (see, for example,~\cite{MTW05}). The following is an unfamiliar condition, but is actually equivalent to the key condition~\eqref{A3w} first introduced by Ma, Trudinger, and Wang when the cost function is $C^4$ (see Subsection~\ref{subsection: MTW} below). We note that our condition~\eqref{QConv} only requires the notion of $c$-exponential map (hence $C^1$ of $c$ and~\eqref{Twist}) to formulate, however the remainder of our results require the aforementioned $C^3$ regularity.
\\

\noindent \textbf{\underline{Quantitative quasiconvexity.}} There is a universal constant $M\geq 1$ such that for any points $x$, $x_0$, $x_1 \in \outerdom$ and $\xbar$, $\xbar_0$, $\xbar_1\in\outertarget$, 
\begin{align}
&-c(x,\xbar(t))+c(x,\xbar_0)-(-c(x_0,\xbar(t))+c(x_0,\xbar_0)) \notag\\
&\qquad \leq Mt\;(-c(x,\xbar_1)+c(x,\xbar_0)-(-c(x_0,\xbar_1)+c(x_0,\xbar_0)))_{_+},\qquad\forall\;t\in[0,1]	\notag\\
&-c(x(s),\xbar)+c(x_0,\xbar)-(-c(x(s),\xbar_0)+c(x_0,\xbar_0)) \notag\\
&\qquad\leq Ms\;(-c(x_1,\xbar)+c(x_0,\xbar)-(-c(x_1,\xbar_0)+c(x_0,\xbar_0)))_{_+},\qquad\forall\; s\in [0,1]\label{QConv}\tag{QQConv}
\end{align}
%where
%\begin{align*}
%f(t)  :&= -c(x_1,\xbar(t))+c(x_0,\xbar(t)),\qquad t \in[0,1]\\
%g(s)  :&=  -c(x(s),\xbar_1)+c(x(s),\xbar_0),\qquad s \in[0,1]	
%\end{align*}
where $\xbar(t)$ is the $c$-segment with respect to $x_0$ from $\xbar_0$ to $\xbar_1$, and $x(s)$ is the $c$-segment with respect to $\xbar_0$ from $x_0$ to $x_1$. 
\begin{rem}\label{rem: H4 remark}
Note that if $\mountain_t$ is any family of $c$-function whose foci are given by $\xbar(t)$, the first inequality above reads
\begin{align*}
\mountain_t(x)-\mountain_0(x)\leq Mt(\mountain_1(x)-\mountain_0(x))_{+}.
\end{align*}
Similarly, if $\mountain$ and $\mountain_0$ are arbitrary $c$-functions with $\xbar$ and $\xbar_0$ as respective foci, the second inequality can be written
\begin{align*}
 \mountain(x(s))-\mountain(x_0)-(\mountain_0(x(s))-\mountain_0(x_0))\leq Mt(\mountain(x_1)-\mountain(x_0)-(\mountain_0(x_1)-\mountain_0(x_0)))_{+}.
\end{align*}
\end{rem}
Finally, we make some other notational conventions used throughout the paper: the symbol $\inner{\cdot}{\cdot}$ will denote the evaluation map between an element of a vector space and an element of its dual space. Also, $\lvert\cdot\rvert_{\calL}$ will denote either the Riemannian volume on $(M, \g{})$ or $(\Mbar, \gbar{})$, the associated Riemannian volumes on a tangent or cotangent space, or the volume induced by the inner product $\g{x, \xbar}$ on a tangent or cotangent space of $M$ (which is comparable to the associated Riemannian volume by a factor depending only on $c$). Finally, $\gnorm[x, \xbar]{\cdot}$, $\gnorm[x]{\cdot}$, and $\gbarnorm[\xbar]{\cdot}$ will denote the length of tangent or cotangent vectors, with respect to the inner products $\g{x, \xbar}$, $\g{x}$, and $\gbar{\xbar}$. 
%% C-CONVEX FUNCTIONS %%
\subsection{$c$-convex functions} We now review the concept of $c$-convexity for functions.
\begin{DEF}\label{def: c-functions}
A \emph{$c$-function} with \emph{focus $\xbar_0 \in \outertarget$} is a scalar function on $\outerdom$ of the form
\begin{equation*}
\mountain(x):=-c(x, \xbar_0)+\lambda_0
\end{equation*}
for some $\lambda_0\in\real$. A function $u$ is \emph{$c$-convex} if for any $x_0$ there is a $c$-function $\mountain$ such that
\begin{align*}
u(x_0)&=\mountain(x_0),\\
u(x)&\geq \mountain(x),\qquad\forall\; x\in\outerdom,
\end{align*}
and we say such an $\mountain$ is \emph{supporting to $u$ at $x_0$}.

If the second inequality above is strict for all $x\neq x_0$, we say that $u$ is \emph{strictly $c$-convex at $x_0$}.
\end{DEF}
\begin{DEF}\label{def: c-subdifferentials}
If $u$ is a $c$-convex function and $x\in\outerdom$, we define the \emph{$c$-subdifferential of $u$ at $x$} as the set-valued mapping given by
\begin{equation*}
\partial_cu(x):=\{\xbar\in\outertarget\mid \text{there exists a }c\text{-function with focus }\xbar\text{ that is supporting to }u\text{ at }x\}.
\end{equation*}
Also, given a Borel measurable set $\arbitrary\subset \outerdom$, we define
\begin{equation*}
\partial_cu\left(\arbitrary\right):=\union_{x\in \arbitrary}\partial_cu(x).
\end{equation*}
Finally, we define the \emph{subdifferential of $u$ at $x$} as the set-valued mapping given by
\begin{equation*}
\partial u(x):=\{p\in \cotanspM{x} \mid u(exp_x(v))\geq u(x)+\langle p, v\rangle +o(\lvert v\rvert),\ v\in\tanspM{x}\}
\end{equation*}
where here $exp_x$ is the Riemannian exponential map.
\end{DEF}
%% ALEKSANDROV SOLUTIONS
With these definitions in hand, we may define some weak notions of solutions to the equation~\eqref{eqn: main classical}.

\begin{DEF}\label{def: Brenier solution}
A $c$-convex function $u$ is a \emph{Brenier solution to~\eqref{OT problem}}, if 
\begin{equation*}
(\partial_cu)_\#\mu=\mubar.
\end{equation*}
\end{DEF}
\begin{rem}\label{rem: brenier solutions a.e.}
Given a $c$-convex function $u$ where $c\in C^1(\outerdom^{\cl}\times\outertarget^{cl})$, it is known that $u$ must be Lipschitz, and hence differentiable $\dVol_M$ almost everywhere (see~\cite{Vil09}). Thus, since $\mu$ is absolutely continuous with respect to $\dVol_M$, it is easy to see that $\partial_cu(x)$ is single valued for $dVol_M$ almost every $x$. Hence, we may reformulate a Brenier solution as a $c$-convex function $u$ such that for any continuous function $\eta \in C(\outertarget)$,
\begin{equation*}
\int_{\outertarget}\eta(\xbar)d \mubar(\xbar) = \int_{\outerdom}\eta(\partial_cu(x))d\mu(x).
\end{equation*}
In particular, under the assumptions on $\mu$ and $\mubar$ in this paper, a Brenier solution will satisfy for any continuous function $\eta \in C(\outertarget)$,
\begin{equation*}
\alpha_1 \int_{\innerdom}\eta(\partial_cu(x))\dVol_M(x) \leq \int_{\innertarget}\eta(\xbar)\dVol_{\Mbar}( \xbar) \leq \alpha_2 \int_{\innerdom}\eta(\partial_cu(x))\dVol_M(x).
\end{equation*}
\end{rem}
As mentioned in Subsection~\ref{subsection: history}, the analysis of the optimal transport map is often through the study of a certain scalar PDE, which becomes the Monge-Amp\`ere equation when the cost is the Euclidean inner product. In this special case, Aleksandrov introduced a notion of solution using the subdifferential of a convex potential function, which in the case of a general cost function $c$ is defined as follows.

\begin{DEF}\label{def: Aleksandrov solution}
A $c$-convex function $u$ is \emph{an Aleksandrov solution of~\eqref{OT problem}} if
\begin{equation}\label{eqn: main}
\alpha_1\Leb{\arbitrary\cap \innerdom}  \leq \left | \partial_c u(\arbitrary)\right |\leq \alpha_2 \Leb{\arbitrary\cap \innerdom}	
\end{equation}
for any Borel measurable $\arbitrary\subset \outerdom$, and $\partial_cu(\outerdom)\subset\innertarget$.
\end{DEF}
By results contained in~\cite{MTW05}, if the support of the target measure $\mubar$ is $c$-convex with respect to the support of the initial measure $\mu$, the two notions of solution coincide.
\begin{thm}[{\cite[Lemma 5.1 and Section 3]{MTW05}}]\label{thm: Brenier implies Aleksandrov}
If $\innertarget$ is $c$-convex with respect to $\outerdom$, then a Brenier solution to~\eqref{OT problem} is also an Aleksandrov solution to~\eqref{OT problem}.
\end{thm}
\subsection{The Ma-Trudinger-Wang Tensor, Loeper's Theorem, and~\eqref{QConv}}\label{subsection: MTW}

Suppose that $c$ is $C^4$ in the sense that mixed derivatives of order $2$ in both variables simultaneously are continuous, and fix local coordinate systems on $M$, $\Mbar$ near $(x, \xbar)$.  Then for $V$, $W\in \tanspM{x}$ and $\eta$, $\zeta\in \cotanspM{x}$, define 
\begin{equation*}
\MTW_{(x, \xbar)}(V, W, \eta, \zeta):=\MTWcoord{i}{j}{k}{l}(x, \xbar)V^iW^j\eta_k\zeta_l.
\end{equation*}
Here regular indices denote derivatives of $c$ with respect to the first variable, while indices with a bar above denote derivatives with respect to the second derivative, and a pair of raised indices denotes the matrix inverse.

\begin{DEF}\label{def: A3w}
We will say that a cost $c$ satisfies condition~\eqref{A3w} if for all $x\in\outerdom$, $\xbar\in\outertarget$, and any $V\in \tanspM{x}$ and $\eta\in \cotanspM{x}$ such that $\inner{\eta}{V}=0$, 
\begin{equation}\label{A3w}\tag{A3w}
\MTW_{(x, \xbar)}(V, V, \eta, \eta)\geq 0.
\end{equation}
\end{DEF}
\begin{rem}\label{rem: other conditions}
There are two conditions related to~\eqref{A3w}, each slightly stronger. 
First, we say that a cost $c$ satisfies condition~\eqref{A3s} if there exists a constant $\delta_0>0$ such that, for all $x\in\outerdom$, $\xbar\in\outertarget$, and any $V\in \tanspM{x}$ and $\eta\in \cotanspM{x}$ such that $\inner{\eta}{V}=0$, 
\begin{equation}\label{A3s}\tag{A3s}
\MTW_{(x, \xbar)}(V, V, \eta, \eta)\geq \delta_0\gnorm[x]{V}^2\gnorm[x]{\eta}^2.
\end{equation}

Next, we say that a cost $c$ satisfies condition~\eqref{NNCC} if there exists a constant $\delta_0>0$ such that, for all $x\in\outerdom$, $\xbar\in\outertarget$, and any $V\in \tanspM{x}$ and $\eta\in \cotanspM{x}$, 
\begin{equation}\label{NNCC}\tag{NNCC}
\MTW_{(x, \xbar)}(V, V, \eta, \eta)\geq \delta_0\gnorm[x]{V}^2\gnorm[x]{\eta}^2.
\end{equation}
The difference with~\eqref{A3w} is the removal of the conditions that $\inner{\eta}{V}=0$. Also,~\eqref{NNCC} stands for ``non-negative cross curvature'' (see~\cite{KM10}).
\end{rem}
%
%\begin{rem}\label{rem: A3w is a tensor}
%It is easy to check that $\MTW$ is a $(2,2)$-tensor, a fact we will use freely. Also, when obvious from context, we will suppress the points $(x, \xbar)$ from the notation of $\MTW$.
%\end{rem}
%\begin{rem}\label{rem: abusive notation}
%If $V\in \tanspM{x}$ and $\eta\in \tanspMbar{\xbar}$, by an abuse of notation we will often write
%\begin{equation*}
%\MTW(V, V, \eta, \eta):=\MTW(V, V, \DbarD{x}{\xbar}\eta, \DbarD{x}{\xbar}\eta).
%\end{equation*}
%\end{rem}
This condition implies an important geometric condition first discovered by Loeper (see~\cite[Theorem 3.2]{L09} and~\cite[Theorem 4.10]{KM10}). Within this paper, we refer to this property as~\eqref{eqn: gLp}, the ``geometric Loeper property.'' We note that this property is also known in the literature as ``DASM'' (double above sliding mountain) or the ``Loeper's Maximum Principle.''
\begin{thm}[Loeper's Theorem]\label{thm: Loeper's lemma}
Let $c$, $\outerdom$, and $\outertarget$ satisfy~\eqref{Twist},~\eqref{Nondeg},~\eqref{Convdm}, and~\eqref{A3w}. Fix any points $\xbar_0$, $\xbar_1\in\outertarget$ and $x_0\in\outerdom$. Then, if $\xbar(t)$ is the $c$-segment with respect to $x_0$ from $\xbar_0$ to $\xbar_1$, then for all $t\in[0, 1]$ and $x\in\outerdom$ we have
\begin{equation}\tag{gLp}\label{eqn: gLp}
-c(x, \xbar(t))+c(x_0, \xbar(t))\leq \max\left\{-c(x, \xbar_0)+c(x_0, \xbar_0),\;-c(x, \xbar_1)+c(x_0, \xbar_1)\right\}.
\end{equation}
An analogous statement holds upon reversing the roles of the domains $\outerdom$ and $\outertarget$. Whenever a cost function $c$ satisfies the conclusion of Theorem~\ref{thm: Loeper's lemma} above, we will say by an abuse of language that \emph{$c$ satisfies~\eqref{eqn: gLp}}.
\end{thm}
%\begin{rem} If we take~\eqref{eqn: gLp} and multiply both sides by $(-1)$, we obtain
%\begin{equation*}
%\min\left\{-c(\xhat, \xbar_0)+c(x, \xbar_0),\;-c(\xhat, \xbar_1)+c(x, \xbar_1)\right\}\leq -c(\xhat, \xbar(s))+c(x, \xbar(s)).
%\end{equation*}
%By exchanging the roles of $\hat x$ and $x$ above, we obtain 
%\begin{equation}\label{eqn: reverse gLp}
%\min\left\{-c(x, \xbar_0)+c(\xhat, \xbar_0),\;-c(x, \xbar_1)+c(\xhat, \xbar_1)\right\}\leq -c(x, \xbar(s))+c(\xhat, \xbar(s)).
%\end{equation}
%In other words,~\eqref{eqn: gLp} also gives a lower bound for the difference $-c(x, \xbar(s))+c(\xhat, \xbar(s))$, noting that to get this the $c$-segment $\xbar(s)$ must be %with respect to $x$ instead of $\xhat$,
%\end{rem}
We point out here that if $c$ satisfies condition~\eqref{QConv}, it also satisfies~\eqref{eqn: gLp}.
\begin{lem}\label{lem: H4 implies gLp}
Let $c$, $\outerdom$, and $\outertarget$ satisfy~\eqref{Twist},~\eqref{Nondeg},~\eqref{Convdm}, and~\eqref{QConv}. Then $c$ satisfies~\eqref{eqn: gLp}.
\end{lem}
\begin{proof}
Fix $\xbar_0$, $\xbar_1\in\innertarget$, and $\xhat\in\outerdom$. Now fix any $x\in \outerdom$. If 
\begin{equation*}
-c(x, \xbar_0)+c(\xhat, \xbar_0)\geq-c(x, \xbar_1)+c(\xhat, \xbar_1)
\end{equation*}
we choose $\xbar(t)$ to be the $c$-segment with respect to $\xhat$ from $\xbar_0$ to $\xbar_1$. Then by~\eqref{QConv}, we see that
\begin{align*}
-c(x, \xbar(t))+c(\xhat, \xbar(t))&\leq -c(x, \xbar_0)+c(\xhat, \xbar_0)+Ms[-c(x, \xbar_1)+c(\xhat, \xbar_1)-(-c(x, \xbar_0)+c(\xhat, \xbar_0))]_+\\
&=-c(x, \xbar_0)+c(\xhat, \xbar_0)\\
&=\max\left\{-c(x, \xbar_0)+c(\xhat, \xbar_0),\;-c(x, \xbar_1)+c(\xhat, \xbar_1)\right\}
\end{align*}
for any $t\in [0, 1]$. On the other hand, if 
\begin{equation*}
-c(x, \xbar_0)+c(\xhat, \xbar_0)\leq-c(x, \xbar_1)+c(\xhat, \xbar_1)
\end{equation*}
we instead take $\xbar(t)$ to be the $c$-segment with respect to $\xhat$ from $\xbar_1$ to $\xbar_0$, and apply the same reasoning to obtain the desired inequality.

The analogous statement holds with the roles of the domains $\outerdom$ and $\innertarget$ reversed.
\end{proof}

We now recall three key geometric properties of $c$-convex functions that hold when $\outerdom$ and $\outertarget$ satisfy~\eqref{Convdm}, and the cost $c$ satisfies~\eqref{eqn: gLp} (hence in particular, if it satisfies~\eqref{QConv}).
\begin{cor}[{\cite[Theorem 3.1 and Proposition 4.4]{L09}}]\label{cor: local to global}
If $c$, $\outerdom$, and $\outertarget$ satisfy~\eqref{Twist}, \eqref{Nondeg}, \eqref{Convdm}, and~\eqref{eqn: gLp}, then for any $c$-convex function $u$ in $\outerdom$ and $x\in \outerdom$
\begin{equation*}
\partial u(x)=\coord{\partial_cu(x)}{x}.
\end{equation*}
Also suppose $u$ is a $c$-convex function. Then, for any $c$-function $\mountain$, a local minimum of the difference $u-\mountain$ is a global minimum. 
\end{cor}
\begin{cor}\label{cor: c-convexity of sublevel sets}
If $c$, $\outerdom$, and $\outertarget$ satisfy~\eqref{Twist},~\eqref{Nondeg},~\eqref{Convdm}, and~\eqref{eqn: gLp}, and $u$ is a $c$-convex function and $\mountain_0$ is a $c$-function with focus $\xbar_0$, then
\begin{equation*}
\{x\in\outerdom \mid u(x)\leq \mountain_0(x)\}
\end{equation*}
is $c$-convex with respect to $\xbar_0$.
\end{cor}

    \begin{proof}
         First suppose $u=\mountain$, a $c$-function with focus $\xbar$, and $x_0, x_1 \in \{ x \in \outerdom \mid\mountain(x)\leq\mountain_0(x)\}$. Since $\outerdom$ and $\outertarget$ are $c$-convex with respect to each other, the $c$-segment $[x_0, x_1]_{\xbar_0}$ is entirely contained in $\outerdom$. Also, since $c$ satisfies~\eqref{eqn: gLp} we see that for any $s\in [0, 1]$
         \begin{equation*}
              \mountain(x(s))-\mountain_0(x(s))\leq \max\{\mountain(x_0)-\mountain_0(x_0),\ \mountain(x_1)-\mountain_0(x_1) \}\leq 0,
         \end{equation*}
         where $x(s)=[x_0,x_1]_{\xbar_0}$. Therefore $[x_0, x_1]_{\xbar_0}\subset \{ x\in \outerdom\mid \mountain(x)\leq \mountain_0(x)\}$, and this set is $c$-convex with respect to $\xbar_0$.
         Now, for a general $c$-convex function $u$ we have
         \begin{equation*}
              \left \{ u\leq \mountain_0 \right \} = \bigcap \limits_{\left \{ \mountain\mid\mountain \leq u \right \}} \left \{ \mountain \leq \mountain_0\right \},
         \end{equation*}
         the intersection being over $c$-functions $\mountain$ that lie below $u$ in $\outerdom$. All sets on the right side are $c$-convex with respect to $\xbar_0$, so the same must be true of their intersection and we are done.
    \end{proof}
%For any $c$-convex function $u$, we will refer to any sublevel set of the form $\{u\leq \mountain\}$ for any $c$-function $\mountain$ as \emph{a section of $u$}.
\begin{rem}\label{rem: quasiconvexity}
Suppose $\xbar(t)$ is a $c$-segment in $\innertarget$ with respect to some $x_0\in\outerdom$ from $\xbar_0$ to $\xbar_1$. For any $x\in\outerdom$, by~\eqref{eqn: gLp}, if $-c(x,\xbar(t))+c(x,\xbar_0)>-c(x_0,\xbar(t))+c(x_0,\xbar_0)$ for \emph{any} $t\in[0, 1]$ then we must have $-c(x,\xbar_1)+c(x,\xbar_0)>-c(x_0,\xbar(t))+c(x_0,\xbar_0)$.

A similar remark holds for the function $-c(x(s),\xbar)+c(x(s),\xbar_0)$ where $x(s)$ is a $c$-segment in $\outerdom$ with respect to $\xbar_0\in\innertarget$ between two points.
%If~\eqref{eqn: gLp} is satisfied, we can see that any function of the form of $f(t)$ or $g(s)$ as in condition~\eqref{QConv} is a quasiconvex function. In particular, if $f$ or $g$ is strictly increasing on any subinterval of $[0, 1]$, it must be non-decreasing after this subinterval.
\end{rem}

%%%%%%%%%%%%%%%%%%%%%%%%%%%%%%%%%%%%%%%%%%%%%%%%
%%%%%%%%%%%% CATACLYSM LEMMA  %%%%%%%%%%%%%%%%%%
%%%%%%%%%%%%%%%%%%%%%%%%%%%%%%%%%%%%%%%%%%%%%%%%

    We end this section showing that if $c$ is $C^4$, the familiar condition~\eqref{A3w} implies the new condition~\eqref{QConv}.
    
    At this point, we would like to mention that this proof is motivated by the key inequality~\eqref{eqn: second derivative lower bound}, which was in turn inspired by calculations in the proof of~\cite[Proposition 4.6]{KM10} by Kim and McCann. Later, we were informed by Alessio Figalli that this result is known in the literature, and would like to thank him for pointing out the reference in~\cite{Vil09}.

    \begin{lem}\label{lem: gLp redux}
         Suppose $c$ is $C^4$ in the sense mentioned above. If $c$, $\outerdom$, and $\outertarget$ satisfy \eqref{Twist}, \eqref{Nondeg} and~\eqref{Convdm}, then the cost satisfies~\eqref{A3w} if and only if it satisfies~\eqref{QConv}.
    \end{lem}
    \begin{proof}
     If $c$ satisfies~\eqref{QConv}, by Lemma~\ref{lem: H4 implies gLp} above it must satisfy~\eqref{eqn: gLp}. In particular, if the cost function is $C^4$ the cost must also satisfy~\eqref{A3w} by~\cite{L09}.
     
        Now suppose that $c$ satisfies~\eqref{A3w}. We will prove the first inequality in~\eqref{QConv}, by a symmetric argument the second inequality will also follow.
        
         Let $f(t):=-c(x,\xbar(t))+c(x_0,\xbar(t))$ where $\xbar(t)$ is the $c$-segment with respect to $x_0$ from $\xbar_0$ to $\xbar_1$. It will then be sufficient to show that
         \begin{align*}
 	f(t)-f(0)\leq Mt(f(1)-f(0))_{+},\qquad\forall\;t\in[0, 1]
\end{align*}
for some universal $M>1$.

         Since $c$ satisfies~\eqref{A3w}, by Theorem~\ref{thm: Loeper's lemma} it satisfies~\eqref{eqn: gLp} and hence
         \begin{align*}
              f(t)-f(0)\leq \max{\{f(1)-f(0), f(0)-f(0)\}}.
         \end{align*}
         In particular, if $f(1)\leq f(0)$ we immediately obtain the corollary. 

         Now suppose that $f(1)>f(0)$. First we claim that if $f'(t_0)>0$ for some $t_0\in [0, 1)$, then $f'(t_1)>0$ for all $t_1>t_0$ as well. Indeed, suppose by contradiction that for some $t_1>t_0$, $f'(t_1)=0$. Define the real valued function $F$ on $\outertargetcoord{x_0}$ by
         \begin{equation*}
              F(\pbar):=-c(x, exp_{x_0}^c(\pbar))+c(x_0, exp_{x_0}^c(\pbar)).
         \end{equation*}
         Then, for the line segment 
         \begin{equation*}
              \pbar(t):=(1-t)\pbar_{(x_0, \xbar_0)}+t\pbar_{(x_0, \xbar_1)}
         \end{equation*}
         we have $f(t)=F(\pbar(t))$ and hence
         \begin{equation*}
              \inner{ D F(\pbar(t_1))}{\pbardot(t_1)}=f'(t_1)=0.
         \end{equation*}
         By reversing the roles of $\outerdom$ and $\outertarget$ in Corollary~\ref{cor: c-convexity of sublevel sets}, the sublevel set 
         \begin{align*}
         \{\pbar\in\outerdomcoord{x_0}\mid F(\pbar)\leq F(\pbar(t_1))\}
         \end{align*} is convex, hence we find that the entire line segment $\{\pbar(t)\mid t\in[0, 1]\}$ is contained in the supporting hyperplane to this sublevel set. In particular $\pbar(t_0)$ is either in the boundary of the sublevel set or in the complement, i.e. $F(\pbar(t_0))\geq F(\pbar(t_1))$. However, since $f'(t_0)>0$, by Remark~\ref{rem: quasiconvexity} above we see that $f$ is increasing on the interval $(t_0, 1)$ and in particular, $F(\pbar(t_1))>F(\pbar(t_0))$, which is a contradiction, thus proving the claim.

         Now, by~\cite[Proof of Theorem 12.36]{Vil09}, there exists a universal $C\geq 0$ such that 
         \begin{equation}\label{eqn: second derivative lower bound}
             f''(t)\geq -C\lvert f'(t)\rvert, \qquad\forall\;t\in[0,1].	
         \end{equation}
         First, suppose that $f'(0)> 0$. Then, by the above claim, $f'(t)>0$ for all $t\in[0, 1]$, thus
         \begin{equation*}
              f''(t)\geq -C\lvert f'(t)\rvert=-Cf'(t), \qquad\forall\;t\in[0,1]
         \end{equation*}
         and we may integrate this inequality to see that 
         \begin{equation}\label{eqn: fderivative inequality}
              f'(t_2)/f'(t_1)\geq e^{-C(t_2-t_1)}	\qquad\text{ whenever } 0\leq t_1\leq t_2\leq 1.
         \end{equation}
         Now, fix $t \in (0,1)$ and define the function $\tilde f(\ttil):= f( t\ttil)$ for $\ttil\in[0, 1]$. Since $f(1)>f(0)$ and $f'>0$ on $[0, 1]$, by Cauchy's mean value theorem and~\eqref{eqn: fderivative inequality}, there exists some $\theta \in (0,1)$ such that
         \begin{equation*}
              \frac{\tilde f(1)-\tilde f(0)}{f(1)-f(0)}=\frac{\tilde f'(\theta_0)}{f'(\theta_0)}= \frac{t f'(t\theta_0)}{f'(\theta_0)}\leq te^{C(1-t)\theta_0}.
         \end{equation*}
         Thus we conclude that
         \begin{equation*}
              f(t)-f(0)\leq e^{C(1-t)\theta_0}t (f(1)-f(0))\leq e^Ct(f(1)-f(0))=Mt(f(1)-f(0)),\qquad\forall \;t\in[0,1],
         \end{equation*}
         with $M:=e^C\geq 1$. In the general case, since $f(1)>f(0)$ by assumption, there must be some $t_0\in [0, 1)$ such that $f'(t_0)\geq 0$ and $f(t_0)=f(0)$, suppose $t_0$ is the maximal such point. Now, if $0\leq t\leq t_0$, again by~\eqref{eqn: gLp} we see that $f(t)\leq \max{(f(0), f(t_0))}=f(0)$. Hence, 
         \begin{equation*}
              f(t)-f(0)\leq 0\leq Mt(f(1)-f(0)).
         \end{equation*}
         Now suppose that $t_0<t\leq 1$. By the definition of $t_0$, we can see that $f'(t_0+\epsilon)>0$ for all $\epsilon>0$ sufficiently small. Hence for all $\epsilon>0$ such that $t_0+\epsilon<t$, we may apply the first portion of this proof to the function $f((1-t)(t_0+\epsilon)+t)$ to conclude
         \begin{align*}
         f(t)-f(0)&=f(t)-f( t_0+\epsilon)+f(t_0+\epsilon)-f(0)\\
         &\leq M\left(\frac{t-t_0-\epsilon}{1-t_0-\epsilon}\right)(f(1)-f( t_0+\epsilon))+f( t_0+\epsilon)-f(0).
         \end{align*}
         Then, if we let $\epsilon \to 0$, since $f(t_0)=f(0)$ we obtain
         \begin{align*}
              f(t)-f(0)&\leq %M\left(\frac{t-t_0}{1-t_0}\right)(f(1)-f( t_0))+f(t_0)-f(0)\\
              M\left(\frac{t-t_0}{1-t_0}\right)(f(1)-f(0))\\
              &\leq Mt(f(1)-f(0)).
         \end{align*}
    \end{proof}

%%%%%%%%%%%%%%%%%%%%%%%%%%%%%%%%%%%%%%%%%%%%%%%%%%%%%%%%
%%%%%%%%%% POLAR DUALS OF LEVEL SETS %%%%%%%%%%%%%%%%%%%
%%%%%%%%%%%%%%%%%%%%%%%%%%%%%%%%%%%%%%%%%%%%%%%%%%%%%%%%
\section{Polar duals of sublevel sets}\label{section: polar duals}
    For the remainder of the paper, we assume that the domains $\innerdom$, $\outerdom$, $\innertarget$, and $\outertarget$, and the cost function $c$ satisfy assumptions~\eqref{Twist},~\eqref{Nondeg},~\eqref{Convdm}, and~\eqref{QConv}.

    In this section, we develop some notions inspired by convex analysis, our main goal is the estimate in Lemma~\ref{lem: sharp growth estimate}, which is the analogue of the estimate~\eqref{eqn: euclidean sharp growth} in the Euclidean case.

    \begin{DEF}\label{def: c-polar dual}
         Given $\arbitrary \subset \outerdom$, $x \in A$, a $c$-function $\mountain$ and $\lambda>0$ we introduce the set 
         \begin{equation*}
              \arbitrary^c_{x,\mountain,\lambda} := \{ \xbar \in \outertarget \mid -c(\xhat,\xbar)+c(x,\xbar)-(\mountain(\xhat)-\mountain(x))\leq \lambda, \;\;\forall\;\xhat \in\partial \arbitrary\},
         \end{equation*}
         which we will refer to as the \emph{$c$-polar dual of $\arbitrary$ associated to $\mountain$, with center $x$,  and length $\lambda$}.
    \end{DEF}
    Let us point out that Definition~\ref{def: c-polar dual} was inspired by the following basic notion in convex geometry.
    \begin{DEF}\label{def: *-polar dual}
         Suppose $\mathcal{\arbitrary}\subset V$ where $V$ is a finite dimensional real vector space and let $V^*$ be its dual. Then for $p_0\in \mathcal{\arbitrary}$, $q_0\in V^*$, and $\lambda\in \real$, the \emph{polar dual of $\mathcal{\arbitrary}$ with respect to $q_0$, with center $p_0$ and height $\lambda$} is defined by
         \begin{equation*}
              \mathcal{\arbitrary}^*_{p_0,q_0,\lambda} := \{ \xbar \in V^* \mid \inner{ q-q_0}{p-p_0} \leq \lambda, \;\;\forall\; p \in \mathcal{\arbitrary}\}.	
         \end{equation*}
    \end{DEF}
    \begin{DEF}\label{def: c-cones}
    Suppose $\mountain$ is a $c$-function, $\arbitrary\subset \outerdom$, $x \in \arbitrary^{\interior}$, and $\lambda>0$. Then we define the \emph{$c$-cone associated to $\mountain$, with vertex at $x$, base $\arbitrary$, and height $\lambda$} as the function
         \begin{equation*}
              K_{x, \mountain, \arbitrary, \lambda}:=\sup \limits_{\mountainhat}\mountainhat,
         \end{equation*}
         where the supremum is taken over all $c$-functions $\mountainhat$ such that 
         \begin{align*}
              \mountainhat(\xhat) &\leq \mountain(\xhat),\;\;\; \forall\; \xhat \in \partial \arbitrary,\\
              \mountainhat(x) &\leq \mountain(x)-\lambda.
         \end{align*}
%         As a particular case, if we are considering a $c$-convex function $u$ and take $\arbitrary = \sublevelset_0:=\{u\leq \mountain_0\}$, and $\lambda = \mountain_0(x_0)-u(x_0)>0$ for some $x_0\in \sublevelset_0$, we will write (whenever $u$ is clear from context)
%         \begin{equation*}
%              K_{x_0, \mountain_0, \arbitrary, \lambda}:=K_{x_0, \sublevelset_0},
%         \end{equation*}
%         and call this the \emph{$c$-cone associated to the sublevel set $\sublevelset_0$ with vertex $x_0$}.
    \end{DEF}
    Note here that a $c$-cone is clearly a $c$-convex function.

    There are some useful relations between $c$-subdifferentials of $c$-cones and $c$-polar duals.
    \begin{lem}\label{lem: map of cones}
         Suppose $\mountain$ is a $c$-function with focus $\xbar\in\outertarget$, and $\arbitrary\subset \outerdom$ is $c$-convex with respect to $\xbar$. If $x_0 \in \arbitrary^{\interior}$ and $\lambda>0$ then
         \begin{equation}\label{eqn: polar duals and cone subdifferential equality}
              \arbitrary^c_{x_0, \mountain, \lambda}=\left(\partial_c K_{x_0, \mountain, \arbitrary, \lambda}\right) (x_0).
         \end{equation}
         Moreover, if $u$ is $c$-convex and $\sublevelset:=\{ u \leq \mountain \}$ is compactly contained in $\outerdom$ with $u(x_0)<\mountain(x_0)$, then
         \begin{equation}\label{eqn: cone subdifferential contained in u subdifferential}
              \sublevelset^c_{x_0, \mountain, \mountain(x_0)-u(x_0)}\subset (\partial_c u)(\sublevelset).
         \end{equation}
    \end{lem}

    \begin{proof}
         Let us write 
         \begin{equation*}
              K_1(x):=K_{x_0, \mountain, \arbitrary, \lambda}(x).
         \end{equation*}
         By definition, if $\xbar_1\in \arbitrary^c_{x_0, \mountain, \lambda}$ we have
         \begin{align*}
              -c(x, \xbar_1)+c(x_0, \xbar_1)+\mountain(x_0)-\lambda\leq \mountain(x)
         \end{align*}
         for all $x\in \partial \arbitrary$, hence the $c$-function $-c(x, \xbar_1)+c(x_0, \xbar_1)+\mountain(x_0)-\lambda$ is admissible in the supremum in the definition of $K_1$. Moreover since $K_1(x_0)=\mountain(x_0)-\lambda$, this $c$-function is supporting to $K_1$ from below at $x_0$, and $\arbitrary^c_{x_0,\mountain, \lambda}\subseteq\left(\partial_c K_1\right) (x_0)$. On the other hand, since $K_1\leq \mountain$ on $\partial \arbitrary$, the focus of any $c$-function supporting to $K_1$ from below at $x_0$ must satisfy the inequality in the definition of $\arbitrary^c_{x_0, \mountain, \lambda}$, proving~\eqref{eqn: polar duals and cone subdifferential equality}. 
         
         Now suppose $\sublevelset$ is compactly contained in $\outerdom$, in particular that $u=\mountain$ on $\partial\sublevelset$. Also write 
         \begin{equation*}
              K_2(x):=K_{x_0, \mountain, \sublevelset, \mountain(x_0)-u(x_0)}
         \end{equation*}
         and suppose that $\mountain_2$ is a $c$-function with focus $\xbar_2\in \outertarget$ supporting to $K_2$ from below at $x_0$, which is contained in $\sublevelset^{\interior}$ since $u(x_0)<\mountain(x_0)$. Then by the definition of $K_2$, 
         \begin{equation*}
              u(x_0)=\mountain(x_0)-(\mountain(x_0)-u(x_0))\leq K_2(x_0)=\mountain_2(x_0).
         \end{equation*}
         Combined with the fact that $\mountain_2\leq \mountain= u$ on $\partial\sublevelset$, this implies that the function $u-\mountain_2$ has a local minimum somewhere in $\sublevelset$, thus by Corollary~\ref{cor: local to global}, this implies that $\xbar_2\in \partial_cu(\sublevelset)$, proving~\eqref{eqn: cone subdifferential contained in u subdifferential}.
    \end{proof}
%%%%%%%%%%%%%%%%%%%%%%%%%%%%%%%%%%%%%%%%%%%%%%%%%%%%%%%%%%%%%%%%%%%
%%%%%%%%%%%% BOUNDS ON POLAR DUALS: EUCLIDEAN CASE %%%%%%%%%%%%%%%%
%%%%%%%%%%%%%%%%%%%%%%%%%%%%%%%%%%%%%%%%%%%%%%%%%%%%%%%%%%%%%%%%%%%

    We also recall the classical John-Cordoba-Gallegos Lemma, which is used extensively in convex analysis and the optimal transport literature,  a proof can be found in \cite{dG76}.

    \begin{lem}\label{lem: john}
         Let $V$ be an Euclidean space with $\dim(V)=n$ and $\mathcal{\arbitrary} \subset V$  a bounded convex subset with nonempty interior. There is an ellipsoid $\mathcal{\ellipsoid}$ centered at a point $p_{cm}\in V$ such that
         \begin{equation*}
              \mathcal{\ellipsoid}\subset \mathcal{\arbitrary} \subset p_{cm}+n\left(\mathcal{\ellipsoid}-p_{cm}\right),
         \end{equation*}
         it will be called the \emph{John ellipsoid of $\mathcal{\arbitrary}$} and $p_{cm}$ its  \emph{center of mass}. It is is unique in the sense that it is the ellipsoid of smallest volume such that $p_{cm}+n\left(\mathcal{E}-p_{cm}\right)$ contains $\mathcal{\arbitrary}$.
    \end{lem}

    To state the main result of this section we need one more concept, the ``dilation'' with respect to a point $\xbar\in\outertarget$ of a set. 

    \begin{DEF}\label{def: dilations}
%For a general convex subset $\mathcal{\arbitrary}$ with positive measure of a vector space $V$, its \emph{dilation by $\kappa>0$} will refer to the dilation with respect to the center of mass $p_{cm}$ of $\mathcal{\arbitrary}$, i.e.
%\begin{equation*}
%\kappa\mathcal{\arbitrary}:=\{p_{cm}+\kappa p\mid p\in\mathcal{\arbitrary}-p_{cm}\}.
%\end{equation*}
%
         Suppose that $\arbitrary\subset\outerdom$ is a set of positive measure that is $c$-convex with respect to some $\xbar\in\outertarget$. Then, its \emph{dilation by $\kappa>0$ with respect to $\xbar$} is defined as the following:
         \begin{equation*}
              \left(\kappa \arbitrary\right)^{\xbar}:=\cExp{\xbar}{p_{cm}+\kappa\left(\coord{\arbitrary}{\xbar}-p_{cm}\right)},
         \end{equation*}
         where $p_{cm}$ is the center of mass of  $\coord{\arbitrary}{\xbar}$. %is a convex set of positive measure, and $\kappa\coord{\arbitrary}{\xbar}$ is its dilation as defined above.
    \end{DEF}
The main focus of this section is the following estimate, which is the analogue of the lower bound~\eqref{eqn: euclidean sharp growth}. We stress here that this estimate does not require the sublevel set $\{u\leq \mountain\}$ to be compactly contained in $\outerdom$.
	\begin{lem}[Sharp growth estimate]\label{lem: sharp growth estimate}
		 Let $u$ be a $c$-convex function on $\outerdom$ and let $\mountain$ be a $c$-function with focus $\xbar\in\outertarget$, such that $\sublevelset := \{ u \leq \mountain \}$ has positive measure. Then for any set $\arbitrary$ which is $c$-convex with respect to $\xbar$ and satisfying (with $M$ as in condition~\eqref{QConv})
		 \begin{equation*}
		      \left (2M \arbitrary \right)^{\xbar} \subset \sublevelset,	
		 \end{equation*}
         we have the inequality
		 \begin{equation*}
		      \sup_{\sublevelset} {(\mountain-u)^n} \geq C^{-1}\Leb{\arbitrary}\Leb{\partial_cu(\arbitrary)}.	
		 \end{equation*}
		 where $C>0$ is a universal constant.
    \end{lem}

%%%%%%%%%%%%%%%%%%%%%%%%%%%%%%%%%%%%%%%%%%
%%%%%%% CONE COMPARISON LEMMA %%%%%%%%%%%%
%%%%%%%%%%%%%%%%%%%%%%%%%%%%%%%%%%%%%%%%%%

    The proof of Lemma~\ref{lem: sharp growth estimate} will itself be divided into a number of shorter lemmas. First, we relate the image of $\arbitrary$ under the $c$-subdifferential of $u$ and the $c$-subdifferential of a $c$-cone with base $\arbitrary$. (see Figure~\ref{figure: section 3a} below).
\begin{figure}[H]
  \centering
    \includegraphics[height=.5\textwidth]{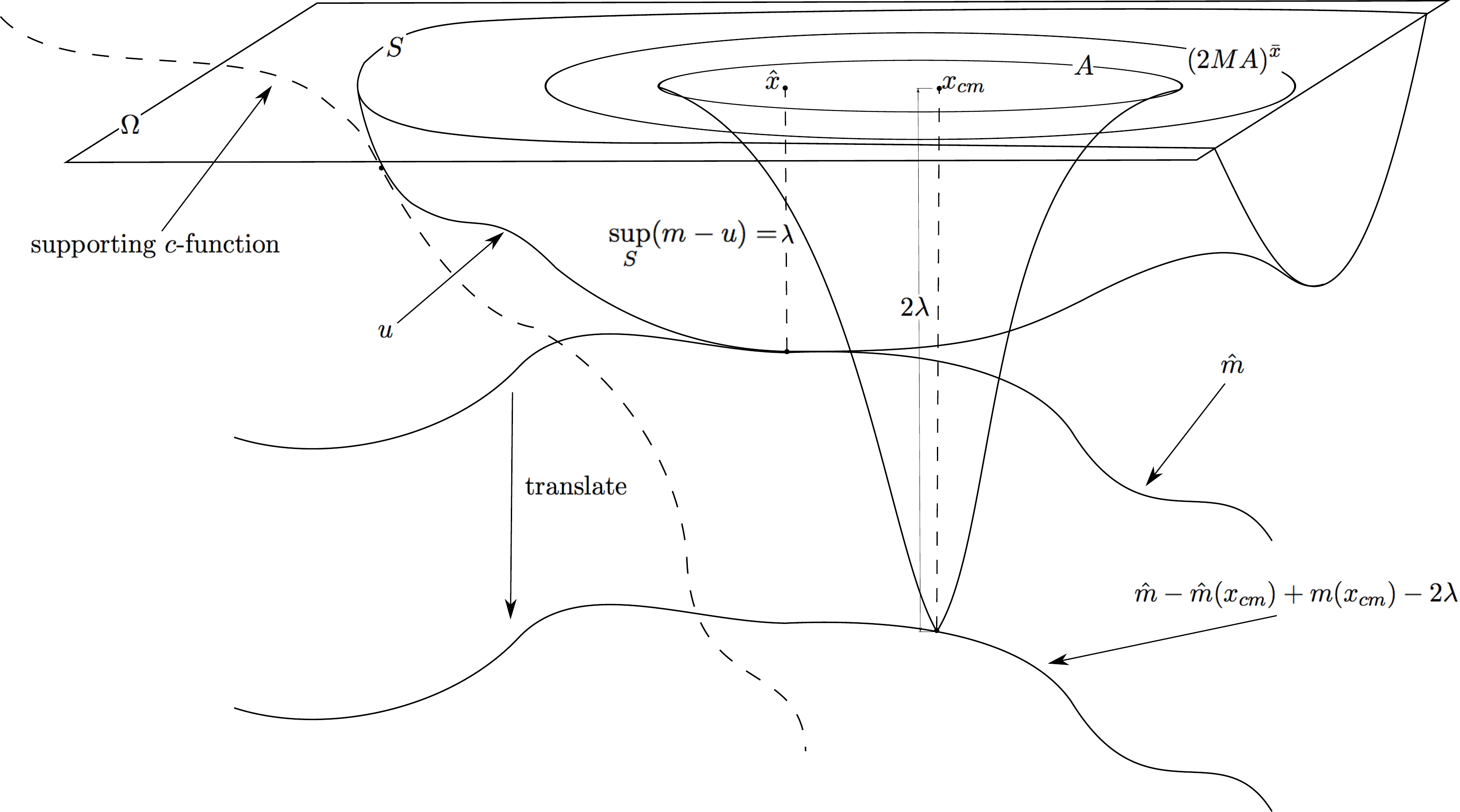}
     \caption{We can take a $c$-function supporting to $u$ somewhere in $\arbitrary$ and show that a vertical translate supports the $c$-cone at the vertex. This can be done as long as $\arbitrary$ is far enough from $\partial \sublevelset$ (note that we do not require $\sublevelset$ to stay away from $\partial \outerdom$).}\label{figure: section 3a}
\end{figure}
    \begin{lem}\label{lem: cone comparison}
         With the same notation as Lemma~\ref{lem: sharp growth estimate} above, let us write 
         \begin{equation*}
              \lambda:=\sup \limits_{\sublevelset} (\mountain-u)>0,
         \end{equation*}   
         and let $x_{cm}$ be such that $p_{(x_{cm},\xbar)}$ is the center of mass of $\arbitrary$. Then if $(2M\arbitrary)^{\xbar} \subset \sublevelset$, we have that 
         \begin{equation*}
              \partial_c u(\arbitrary) \subset \partial_c K_{x_{cm},\mountain, \arbitrary, 2\lambda}(x_{cm}).
         \end{equation*}
    \end{lem}

    \begin{proof}	
         Note that $x_{cm}\in\arbitrary^{\interior}$. Fix an arbitrary $c$-function $\mountainhat$ with focus $\xbarhat$ supporting to $u$ at $\xhat \in \arbitrary$. By definition,
         \begin{equation*}
              K_{x_{cm},\mountain,\arbitrary,2\lambda}(x_{cm})=\mountain(x_{cm})-2\lambda.
         \end{equation*}
         Then, to prove the lemma we only need to show that $\mountainhat-\mountainhat(x_{cm})+\mountain(x_{cm})-2\lambda$ is an admissible $c$-function in the definition of $K_{x_0, \mountain, \arbitrary, 2\lambda}$. In other words, we will show that
         \begin{equation*}
              \mountainhat(x)-\mountainhat(x_{cm})+\mountain(x_{cm})-2\lambda \leq \mountain(x) 
         \end{equation*}
         for all $x\in\partial\arbitrary$. 
         
         Suppose that the above inequality does not hold, then there is some $x \in\partial \arbitrary$ such that
         \begin{equation}\label{eqn: cone comparison contradiction}
              \mountainhat(x)-\mountainhat(x_{cm})+\mountain(x_{cm})-\mountain(x) >2\lambda.	
         \end{equation}

         Since $x\in \partial\arbitrary$ and $\left(2M\arbitrary\right)^{\xbar} \subset \sublevelset$ there exists some $x_1\in \partial \sublevelset$ and some $0\leq s_m \leq(2M)^{-1}$ such that $x=x(s_m)$, where $x(s)$ is  the $c$-segment with respect to $\xbar$ from $x_{cm}$ to $x_1$. Note that $M\geq 1$, hence $(2M)^{-1}< 1$. 
%          the function $g:[0,1]\to\mathbb{R}$ given by
%         \begin{align*}
%              g(s) & :=\mountainhat(x(s))-\mountain(x(s)).
%                   & = -c(x(s),\xbarhat)+c(x_{cm},\xbarhat)+\mountainhat(x_{cm})-(-c(x(s),\xbar)+c(x_{cm},\xbar)+\mountain(x_{cm})).
%         \end{align*}
         The inequality~\eqref{eqn: cone comparison contradiction} can now be written as $\mountainhat(x(s_m))-\mountain(x(s_m))-(\mountainhat(x_{cm})-\mountain(x_{cm}))> 2\lambda>0$, hence by Remark~\ref{rem: quasiconvexity} we conclude that $\mountainhat(x_1)-\mountain(x_1)>\mountainhat(x_{cm})-\mountain(x_{cm})$. Thus we may use~\eqref{QConv} and Remark~\ref{rem: H4 remark} to obtain
         \begin{align*}
              0<\mountainhat(x(s_m))-\mountainhat(x_{cm})-(\mountain(x(s_m))-\mountain(x_{cm})) & \leq Ms_m(\mountainhat(x_1)-\mountainhat(x_{cm})-(\mountain(x_1)-\mountain(x_{cm})))\notag\\
              & \leq \tfrac{1}{2} \left [\mountainhat(x_1)-\mountain(x_1)-\mountainhat(x_{cm})+\mountain(x_{cm}) \right ]\notag\\
              & \leq \tfrac{1}{2} \left [u(x_1)-\mountain(x_1)-\mountainhat(x_{cm})+\mountain(x_{cm})\right ]\notag\\
              & \leq \tfrac{1}{2}(\mountain(x_{cm})-\mountainhat(x_{cm})).
         \end{align*}
         Here we have used that $s_m\leq (2M)^{-1}$, the fact that $\mountainhat$ is supporting to $u$, and the fact that $u(x_1)\leq\mountain(x_1)$ since $x_1\in\sublevelset$. Combining this with~\eqref{eqn: cone comparison contradiction} we see that
\begin{align}\label{eqn: upper bound for g}
              0<2\lambda &<  \tfrac{1}{2}(\mountain(x_{cm})-\mountainhat(x_{cm})).
         \end{align}
         
         Next consider 
%         the function $\ghat:[0,1]\to \mathbb{R}$,
%         \begin{equation*}
%              \ghat(s):=\mountainhat(\xhat(s))-\mountain(\xhat(s)),
%         \end{equation*}
%         where 
$\xhat(s)$, the $c$-segment with respect to $\xbar_0$ from $x_{cm}$ to $\xhat_1$, where $\xhat_1 \in \partial \sublevelset$ is the unique point such that $\xhat$ lies on this segment (recall $\xhat$ is a point in $\arbitrary$ where $\mountainhat$ touches $u$ from below). As before, $\xhat \in \arbitrary$ and $\left(2M\arbitrary\right)^{\xbar} \subset \sublevelset$ implies $\xhat=\xhat( \hat s)$, for some $0\leq \hat s \leq(2M)^{-1}$. If $\mountainhat(\xhat)-\mountain(\xhat)\leq\mountainhat(x_{cm})-\mountain(x_{cm})$ we can calculate
         \begin{align*}
              \mountain(x_{cm})-\mountainhat(x_{cm})
%              &=-\ghat(0)\\
%              &\leq -\ghat(\hat s)\\
              &=\mountain(\xhat)-\mountainhat(\xhat)\\
              &=\mountain(\xhat)-u(\xhat)\\
              & \leq \lambda,
         \end{align*}
         which would contradict~\eqref{eqn: upper bound for g}. On the other hand, suppose $\mountainhat(\xhat)-\mountain(\xhat)>\mountainhat(x_{cm})-\mountain(x_{cm})$. Again, by Remark~\ref{rem: quasiconvexity} we have $\mountainhat(\xhat_1)-\mountain(\xhat_1)>\mountainhat(x_{cm})-\mountain(x_{cm})$ and we use~\eqref{QConv} and Remark~\ref{rem: H4 remark} to obtain in a similar manner to~\eqref{eqn: upper bound for g} above,
         \begin{align*}
              \mountainhat(\xhat(\hat s))-\mountainhat(x_{cm})-(\mountain(\xhat(\hat s))-\mountain(x_{cm})) & \leq M\hat s ( \mountainhat( \xhat_1)-\mountainhat(x_{cm})-(\mountain( \xhat_1)-\mountain(x_{cm})))\\
              & \leq \tfrac{1}{2} [\mountainhat( \xhat_1)-\mountain( \xhat_1)-\mountainhat(x_{cm})+\mountain(x_{cm}) ]\\
              & \leq \tfrac{1}{2} [u( \xhat_1)-\mountain( \xhat_1)-\mountainhat(x_{cm})+\mountain(x_{cm})]\\
              & \leq \tfrac{1}{2} (\mountain(x_{cm})-\mountainhat(x_{cm})).
%              & = -\tfrac{1}{2}\ghat(0),
         \end{align*}
%         and in particular, $-\tfrac{1}{2}\ghat(0)\leq-\ghat(\hat s)$. Therefore
In particular, after rearranging terms
         \begin{align*}
         \tfrac{1}{2}(\mountain(x_{cm})-\mountainhat(x_{cm}))
%          -\tfrac{1}{2}\ghat(0)\\
%              &\leq -\ghat(\hat s)\\
              &\leq\mountain(\xhat)-\mountainhat(\xhat)\\
              &=\mountain(\xhat)-u(\xhat)\\
              &\leq \lambda,
         \end{align*}
         which again contradicts~\eqref{eqn: upper bound for g}. As a result,~\eqref{eqn: cone comparison contradiction} cannot hold and the lemma is proved. 
%         On the other hand, by combining with~\eqref{eqn: cone comparison contradiction} and~\eqref{eqn: upper bound for g}, we see that
%         \begin{align*}
%              2\lambda &< g(s_m)-g(0) \leq \tfrac{1}{2}(\mountainhat(x_0)-\mountain(x_0))
%         \end{align*}
%         but we just showed that we always have $\tfrac{1}{2}(\mountain(x_{cm})-\mountainhat(x_{cm}))\leq \lambda$, and since $\lambda>0$ this gives a contradiction, and the lemma is proved.
    \end{proof}

%%%%%%%%%%%%%%%%%%%%%%%%%%%%%%%%%%%%%%%%%%
%%%%%%% AFFINE RENORMALIZATION %%%%%%%%%%%
%%%%%%%%%%%%%%%%%%%%%%%%%%%%%%%%%%%%%%%%%%
    Next, we need an estimate of the volume of the polar dual of a set (in a vector space), by the volume of the original set. This result is essentially a rougher version of the Blaschke-Santal{\'o} inequality (see~\cite[Section 7.4]{Sch93}).
    \begin{lem}\label{lem: dual bounds}
         Suppose $V$ is an $n$-dimensional Euclidean vector space and $\mathcal{\arbitrary}\subset V$ is a bounded, convex set with nonempty interior and center of mass $p_{cm}$. Also let $q_0 \in V^*$ and $\lambda>0$. Then,
	     \begin{equation*}
	          \Leb{\mathcal{\arbitrary}^*_{p_{cm},q_0,\lambda}}\leq C\Leb{\mathcal{\arbitrary}}^{-1}\lambda^n
	     \end{equation*}
	     for some universal $C>0$.
    \end{lem}

    \begin{proof}
    If we let $\mathcal{\ellipsoid}_0$ be the John's ellipsoid associated to $\mathcal{\arbitrary}$, there exists a linear transformation $T: V\to V$ such that 
    \begin{equation*}
              T(\mathcal{\ellipsoid}_0-p_{cm})=B_1(0)
         \end{equation*}
    and 
         \begin{equation}\label{eqn: determinant bound on T}
              \det{T}=\frac{\Leb{B_1(0)}}{\Leb{\mathcal{\ellipsoid}_0}}\leq \frac{C}{\Leb{\mathcal{\arbitrary}}}.
              \end{equation}
              Thus we can compute,
              \begin{align*}
              \mathcal{\arbitrary}^*_{p_{cm},q_0,\lambda } &=  \left \{ q \in V^* \mid \inner{q-q_0}{p-p_{cm}} \leq \lambda, \quad\forall\;p\in \mathcal{\arbitrary}\right \}\\
              &=\left \{ \qcheck \in V^* \mid\inner{\qcheck}{T^{-1}\pcheck} \leq \lambda, \quad\forall\;\pcheck\in T\mathcal{\arbitrary}-Tp_{cm}\right \}+q_0\\
              &=\left \{ \qcheck \in V^* \mid\inner{\lambda^{-1}(T^*)^{-1}\qcheck}{\pcheck} \leq 1, \quad\forall\;\pcheck\in T(\mathcal{\arbitrary}-p_{cm})\right\}+q_0\\
              & =\lambda T^*\left(\left(T(\mathcal{\arbitrary}-p_{cm}) \right)^*_{0,0,1}\right)+q_0\\
              &\subset \lambda T^*\left(\left(T\left(\mathcal{\ellipsoid}_0-p_{cm}\right)\right)^*_{0, 0, 1}\right)+q_0\\
              &=\lambda T^*\left(B_1(0)^*_{0, 0, 1}\right)+q_0\\
              &=\lambda T^*\left(B^*_1(0)\right)+q_0,
         \end{align*}
     where $B^*_1(0)$ is the unit ball in $V^*$.         

         Thus combining with~\eqref{eqn: determinant bound on T} we obtain
         \begin{align*}
              \Leb{\mathcal{\arbitrary}^*_{p_{cm},q_0,\lambda }}%&\leq\Leb{ \left ( p_0+\tfrac{1}{2}\left(\mathcal{\arbitrary}-p_{cm}\right) \right )^*_{p_0,q_0,\lambda } }\\
              &\leq \Leb{\lambda T^*\left(B^*_1(0)\right)+q_0}\\
              &= \lambda^n\det{(T^*)}\Leb{B_1(0)}\\
              &\leq  C\lambda^n\Leb{\mathcal{\arbitrary}}^{-1}.
         \end{align*}
    \end{proof}

    We now show how~\eqref{QConv} can be used to relate the $c$-polar dual of a set with the usual polar dual (in the cotangent coordinates) of an ellipsoid contained in the set.

    \begin{lem}\label{lem: dual upper bound} 
         Suppose $u$ is $c$-convex, $\mountain$ is a $c$-function with focus $\xbar \in\outertarget$, $\arbitrary \subset \outerdom$ is $c$-convex with respect to $\xbar$, $x\in \arbitrary$, and $u\leq \mountain$ on $\arbitrary$. Then for any $\lambda>0$ we have
         \begin{equation*}
              \coord{\arbitrary^c_{x,\mountain,\lambda}}{x} \subset \left ( \transformadj{x}{\xbar}\coord{\arbitrary}{\xbar} \right )^*_{\qbar_{(x, \xbar)},\pbar_{(x,\xbar)},M\lambda}
         \end{equation*}
         where $M$ is the constant in~\eqref{QConv} and
         \begin{equation*}
              \qbar_{(x,\xbar)} := \transformadj{x}{\xbar}p_{(x,\xbar)}.
         \end{equation*}
    \end{lem}

    \begin{proof}
         Fix some $\xbar_1\in \arbitrary^c_{x, \mountain, \lambda}$. Consider the family of $c$-functions, 
         \begin{equation*}
              \mountain_t(y):=-c(y, \xbar(t))+c(x, \xbar(t))
         \end{equation*}
         where $\xbar(t)$ is the $c$-segment with respect to $x$ going from $\xbar$ to $\xbar_1$. Since $c$ is $C^3$ on $\outerdom^{\cl}\times\outertarget^{\cl}$, for each $y\in\outerdom$, $\mountain_t(y)$ is a differentiable function of $t$ at $0$ and we may calculate,
         \begin{align*}%\label{eqn: f'(0)}
               \left.\tfrac{d}{dt}\mountain_t(y)\right\vert_{t=0}&=\inner{-\Dbar c(y, \xbar)+\Dbar c(x, \xbar)}{\xbardot(0)}\notag\\
              &=\inner{p_{(y, \xbar)}-p_{(x,\xbar)}}{\transform{x}{\xbar}(\pbar_{(x,\xbar_1)}-\pbar_{(x,\xbar)})}\notag\\
              &=\inner{\transformadj{x}{\xbar}p_{(y, \xbar)}-\qbar_{(x,\xbar)}}{\pbar_{(x,\xbar_1)}-\pbar_{(x,\xbar)}}.
         \end{align*}
         Thus to prove the claimed inclusion it suffices to show that $\left.\tfrac{d}{dt}\mountain_t(y)\right\vert_{t=0}\leq M\lambda$, for all  $y\in \arbitrary$. To this end, fix a $y\in\arbitrary$. Note that if $\left.\tfrac{d}{dt}\mountain_t(y)\right\vert_{t=0}\leq 0$ there is nothing to prove, so suppose that $\left.\tfrac{d}{dt}\mountain_t(y)\right\vert_{t=0}>0$. Then, by Remark~\ref{rem: quasiconvexity}, we have $\mountain_1(y)>\mountain_0(y)$. Hence, using~\eqref{QConv} and Remark~\ref{rem: H4 remark}, and recalling that $y \in \arbitrary$ and $\xbar_1 \in \arbitrary^c_{x, \mountain, \lambda}$, we obtain 
         \begin{align*}
              \frac{\mountain_t(y)-\mountain_0(y)}{t}&\leq M(\mountain_1(y)-\mountain_0(y))\\
              &=M[-c(y, \xbar_1)+c(x, \xbar_1)-(\mountain(y)-\mountain(x))]\\
  %            &\leq M[-c(x_1, \xbar_1)+c(x, \xbar_1)-(u(x_1)-\mountain(x))]\\
              &\leq M\lambda
         \end{align*}
         for all $t\in (0, 1)$. Thus taking $t\to 0^+$ we obtain the desired inequality.
    \end{proof}

    \begin{proof}[Proof of Lemma \ref{lem: sharp growth estimate}]
         For the sake of brevity, in what follows we will use the notation
         \begin{align*}
              p:&= p_{(x, \xbar)},\;\; \pbar:=\pbar_{(x, \xbar)},\;\; \qbar:=\transformadj{x}{\xbar}p
         \end{align*}
         and 
\begin{align*}
\lambda:=\sup_{\sublevelset}{(\mountain-u)}. 
\end{align*}
         
         First note that if $\arbitrary^{\interior}=\emptyset$, since it is $c$-convex with respect to $\xbar$ we would have $\Leb{\arbitrary}=0$ by~\eqref{Nondeg}, immediately proving the lemma.
         
         Thus we, assume that $\arbitrary$ has nonempty interior, and we may apply John's Lemma (Lemma~\ref{lem: john}) to the set. Let $x_{cm}\in\outerdom$ be such that $p_{(x_{cm}, \xbar)}$ is the center of mass of $\arbitrary$. By assumption, $(2M \arbitrary)^{\xbar} \subset \sublevelset$, in which case combining Lemma~\ref{lem: cone comparison} with~\eqref{eqn: polar duals and cone subdifferential equality} from Lemma~\ref{lem: map of cones}, we have 
         \begin{equation*}
              \partial_cu(\arbitrary) \subset \partial_cK_{x_{cm},\mountain,\arbitrary,2\lambda}(x_{cm})=\arbitrary^c_{x_{cm},\mountain, 2\lambda},
         \end{equation*}
%         Using Lemma~\ref{lem: map of cones}, we can replace the set on the right to conclude that
%         \begin{equation*}
%              \partial_cu(\arbitrary) \subset \arbitrary^c_{x_{cm},\mountain, 2\lambda}.
%         \end{equation*}
         and by combining this with Lemma~\ref{lem: dual upper bound} we obtain 
         \begin{align}
              \Leb{ \partial_cu (\arbitrary  )} & \leq C\Leb{\coord{\arbitrary^c_{x_{cm},\mountain,2\lambda}}{x_{cm}}}\notag\\
              &\leq C \Leb{\left ( \transformadj{x}{\xbar}(\coord{\arbitrary}{\xbar}) \right )^*_{\qbar,\pbar, M\lambda}}\label{eqn: lemma 3.10 inequality}.
         \end{align} 
         
         Finally, since $\transformadj{x}{\xbar}\left(\coord{\arbitrary}{\xbar}\right)$ has nonempty interior by~\eqref{Nondeg}, the point $\qbar$ is the center of mass of this set. Hence, we can apply Lemma~\ref{lem: dual bounds} to this set, with $p_{cm}=\qbar$ and $q_0=\pbar$ in the statement of the lemma, which yields (also using~\eqref{Nondeg}),
         \begin{align*}
         	   \Leb{\left ( \transformadj{x}{\xbar}(\coord{\arbitrary}{\xbar}) \right )^*_{\qbar,\pbar, M\lambda}}
              &\leq C\Leb{\transformadj{x}{\xbar}(\coord{\arbitrary}{\xbar})}^{-1}\lambda^n\\
              &\leq C\Leb{\arbitrary}^{-1}\lambda^n.
         \end{align*}
         Combining this with~\eqref{eqn: lemma 3.10 inequality} and rearranging terms, the lemma is proved.
    \end{proof}

%%%%%%%%%%%%%%%%%%%%%%%%%%%%%%%%%%%%%%%%%%%%%%%%%%%%%%%%%%%%%%%
%%%%% ALEKSANDROV ESTIMATES%%%%%%%%%%%%%%%%%%%%%%%%%%%%%%%%%%%%
%%%%%%%%%%%%%%%%%%%%%%%%%%%%%%%%%%%%%%%%%%%%%%%%%%%%%%%%%%%%%%%
\section{The Aleksandrov estimate}\label{section: aleksandrov estimate}

    In this section we extend the Aleksandrov estimate (generalizing~\eqref{eqn: euclidean aleksandrov}) to more general costs. Namely, we prove (see Definition~\ref{def: notations for aleksandrov} below where we introduce $\supsegment{\subcoord}{\e{}}$ and $\plane{\pm \w{}}{\subcoord}$)
        \begin{thm}\label{thm: aleksandrov}
         Suppose $u$ is a $c$-convex function, $\mountain$ a $c$-function with focus $\xbar\in\innertarget$, and $\sublevelset:=\{ u\leq \mountain \}$ is compactly contained in $\outerdom$. Let $x_0$ be any point in the interior of $\sublevelset$ and $\e{} \in \cotanspMbar{\xbar}$ be of unit length. Then for some universal $C>0$ we have
         \begin{equation*}
              (\mountain(x_0)-u(x_0))^n\leq \frac{C}{\supsegment{\subcoord}{\e{}}}\;\frac{d(p_{(x_0, \xbar)},\plane{\e{}}{\subcoord} \cup \plane{-\e{}}{\subcoord})}{d(\plane{\e{}}{\subcoord}, \plane{-\e{}}{\subcoord})} \; \Leb{\sublevelset}\Leb{\partial_cu(\sublevelset)}.
         \end{equation*}
    \end{thm}
    
    The proof will involve ``trapping'' the set $\sublevelset$ with sets of the form $\{\mountain_i\leq \mountain\}$, where $\mountain_i$ are conveniently chosen $c$-functions. Towards this end, we will establish a series of lemmas to construct an appropriate family of such $c$-functions. We start with some notation, which we will extensively use in this section and the next.

    \begin{DEF}\label{def: notations for aleksandrov}
         Suppose that $V$ is an inner product space, $\e{}\in V$ is nonzero, and $\mathcal{\arbitrary}\subset V$ is a convex set. Then, we define $\plane{\e{}}{\mathcal{\arbitrary}}$ to be the supporting hyperplane to $\mathcal{\arbitrary}$ where $\e{}$ is an outward pointing normal direction.

         If $\arbitrary\subset V$ is closed and convex, and $\e{}\in V$ nonzero, we define $\supsegment{A}{\e{}}$ as the maximum among all lengths of segments contained in $A$ that are parallel to $\e{}$.
\end{DEF}
\begin{DEF}\label{def: music}
         For $\e{}\in\cotanspMbar{\xbar}$ we denote its image under the \emph{musical isomorphism} by $\raisebar{\e{}}\in\tanspMbar{\xbar}$, uniquely determined by
         \begin{align*}
              \inner{\raisebar{\e{}}}{\w{}}= \gbar{\xbar}(\e{},\w{})
         \end{align*}
         for any $\w{}\in\cotanspMbar{\xbar}$. The point $\xbar$ will be understood from context.

    \end{DEF}

    \begin{rem}\label{rem: beer bet}
         Let $\mathcal{\arbitrary}$ be a convex subset of a vector space $V$ with inner product $(\cdot, \cdot)$ and $\e{}\in V$ is a unit normal vector, then the distance from a point $p_0\in \mathcal{\arbitrary}$ to $\plane{\e{}}{\mathcal{\arbitrary}}$ is given by
         \begin{align*}
              d(p_0, \plane{\e{}}{\mathcal{\arbitrary}}):=\sup_{p\in \mathcal{\arbitrary}^{\cl}}{(p-p_0, \e{})}=\sup_{p\in\partial \mathcal{\arbitrary}}{(p-p_0, \e{})}.
         \end{align*}

%         Let $p^*$ be the orthogonal projection along $e$ of the point $p_0$ onto $\plane{\e{}}{\arbitrary}$, and let $p_b\in \plane{\e{}}{\arbitrary}\cap \partial\arbitrary$. Then, for any $p\in \arbitrary^{\cl}$ since $\e{}$ is the outward pointing unit normal, we have $\inner{p-p_b}{\e{}}\leq 0$, while $\inner{p^*-p_b}{\e{}}=0$ and $d(p_0, \plane{\e{}}{\arbitrary})=\inner{p^*-p_0}{\e{}}$. Combining these,
%         \begin{align*}
%              \inner{p-p_0}{\e{}}&=\inner{p-p_b}{\e{}}+\inner{p_b-p^*}{\e{}}+\inner{p^*-p_0}{\e{}}\\
%              &\leq \inner{p_b-p^*}{\e{}}+\inner{p^*-p_0}{\e{}}\\
%              &=d(p_0, \plane{\e{}}{\arbitrary}),
%         \end{align*}
%         while 
%         \begin{align*}
%              \inner{p_b-p_0}{\e{}}&=\inner{p_b-p^*}{\e{}}+\inner{p^*-p_0}{\e{}}\\
%              &=d(p_0, \plane{\e{}}{\arbitrary}).
%         \end{align*}
   \end{rem}

    For the remainder of this section, we assume that $u$ is a $c$-convex function and $\mountain$ is a $c$-function with focus $\xbar\in \innertarget$. Defining $\sublevelset:=\{u\leq \mountain\}$, we also fix a $x_0\in \sublevelset^{\interior}$ and assume $\sublevelset$ is compactly contained in $\outerdom$ (so in particular, $u=\mountain$ on $\partial \sublevelset$). We will also write 
    \begin{align*}
    p_0:=p_{(x_0, \xbar)},\ \pbar:=\pbar_{(x_0, \xbar)}.
    \end{align*}
     Additionally, we will mainly be concerned with supporting hyperplanes to the set $\subcoord$. Therefore, when the subscript in the notation $\plane{\e{}}{\mathcal{\arbitrary}}$ is omitted, it will be understood that the set in question is $\subcoord$, the point $\xbar$ involved will be clear from context.

    In the first lemma, we relate the distance from the point $p_0$ to the supporting hyperplane $\plane{\e{}}{}$ to $\subcoord$, with the difference between the original $c$-function $\mountain$ and a certain $c$-function $\mountain^{\e{}}_t$ defined in the lemma below. Note that in this proof, we require that $c$ is $C^3$.
%
%    
%    control the ``height'' of the sublevel set $\sublevelset$ at the point $x_0$, by the difference between the original $c$-function $\mountain$ and certain $c$-functions $\mountain^i_t$ defined in the lemma below. 

    \begin{lem}\label{lem: dual segments length 2}
         For any unit vector $\e{}\in\cotanspMbar{\xbar}$ define the $c$-segment
         \begin{equation*}
              \xbar^{\e{}}(t):	= \cExp{x_0}{\pbar+t\transforminv{x_0}{\xbar}\raisebar{\e{}}}
         \end{equation*}
         and the family of $c$-functions
         \begin{equation*}
		      \mountain_t^{\e{}}(x):=-c(x,\xbar^{\e{}}(t))-\sup \limits_{y\in \sublevelset} (-c(y,\xbar^{\e{}}(t))-\mountain(y)).
		 \end{equation*}
%		 where $a_t \in \mathbb{R}$ is chosen so that $\sup \limits_\sublevelset (\mountain_t^{\e{}}-\mountain_0) = 0$. 
Then, there is a universal $d>0$ such that both of these families are well-defined for $t\in[0, d]$. Moreover, there is a universal $C>0$ such that 
         \begin{equation}\label{eqn: mountain plane distance inequality}
              C^{-1}td(p_0,\plane{\e{}}{})\leq \mountain(x_0)-\mountain^{\e{}}_{t}(x_0)\leq C t(t+d(p_0,\plane{\e{}}{})),\;\;\forall\;t\in[0,d].
         \end{equation}
    \end{lem}

    \begin{proof}
         First, by~\eqref{Nondeg} and the compactness of $\outerdom^{\cl}$ and $\outertarget^{\cl}$, the length of $\transforminv{x_0}{\xbar}\raisebar{\e{}}$ is controlled uniformly in the choice of $\e{}$, $x_0$, and $\xbar$. As a result, since $\innertarget$ is compactly contained in $\outertarget$ there exists some universal $d>0$ for which $\xbar(t)$ and hence $\mountain_t$ are well-defined on $[0, d]$.
         
         Then, if we define $\mountaintilde^{\e{}}_t(y):=\mountain^{\e{}}_t(y)-\mountain^{\e{}}_t(x_0)+\mountain(x_0)$, note that %by Definition~\ref{def: aleksandrov setup},
         \begin{align}
               \sup \limits_{y \in \sublevelset} (\mountaintilde^{\e{}}_t(y)-\mountaintilde^{\e{}}_0(y))&=\sup \limits_{y \in \sublevelset} (\mountain^{\e{}}_t(y)-\mountain^{\e{}}_t(x_0)-(\mountain^{\e{}}_0(y)-\mountain^{\e{}}_0(x_0)))\notag\\
               &=\sup \limits_{y \in \sublevelset} (-c(y,\xbar^{\e{}}(t))+c(x_0,\xbar^{\e{}}(t))-(-c(y,\xbar)+c(x_0,\xbar)))\notag\\
               &=\mountain(x_0)-\mountain^{\e{}}_t(x_0).\label{eqn: mountaintilde and sup}
         \end{align}
Fix some $y\in\sublevelset$, then for a fixed $0<t\leq d$ we may apply~\eqref{QConv} and Remark~\ref{rem: H4 remark} to $\mountaintilde^{\e{}}_{ht}(y)$, regarding $0< h\leq 1$ as the parameter, to obtain
         \begin{align*}
              \mountaintilde^{\e{}}_{ht}(y)-\mountaintilde^{\e{}}_0(y) &\leq Mh(\mountaintilde^{\e{}}_t(y)-\mountaintilde^{\e{}}_0(y)).
%              &=Mh(\mountaintilde^{\e{}}_t(y)-\mountain(y)),
         \end{align*}
         Then dividing by $h$ and taking $h\to 0^+$ results in the inequality
         \begin{align*}
              t \left.\tfrac{d}{ds}\mountaintilde^{\e{}}_s(y)\right\vert_{s=0}&=\left.\tfrac{d}{dh}\mountaintilde^{\e{}}_{ht}(y)\right\vert_{h=0}\\
              &\leq M(\mountaintilde^{\e{}}_t(y)-\mountaintilde^{\e{}}_0(y))
         \end{align*}
         for any $0<t\leq d$. On the other hand,
         \begin{align*}
              \left.\tfrac{d}{ds}\mountaintilde^{\e{}}_s(y)\right\vert_{s=0}&=\inner{-\Dbar c(y, \xbar)+\Dbar c(x_0, \xbar)}{\xbardot^{\e{}}(0)}\notag\\
              &=\inner{p_{(y, \xbar)}-p_0}{\transform{x_0}{\xbar}\left(\transforminv{x_0}{\xbar}\raisebar{\e{}}\right)}\notag\\
%              &=\sum_{j=1}^n{\innergbar[\xbar]{p_{(y, \xbar)}-p_0}{\e{j}}\inner{\e{j}}{\raisebar{\e{i}}}}\notag\\
%              &=\sum_{j=1}^n{\innergbar[\xbar]{p_{(y, \xbar)}-p_0}{\e{j}}\innergbar[\xbar]{\e{j}}{\e{i}}}\notag\\
              &=\innergbar[\xbar]{p_{(y, \xbar)}-p_0}{\e{}}.%\label{eqn: distance to plane and derivative equality}.
         \end{align*}
         By taking a supremum over $y\in\sublevelset$ and using~\eqref{eqn: mountaintilde and sup} and Remark~\ref{rem: beer bet}, we find
         \begin{align*}
              d(p_0, \plane{\e{}}{})&=\sup_{p\in \subcoord}{\innergbar[\xbar]{p-p_0}{\e{}}}\\
              &=\sup_{y\in\sublevelset}{\left.\tfrac{d}{ds}\mountaintilde^{\e{}}_s(y)\right\vert_{s=0}}\\
              &\leq Mt^{-1}\sup_{y \in \sublevelset} (\mountaintilde^{\e{}}_t(y)-\mountaintilde^{\e{}}_0(y))\\
              &=Mt^{-1}(\mountain(x_0)-\mountain^{\e{}}_t(x_0)),
         \end{align*}
         in other words, for every $t$ we have
         \begin{equation*}%\label{eqn: dual segment lower bound}
              \mountain(x_0)-\mountain^{\e{}}_t(x_0) \geq M^{-1}t d(p_0,\plane{\e{}}{}),
         \end{equation*}
         proving the lower bound in~\eqref{eqn: mountain plane distance inequality}.
         
         To prove the upper bound, note that thanks to \eqref{eqn: mountaintilde and sup} it suffices to obtain a bound
         \begin{equation*}
              \mountaintilde^{\e{}}_t(y)-\mountaintilde^{\e{}}_0(y) \leq Ct \left ( t+d(p_0,\plane{\e{} }{} )  \right ),\;\;\forall \; y\in\sublevelset,\;t\in[0, d].
         \end{equation*}
         We observe that $\mountaintilde^{\e{}}_t(y)-\mountaintilde^{\e{}}_0(y) = -c(y,\xbar^{\e{}}(t))+c(x_0,\xbar^{\e{}}(t))-\left (-c(y,\xbar)+c(x_0,\xbar) \right)$, and do a second order Taylor expansion in $t$ to get for some universal $C\geq 1$ (note here $C$ also depends on the $C^3$ norm of $c$),
         \begin{equation*}
         	  \mountaintilde^{\e{}}_t(y)-\mountaintilde^{\e{}}_0(y) \leq t \left.\tfrac{d}{ds}\right\vert_{s=0} \mountaintilde^{\e{}}_s(y)+Ct^2
         \end{equation*}
         but we already saw that $\left. \tfrac{d}{ds} \right \vert_{s=0}\mountaintilde^{\e{}}_s(y)=d(p_0, \plane{\e{}}{})$, and with this the upper bound is proved.
    \end{proof}	
     
    In what follows we will make use of the celebrated Bishop-Phelps Theorem from convex analysis, a proof of which can be found in \cite[Section 7.1]{AsiEll1980}.

    \begin{thm}[Bishop-Phelps]\label{thm: Bishop-Phelps}
    	 Let $\mathcal{\arbitrary}$ be a convex set in Euclidean space and $\e{}$  a unit vector, given $p \in \mathcal{\arbitrary}$ and $\rho \in (0,1)$ there is a unit vector $\ehat{}$ and a $p_{\e{}}\in\partial \mathcal{\arbitrary}$ such that $\plane{\ehat{}}{\mathcal{\arbitrary}}$ is supporting to $\mathcal{\arbitrary}$ at $p_{\e{}}$ and
	 %hyperplane $\Pi^*$ with outward unit normal $e^*$ such that $\Pi^*$ is supporting to $\mathcal{\arbitrary}$ at a point $x^* \in \partial \mathcal{\arbitrary}$, where $x^*$ and $e^*$ are such that
         \begin{align*}
              \norm{p-p_{\e{}}} &\leq \rho^{-1} d(p,\plane{\e[]{}}{\mathcal{\arbitrary}}),\\
              \norm{\e[]{}-\ehat{}}&\leq 2\rho.	
         \end{align*}
%         Here the normals $e,e^*$ are taken so that they both point towards the exterior of $K$.
    \end{thm}
    We now adapt the Bishop-Phelps theorem to the framework of $c$-convex geometry.
    \begin{lem}\label{lem: perturbed plane}
    	 For any  unit length $\e{}$ we can find a $x_{\e{}} \in \partial \sublevelset$ and a $c$-function $\mountainhat^{\e{}}$ with focus $\xbar_v\in\outertarget$ which is supporting to $u$ at $x_{\e{}}$ and such that 
         \begin{align}
             \gbarnorm[\xbar]{p_0-p_{(x_{\e{}},\xbar)}} &\leq \rho^{-1} d(p_0,\plane{\e{}}{}), \label{eqn: close to touching point}\\
             \gbarnorm[\xbar]{\e{}-\ehat{}}&\leq 2\rho\label{eqn: normal is close},
         \end{align}
         where $\ehat{}$ is the outer unit normal to $\coord{\{ \mountainhat^{\e{}} \leq \mountain\}}{\xbar}$ at $p_{(x_{\e{}},\xbar)}$ and  $\xbar_{\e{}}$ is given (for some $\lambda>0$) by        
         \begin{align*}
              \xbar_{\e{}}:= \cExp{x_{\e{}}}{\pbar_{(x_{\e{}}, \xbar)}+\lambda \transforminv{x_{\e{}}}{\xbar}{\raisebar{\ehat{}}}}.
         \end{align*}
         %where $\w[']{i}\in\tanspMbar{\xbar_0}$ is such that
         %\begin{align*}
              % \inner{\w[']{i}}{v}=\innergbar{\ehat{\pm i}}{v}
         %\end{align*}
         %for all $v\in\cotanspMbar{\xbar_0}$.
    \end{lem}

    \begin{proof} 
	     We can apply Theorem~\ref{thm: Bishop-Phelps} to $p_0\in \subcoord$, the convex set $\subcoord$, and the supporting plane $\plane{\e{}}{}$ to obtain $p_{{\e{}}}:=p_{(x_{\e{}},\xbar)}\in \partial\subcoord$ and a unit vector $\ehat{}\in\cotanspMbar{\xbar}$ where 
         \begin{align*}
              \gbarnorm[\xbar]{p_0-p_{{\e{}}}} &\leq \rho^{-1} d(p_0,\plane{\e{}}{}),\\
              \gbarnorm[\xbar]{\e{}-\ehat{{\e{}}}}&\leq 2\rho,	
         \end{align*}
         and $\plane{\ehat{}}{}$ is supporting to $\subcoord$ at $p_{{\e{}}}$. %, where $\Pi'_{i}$ is the hyperplane with outward unit normal $\ehat{\pm i}$. 
         
         Now if we write $U(p):=u(exp^c_{\xbar}(p))-\mountain(exp^c_{\xbar}(p))$, we see that $U$ is a semi-convex function, and $\subcoord=\{U(p)\leq 0\}$. By applying adding an appropriate quadratic function to $U$ and applying~\cite[Corollary 23.7.1]{Roc70} (this corollary may be applied since we assume $\sublevelset$ is compactly contained in $\outerdom$) we can see there exists some $\lambda>0$ such that $\lambda \raisebar{\ehat{}}\in \partial U(p_{{\e{}}})$. Here we have made the identifications
         \begin{align*}
              \cotansp{p_{{\e{}}}}{\cotanspMbar{\xbar}}\cong\cotanspMbar{\xbar}\cong\tanspMbar{\xbar}
         \end{align*}
         under the metric $\gbar{\xbar}$. From the definition of $x_{{\e{}}}$ we see that
	     \begin{align*}
	         \transform{x_{\e{}}}{\xbar}\pbar_{(x_{{\e{}}}, \xbar)}+\lambda \raisebar{\ehat{}}&=-\left(D_{p_{{\e{}}}}exp^c_{\xbar}\right)^*Dc(exp^c_{\xbar}(p_{{\e{}}}),       \xbar)+\lambda \raisebar{\ehat{}}\\
	          &= \lambda \raisebar{\ehat{}}+D_p\mountain(exp^c_{\xbar}(p_{{\e{}}}))\\
	          &\in \partial(U(\cdot)+\mountain_0\circ \cExp{\xbar}{\cdot})(p_{\e{}})\\
	          &=\partial (u\circ \cExp{\xbar}{\cdot})(p_{{\e{}}})
	     \end{align*}
	     hence by Remark~\ref{rem: differential of c-exp},
	     \begin{equation*}
	         \pbar_{(x_{{\e{}}}, \xbar)}+\lambda \transforminv{x_{{\e{}}}}{\xbar}\raisebar{\ehat{}}\in \partial u(x_{{\e{}}}).
	     \end{equation*}
	     Then by Corollary~\ref{cor: local to global} we see that
	     \begin{equation*}
	          \xbar_{{\e{}}}:=\cExp{x_{{\e{}}}}{\pbar_{(x_{{\e{}}}, \xbar)}+\lambda_{i}\transforminv{x_{{\e{}}}}{\xbar}\raisebar{\ehat{{\e{}}}}}\in \partial_c u(x_{{\e{}}}),
	     \end{equation*}
	     hence the $c$-function defined by
	     \begin{equation*}
	          \mountainhat^{{\e{}}}(x):=-c(x, \xbar_{{\e{}}})+c(x_{{\e{}}}, \xbar_{i})+u(x_{{\e{}}})
	     \end{equation*}
	     is supporting to $u$ from below at $x_{{\e{}}}$.
	         
	     Since $x_{{\e{}}}\in\outerdom^{\interior}$, differentiating we see that $\ehat{}$ is the outer unit normal to $\coord{\{ \mountainhat^{\e{}} \leq \mountain\}}{\xbar}$ at $p_{x_{{\e{}}}}$.
    	 
    \end{proof}

    The next lemma is essentially an observation in convex analysis in Euclidean space. 

    \begin{lem}\label{lem: trapping planes}
         Given any unit length $\e{} \in \cotanspMbar{\xbar}$ it can be completed to a basis $\{\e{i}\}_{i=1}^n$ where $\e{1}=\e{}$ and $\{\e{i}\}_{i=2}^{n}$ is an orthonormal set in $\gbar{\xbar}$, satisfying
         \begin{align}\label{eqn: constructed basis is nondegenerate}
              \prod_{i=1}^n{\lambda_i}&\leq C \Leb{\ch{\{0, \lambda_i\e{i}\mid 1\leq i\leq n\}}}
         \end{align}
         for any collection of numbers $\lambda_i\geq 0$ and some universal $C>0$.
         At the same time, 
%         \begin{align*}
%         \lvert \innergbar[\xbar]{\e{1}}{\e{i}}\rvert \leq 1-1/n
%         \end{align*}
%         for $i\neq 1$, and 
         \begin{equation}\label{eqn: S comparable to box}
              \Leb{\sublevelset} \geq C^{-1}\supsegment{\subcoord}{\e{}} \prod \limits_{i=1}^n d(\plane{\e{i}}{},\plane{-\e{i}}{})	
         \end{equation}
          (here $\supsegment{\subcoord}{\e{}}$ is as in Definition~\ref{def: notations for aleksandrov}).
          % and $\{\Pi_{\pm i}\}^n_{i=1}$ are determined by $\{\e{i}\}_{i=1}^{n}$ following Definition \ref{def: aleksandrov setup}.
    \end{lem}
	
    \begin{proof}
         Again, $\subcoord \subset \cotanspMbar{\xbar} $ is convex, bounded, and has nonempty interior by~\eqref{Twist} and~\eqref{Nondeg}, thus it has an associated John ellipsoid $\mathcal{\ellipsoid}$
%         , i.e.
%         \begin{align*}
%              \mathcal{\ellipsoid}\subset\subcoord\subset p_{cm}+n\left(\mathcal{\ellipsoid}-p_{cm}\right)
%         \end{align*}
         with center of mass $p_{cm}$. Then there is a basis $\{\w{i}\}_{i=1}^n$ orthonormal with respect to $\gbar{\xbar}$ with vectors parallel to the principal axes of $\mathcal{\ellipsoid}$, and by \eqref{Nondeg} we see 
         \begin{equation}\label{eqn: volume inequality with parallelpiped}
              \Leb{\sublevelset}\geq C^{-1}\Leb{\subcoord}\geq C^{-1}\Leb{\mathcal{\ellipsoid}}\geq C^{-1}\prod \limits_{i=1}^n d(\plane{\w{i}}{},\plane{-\w{i}}{}).
         \end{equation}
%         here $\Pi^*_{\pm i}$ are the supporting hyperplanes to $\subcoord$ with outward normal vectors $\pm \e[*]{i}$. 
         for some universal $C>0$. By renumbering the basis and replacing $\w{i}$ by $-\w{i}$ if necessary, we may assume that each $\innergbar[\xbar]{\e{}}{\w{i}}\geq 0$ and
         \begin{align}\label{eqn: maximal inner product}
              \innergbar[\xbar]{\e{}}{\w{1}} =\sup_{1\leq i\leq n} \innergbar[\xbar]{\e{}}{\w{i}} \geq \frac{1}{n}.
         \end{align}
          We then define another basis by $\e{1}:=\e{}$, and $\e{i}:=\w{i}$ for $i\neq 1$. 
         
         We will now show this basis has the claimed properties. First, since $\e{}$ has unit length and $\{\w{i}\}$ is orthonormal,~\eqref{eqn: constructed basis is nondegenerate} immediately follows from~\eqref{eqn: maximal inner product}.
%Next, note that
%         \begin{equation*}
%         	  1 = \sum \limits_{i=1}^n \innergbar[\xbar]{\e{}}{\w{i}}^2 \leq n  \innergbar[\xbar]{\e{}}{\w{1}}^2 \Rightarrow |\innergbar[\xbar]{\e{}}{\w{1}}|^2 \geq 1/n
%         \end{equation*}
%         so for $i\neq 1$ we have $\lvert \innergbar[\xbar]{\e{}}{\w{i}}\rvert\leq \sqrt{1-1/n}$. 
         Further,~\eqref{eqn: volume inequality with parallelpiped} and the definition of $\{\e{i}\}_{i=1}^n$ give
         \begin{align*}
              \Leb{\sublevelset} \geq C^{-1} d(\plane{\w{1}}{},\plane{-\w{1}}{})\prod \limits_{i=2}^n d(\plane{\e{i}}{},\plane{-\e{i}}{}) \geq C^{-1} d(\plane{\w{1}}{},\plane{-\w{1}}{}) \prod \limits_{i=1}^n d(\plane{\e{i}}{},\plane{-\e{i}}{}),
         \end{align*}
        here for the second inequality we used that $d(\plane{\e{1}}{},\plane{-\e{1}}{})$ is bounded from above by the diameter of $\outerdomcoord{\xbar}$, and thus by a universal constant. To finish, let $[p_1, p_2]\subset\subcoord$ be the line segment in the $v$ direction which achieves the length $\supsegment{\subcoord}{\e{}}$. Then by Remark~\ref{rem: beer bet} and~\eqref{eqn: maximal inner product}, we have
         \begin{align*}
         	   d(\plane{\w{1}}{},\plane{-\w{1}}{}) &=d(\plane{\w{1}}{},p_0)+d(p_0, \plane{-\w{1}}{})\\
	   &\geq \innergbar[\xbar]{p_2-p_1}{\w{1}}\\	
	   &=\gbarnorm[\xbar]{p_2-p_1}\cdot\innergbar[\xbar]{\e{}}{\w{1}}\\   	   
%	   &=\sup_{p_1,p_2\in \subcoord}{\sum_{i=1}^n\innergbar[\xbar]{\w{1}}{\e{i}}\innergbar[\xbar]{p_1-p_2}{\e{i}}}\\
%	&\geq \sup_{p_1,p_2\in \subcoord}{\innergbar[\xbar]{\w{1}}{\e{1}}\innergbar[\xbar]{p_1-p_2}{\e{1}}}\\
	&\geq \supsegment{\subcoord}{\e{}}/n.
	            \end{align*}
    \end{proof}	      
 
The three previous Lemmas~\ref{lem: dual segments length 2},~\ref{lem: perturbed plane}, and~\ref{lem: trapping planes} culminate in the following result. 
%We aim to find a $c$-function $\mountain_i$ for each index $i$ such that:
%\begin{enumerate}
%\item The 
%\item The 
%\end{enumerate}
    \begin{lem}\label{lem: good c-half spaces}	
         Let $\e{}\in\cotanspMbar{\xbar}$ be of unit length and let $\{\e{i}\}_{i=1}^n$ be the basis associated to it by Lemma \ref{lem: trapping planes}. Then, for some positive universal constant $C$ we have the following:

         \begin{enumerate}
              \item For each $1\leq i \leq n$, there exists a $c$-function $\mountain_{i}$ such that $\mountain_{i}\leq \mountain$ in $\sublevelset$, and
	         \begin{equation}\label{eqn: quotient bounds}
	              C^{-1}\frac{\mountain(x_0)-u(x_0) }{d(\xbar,\xbar_{i})d(p_{(x_0, \xbar)},\plane{\e{i}}{})} \leq \frac{\mountain(x_0)-u(x_0)}{\mountain(x_0)-\mountain_{i}(x_0)}\leq C	
	         \end{equation}
	         where $\xbar_i$ is the foci of $\mountain_i$.
	         \item With $\xbar_i$ as defined above we have the volume bound     
	         \begin{align}\label{eqn: convex hull size}
	              \prod_{i=1}^n{\gnorm[x_0, \xbar]{\pbar_{(x_0, \xbar_i)}-\pbar_{(x_0, \xbar)}}}\leq C\Leb{\ch{\{\pbar_{(x_0, \xbar)},\; \pbar_{(x_0, \xbar_1)}, \;\ldots\;, \pbar_{(x_0,\xbar_n)}} \}}.
	         \end{align}
         \end{enumerate}
    \end{lem}
    
    \begin{proof}
         %Let us write $ \lambda := \mountain_0(x_0)-u(x_0)$. 
         For each $i$ we will obtain $\mountain_{i}$ in one of two different ways, depending on the distance $d(p_0, \plane{\e{i}}{}$). First, since $\plane{\e{i}}{}$ is supporting to $\subcoord$, we see that $d(p_0, \partial \subcoord)\leq d(p_0, \plane{\e{i}}{})$. Then, since $u$ is uniformly Lipschitz with constant depending on $c$ and $\diam{\left(\outertarget\right)}$, and we also have $u=\mountain$ on $\partial \sublevelset_0$, we note that for a universal $C>0$
         \begin{equation*}
             \mountain(x_0)-u(x_0) \leq C d(p_0,\plane{\e{i}}{}).
         \end{equation*}
% comes from the Lipschitz norm of the cost, here we used the fact that $\mountain=u$ on $\partial \sublevelset_0$ and in particular at the point where $\coord{\sublevelset_0}{\xbar_0}$ is touched by $\Pi_i$. 
         Then let $\mountain^{\e{i}}_t$ defined for $t\in[0, d]$ be given by Lemma \ref{lem: dual segments length 2}. By combining the above inequality with the first inequality in~\eqref{eqn: mountain plane distance inequality}, for any $t\in[0, d]$ we have	 
         \begin{equation}\label{eqn: lambda ratio bound}
              \frac{\mountain(x_0)-u(x_0)}{\mountain(x_0)-\mountain^{\e{i}}_{t}(x_0)} \leq \frac{C d(p_0,\plane{\e{i}}{})}{C^{-1}td(p_0,\plane{\e{i}}{})} \leq Ct^{-1}.
         \end{equation}
%         On the other hand, by Proposition~\ref{prop: plane distances},
%         \begin{equation*}
%              \mountain(x_0)-\mountain^i_{t}(x_0) \leq Ct\left ( d(p_0,\Pi_i)+t \right )	
%         \end{equation*}
         Now, fix some $0<\delta<d$ small which will be universal, to be determined later (recall that $d$ here is universal). 
         
         First, suppose $d(p_0,\plane{\e{i}}{}) \geq \delta$. By applying the second inequality in~\eqref{eqn: mountain plane distance inequality} and then taking $t=\delta$ in~\eqref{eqn: lambda ratio bound} we obtain
         \begin{align*}
         	  C^{-1}\delta^{-1} \frac{\mountain(x_0)-u(x_0) }{d(p_0,\plane{\e{i}}{})}&\leq \frac{C(\mountain(x_0)-u(x_0))}{\delta(\delta+d(p_0, \plane{\e{i}}{}))} \notag\\
	          &\leq \frac{C(\mountain(x_0)-u(x_0))}{\mountain(x_0)-\mountain^{\e{i}}_{\delta}(x_0)}\notag\\
	          &\leq C\delta^{-1}.
         \end{align*}
         Thus for $i$ such that $d(p_0,\plane{\e{i}}{}) \geq \delta$, we will set $\mountain_{i}:=\mountain^{\e{i}}_\delta$. Since $d(\xbar,\xbar_{i}(\delta))\sim\delta$ with a universal constant of proportionality by~\eqref{Twist} and~\eqref{Nondeg}, and $\delta$ will be chosen to be universal, we obtain~\eqref{eqn: quotient bounds} in this case.
%         
%          we also obtain the bounds
%         \begin{align}
%         	  C^{-1}\delta^{-1} \frac{\mountain(x_0)-u(x_0) }{d(p_0,\plane{\e{i}}{})}\leq \frac{\mountain(x_0)-u(x_0)}{\mountain(x_0)-\mountain_{i}(x_0)}\leq C\delta^{-1}.
%         \end{align}
         %where $\xbar_i$ is the focus of $\mountain^i_\delta$. 
         
         Now suppose $d(p_0,\plane{\e{i}}{})< \delta$. Apply Lemma \ref{lem: perturbed plane} to $\e{}=\e{i}$ with some $\rho>0$ to be determined (which will again be universal), and let $\mountainhat^{\e{i}}$, $x_{\e{i}}$, $\xbar_{\e{i}}$, and $\ehat{\e{i}}$ be as given by the lemma. For ease of notation set 
         \begin{align*}%\label{eqn: second case notation}
         \mountain_{i}:=\mountainhat^{\e{i}},\ x_i:=x_{\e{i}},\ \xbar_i:=\xbar_{\e{i}},\ \ehat{i}:=\ehat{\e{i}},\  p_{i}:=p_{(x_{\e{i}}, \xbar)}.
         \end{align*}
          Since $\mountain_{i}$ is supporting to $u$ from below at $x_{i}$ in this case, we have
         \begin{equation}\label{eqn: lambda bound 2}
              \frac{\mountain(x_0)-u(x_0)}{\mountain(x_0)-\mountain_{i}(x_0)} \leq 1.
         \end{equation}
         Also, since $x_{i}\in\partial \sublevelset_0$, we see $\mountain_{i}(x_{i})=u(x_{i})=\mountain(x_{i})$ and thus for some universal $C>0$,
         \begin{align*}
              \mountain(x_0)- \mountain_{i}(x_0)&=-c(x_0, \xbar)+c(x_{i}, \xbar)-(-c(x_0, \xbar_{i})+c(x_{i}, \xbar_{i}))\\
              &\leq Cd(\xbar,\xbar_{i})\gbarnorm[\xbar]{p_0-p_{i}},
         \end{align*}
         hence from~\eqref{eqn: close to touching point} we conclude that
         \begin{equation*}
              \mountain(x_0)- \mountain_{i}(x_0)\leq C\rho^{-1}d(\xbar,\xbar_{i})d(p_0,\plane{\e{i}}{}).
         \end{equation*} 
         Then by combining with~\eqref{eqn: lambda bound 2}, we obtain
         \begin{equation}\label{eqn: quotient bounds II}
              C^{-1}\rho \frac{\mountain(x_0)-u(x_0)}{d(\xbar,\xbar_{i})d(p_0,\plane{\e{i}}{})} \leq \frac{\mountain(x_0)-u(x_0)}{\mountain(x_0)-\mountain_{i}(x_0)} \leq 1.    	  
         \end{equation}
Since $\rho$ will be universal, we again obtain~\eqref{eqn: quotient bounds} in this case.
%         Now we have $c$-functions $\{ \mountain_i\}_{i=1}^n$ which satisfy the bounds in the first part of the Lemma, with a constant $C$ that is universal save for its dependence on the parameters $\rho$ and $\delta$. 

         We now make a universal choice of $\delta$ and $\rho$, but in such a way that the volume bound~\eqref{eqn: convex hull size} holds. Note that in all cases, $\xbar_i \neq \xbar$, so we may consider the unit vectors
         \begin{equation*}
              \wbar{i} := \frac{\pbar_{(x_0,\xbar_i)}-\pbar_{(x_0,\xbar)}}{\gnorm[x_0,\xbar]{\pbar_{(x_0,\xbar_i)}-\pbar_{(x_0,\xbar)}}}.
         \end{equation*}
In order to obtain~\eqref{eqn: convex hull size}, we must now control the angles between each of the $\wbar{i}$. First note that if $i$ is such that $d(p_0,\plane{\e{i}}{})\geq \delta$, by the choice of $\mountain_i$ we immediately obtain that $\wbar{i}=\transforminv{x_0}{\xbar}\raisebar{\e{i}}$. On the other hand, suppose we are in the second case $d(p_0,\plane{\e{i}}{})< \delta$. 
         Consider the mapping 
         \begin{align*}
         x\mapsto \frac{\transform{x}{\xbar}(\pbar_{(x, \xbar_{i})}-\pbar_{(x, \xbar)})}{\gbarnorm[\xbar]{\transform{x}{\xbar}(\pbar_{(x, \xbar_{i})}-\pbar_{(x, \xbar)})}}
         \end{align*}%\transform{x}{\xbar}(\pbar_{(x,\xbar_{i})}-\pbar_{(x, \xbar)})$
          from $\outerdom$ to $\tanspMbar{\xbar}$. Since $c$ is $C^3$, we can see that this mapping is uniformly Lipschitz in $x$, with a Lipschitz constant bounded by $C$ %d(\xbar_{i}, \xbar)$
           for some universal $C>0$. 
           %At the same time, we see that $\gnorm[x, \xbar]{\pbar_{(x,\xbar_{i})}-\pbar_{(x,\xbar)}}\sim d(\xbar_{i}, \xbar)$, with a universal constant of proportionality, independent of $x\in\outerdom$. 
           Hence, we may compute        
         \begin{align}
              &\gnorm[x_0, \xbar]{\wbar{i}-\transforminv{x_0}{\xbar}\raisebar{\ehat{i}}}\notag\\
              &= \gbarnorm[\xbar]{\transform{x_0}{\xbar}\wbar{i}-\raisebar{\ehat{i}}}\notag\\
              &= \gbarnorm[\xbar]{\frac{\transform{x_0}{\xbar}(\pbar_{(x_0, \xbar_{i})}-\pbar_{(x_0, \xbar)})}{\gbarnorm[\xbar]{\transform{x_0}{\xbar}(\pbar_{(x_0, \xbar_{i})}-\pbar_{(x_0, \xbar)})}}-\frac{\transform{x_{i}}{\xbar}(\pbar_{(x_{i}, \xbar_{i})}-\pbar_{(x_{i}, \xbar)})}{\gbarnorm[\xbar]{\transform{x_{i}}{\xbar}(\pbar_{(x_{i}, \xbar_{i})}-\pbar_{(x_{i}, \xbar)})}}}\notag\\
              &\leq Cd(x_{i}, x_0)\notag\\
              &\leq C\gbarnorm[\xbar]{p_{i}-p_{0}}\notag\\
              &\leq C \rho^{-1} d(p_0,\plane{\e[]{i}}{})\notag\\
              &\leq C\frac{\delta}{\rho}\label{eqn: first rho bound},
         \end{align}
         we have used~\eqref{eqn: close to touching point} to obtain the second to last line. 
         At the same time, by~\eqref{eqn: normal is close}, we see that
         \begin{align}
              \gnorm[x_0, \xbar]{\transforminv{x_0}{\xbar}\raisebar{\e{i}}-\transforminv{x_0}{\xbar}\raisebar{\ehat{i}}}&=\gbarnorm[\xbar]{\raisebar{\e{i}}-\raisebar{\ehat{i}}}\notag\\
              &=\gbarnorm[\xbar]{\e{i}-\ehat{i}}\notag\\
              &\leq 2\rho\label{eqn: second rho bound}.
         \end{align}
         Thus, by taking $\rho:=\sqrt{\delta}$ and combining~\eqref{eqn: first rho bound} with~\eqref{eqn: second rho bound}, we obtain the inequality
         \begin{align*}
              \gnorm[x_0, \xbar]{\transforminv{x_0}{\xbar}\raisebar{\e{i}}-\wbar{i}}\leq C\sqrt{\delta}
         \end{align*}
         for some universal $C>0$. 
         
         Since $\transforminv{x_0}{\xbar}\raisebar{\e{i}}$ and $\wbar{i}$ are all unit vectors in $\g{x_0, \xbar}$, we may take $\delta$ sufficiently small (which can be done in a universal way) to deduce that the basis $\{\wbar{i}\}_{i=1}^n$ can be obtained by applying arbitrarily small rotations to the elements in the basis $\{\transforminv{x_0}{\xbar}\raisebar{\e{i}}\}_{i=1}^n$ (which, in turn, spans a parallelepiped that is uniformly close to the orthogonal basis by Lemma~\ref{lem: trapping planes}).
         Now note that by~\eqref{Nondeg} and the compactness of the domains $\outerdom^{\cl}$ and $\outertarget^{\cl}$, the volume defined by the inner product $\g{x_0, \xbar}$ is comparable to the Riemannian volume $\Leb{\cdot}$ on $\cotanspMbar{\xbar}$, with a constant of proportionality that is universal (in particular, independent of $x_0$). As a result, we can conclude the bound
         \begin{align*}
              &\Leb{\ch{\{0,\ \gnorm[x_0, \xbar]{\pbar_{(x_0, \xbar_i)}-\pbar_{(x_0, \xbar)}}\transforminv{x_0}{\xbar}\raisebar{\e{i}}\mid 1\leq i\leq n\}}}\\
              &\qquad\leq C \Leb{\ch{\{0,\ \gnorm[x_0, \xbar]{\pbar_{(x_0, \xbar_i)}-\pbar_{(x_0, \xbar)}}\wbar{i}\mid 1\leq i\leq n\}}}\\
              &\qquad=C \Leb{\ch{\{\pbar_{(x_0, \xbar)},\ \pbar_{(x_0, \xbar_i)}\mid 1\leq i\leq n\}}}
         \end{align*}
         for a universal constant $C>0$. 
         %Then, first by~\eqref{eqn: maximal inner product} and recalling that $\{\e{i}\}_{i=2}^n$ is an orthonormal collection, and then by~\eqref{Nondeg} and the compactness of the domains $\outerdom^{\cl}$ and $\outertarget^{\cl}$, we can see that          
                  Finally, first by taking $\lambda_i=\gnorm[x_0, \xbar]{\pbar_{(x_0, \xbar_i)}-\pbar_{(x_0, \xbar)}}$ in~\eqref{eqn: constructed basis is nondegenerate} for each $i$, and then by~\eqref{Nondeg} and the compactness of the domains $\outerdom^{\cl}$ and $\outertarget^{\cl}$, we can see that 
         \begin{align*}
              \prod_{i=1}^n{\gnorm[x_0, \xbar]{\pbar_{(x_0, \xbar_i)}-\pbar_{(x_0, \xbar)}}}&\leq C \Leb{\ch{\{0,\ \gnorm[x_0, \xbar]{\pbar_{(x_0, \xbar_i)}-\pbar_{(x_0, \xbar)}}\e{i}\mid 1\leq i\leq n\}}}\\
              &\leq C \Leb{\ch{\{0,\ \gnorm[x_0, \xbar]{\pbar_{(x_0, \xbar_i)}-\pbar_{(x_0, \xbar)}}\transforminv{x_0}{\xbar}\raisebar{\e{i}}\mid 1\leq i\leq n\}}}
         \end{align*}
         completing the proof of~\eqref{eqn: convex hull size}.
    \end{proof}

    Before we prove Theorem \ref{thm: aleksandrov} we shall show that the $c$-polar dual of $\{\mountain_1 \leq \mountain\}$ contains a $c$-segment with a direction and length determined by the foci of $\mountain$ and $\mountain_1$, and the height $\lambda$.
 
    \begin{lem}\label{lem: dual segments length}
         Suppose $\mountain$ and $\mountainhat$ are $c$-functions with foci $\xbar$ and $\xbarhat$ in $\innertarget$. Also suppose that $x_0$ is in the interior of $\halfspace := \{\mountainhat \leq \mountain\}$. Then for any $\lambda>0$ we have
         \begin{equation*}
              \left \{ \xbar (t) \mid \forall\; 0<t<\min {\left\{1,M^{-1}\tfrac{\lambda}{\mountain(x_0)-\mountainhat(x_0)}\right\}} \right \}\subset \halfspace^c_{x_0, \mountain,\lambda},
         \end{equation*}
         where $M$ is the constant in~\eqref{QConv} and $\xbar(t)$ is the $c$-segment with respect to $x_0$ from $\xbar$ to $\xbarhat$.
    \end{lem}

\begin{proof}
%Let
%\begin{align*}
% \halfspace:=\{x\in\outerdom\mid\mountain_1(x)\leq \mountain_0(x)\},
%\end{align*}
%fix an arbitrary $x\in \halfspace$, and 
Let us define the $c$-function
\begin{equation*}
              \mountain_t(x):= -c(x,\xbar(t))+c(x_0,\xbar(t)).
         \end{equation*}
         Observe that it is sufficient to show
         \begin{align*}
             \mountain_t(x)+\mountain(x_0)-\mountain(x)&\leq \lambda,\qquad \forall\; 0<t<\min {\left\{1,M^{-1}\tfrac{\lambda}{\mountain(x_0)-\mountainhat(x_0)}\right\}}
         \end{align*}
         for any $x\in \partial\halfspace$, to this end let us fix $x\in \partial\halfspace$.
%         , and define $f:[0,1] \to \mathbb{R}$ by
%         \begin{equation*}
%              f(t):= -c(x,\xbar(t))+c(x_0,\xbar(t)).
%         \end{equation*}
         If $\mountain_1(x)-\mountain_0(x)\leq 0$, then by~\eqref{eqn: gLp} we conclude for every $t \in [0,1]$ that 
         \begin{align*}
              \mountain_t(x)+\mountain(x_0)-\mountain(x)=\mountain_t(x)-\mountain_0(x)\leq 0\leq \lambda.
         \end{align*}
         Otherwise, $\mountain_1(x)-\mountain_0(x)>0$, in which case due to~\eqref{QConv} and the fact that $\mountainhat(x)\leq \mountain(x)$ we obtain
         \begin{align*}
              \mountain_t(x)+\mountain(x_0)-\mountain(x)&\leq Mt(\mountain_1(x)-\mountain_0(x))\\
%              &\leq Mt(f(1)-f(0))\\
              & = Mt(\mountainhat(x)-\mountainhat(x_0)-\mountain(x)+\mountain(x_0))\\
             % & \leq Mt(u(x)-\mountain(x)+\mountain(x_0)-\mountainhat(x_0))\\
              &\leq Mt(\mountain(x_0)-\mountainhat(x_0)).
         \end{align*}
         We obtain the desired inequality whenever $t < \min\{ 1, M^{-1}\tfrac{\lambda}{\mountain(x_0)-\mountainhat(x_0)}\}$, proving the lemma.
    \end{proof}

    Finally, the Aleksandrov estimate follows by combining Lemma \ref{lem: good c-half spaces} and Lemma \ref{lem: dual segments length}.

    \begin{proof}[Proof of Theorem~\ref{thm: aleksandrov}]
         Let $\{\e{i}\}_{i=1}^n$ and $\mountain_i$ with foci $\xbar_i$  be obtained by applying Lemma~\ref{lem: good c-half spaces} to $\e{}$. Since $\sublevelset \subset \{ \mountain_i\leq \mountain\}$ we may apply Lemma~\ref{lem: dual segments length} with $\lambda := \mountain(x_0)-u(x_0)$ to obtain for each $i$, a portion of a line segment
         \begin{align*}
              \gammabar_i := \left \{ (1-t)\pbar_0 +t\pbar_{(\xbar_i,\xbar)} \mid \forall\; 0<t<\min {\left\{1,M^{-1}\tfrac{\lambda}{\mountain(x_0)-\mountain_{i}(x_0)}\right\}} \right \}\subset \coord{\sublevelset^c_{x_0, \mountain,\lambda}}{x_0}.
         \end{align*}
         %Note that $\gamma_{i}$ is a segment with $\pbar_0$ as one if endpoints and direction given by $\ebar[*]{i}$. 
         Since $\coord{\sublevelset^c_{x_0, \mountain,\lambda}}{x_0}=\coord{\partial_cK_{x_0, \mountain, \sublevelset, \lambda}(x_0)}{x_0}$ by~\eqref{eqn: polar duals and cone subdifferential equality} of Lemma~\ref{lem: map of cones}, the set is convex by Corollary~\ref{cor: local to global}. Thus, we see that
          \begin{align}\label{eqn: segments inclusion}
              \ch\left \{\gammabar_{i}\mid 1\leq i\leq n\right \}\subset \coord{\sublevelset^c_{x_0, \mountain,\lambda}}{x_0}.
         \end{align} 
         
         We will now compute the length of $\gammabar_{i}$. By the second inequality in~\eqref{eqn: quotient bounds} of Lemma~\ref{lem: good c-half spaces},
         \begin{align*}
              \frac{\lambda}{\mountain(x_0)- \mountain_{i}(x_0)}\leq C
         \end{align*}
         for all $i$ and a universal constant $C$. Thus, there exists another universal constant $C$ such that 
         \begin{align*}
              C^{-1}\frac{\lambda}{\mountain(x_0)-\mountain_{i}(x_0)}\leq\min {\left\{1,M^{-1}\frac{\lambda}{\mountain(x_0)-\mountain_{i}(x_0)}\right\}}
         \end{align*}
         for each $i$. Hence we can calculate
         \begin{align*}
              \gnorm[x_0, \xbar]{\gammabar_{i}}&= d(\xbar_0,\xbar_i)\frac{C^{-1}\lambda}{\mountain(x_0)-\mountain_{i}(x_0)} \geq C^{-1}\frac{\lambda}{d(p_0, \plane{\e{i}}{})}
         \end{align*}
         where this time we have used the first inequality in~\eqref{eqn: quotient bounds} of Lemma~\ref{lem: good c-half spaces}. Combining this bound with inequality~\eqref{eqn: convex hull size} of  Lemma~\ref{lem: good c-half spaces}, we obtain
         \begin{equation*}
         	  \Leb{\ch\left \{\gammabar_{i}\mid 1\leq i\leq n\right \}} \geq C^{-1}\lambda^n \prod \limits_{i=1}^n\frac{1}{d(p_0, \plane{\e{i}}{})},
         \end{equation*}
         and since $d(p_0,\plane{\e{i}}{}) \leq d(\plane{-\e{i}}{},\plane{\e{i}}{})$ for all $i$, we have
         \begin{equation*}\label{eqn: convex hull bound}
         	  \Leb{\ch\left \{\gammabar_{i}\mid 1\leq i\leq n\right \}} \geq C^{-1}\lambda^n \frac{1}{d(p_0,\plane{\e{1}}{})}\prod \limits_{i=2}^n\frac{1}{d(\plane{\e{i}}{}, \plane{\e{i}}{})}.
         \end{equation*}
         We repeat the same argument with $-v$ instead of $v$, obtaining a second convex hull inside $\coord{\sublevelset^c_{x_0, \mountain,\lambda}}{x_0}$ whose measure has an analogous lower bound. Since the intersection of these two convex hulls lies in a lower dimensional hyperplane, it must have measure zero, and the volume of the union of the two sets is bounded below by the sum of their volumes. Hence recalling the inclusion~\eqref{eqn: segments inclusion} we obtain
         \begin{align}
              \lambda^n \left(\frac{1}{d(p_0, \plane{\e{1}}{})}+\frac{1}{d(p_0, \plane{-\e{1}}{})}\right) \prod \limits_{i=2}^n \frac{1}{d(\plane{-\e{i}}{},\plane{\e{i}}{})}&\leq C \Leb{\coord{\sublevelset^c_{x_0, \mountain,\lambda}}{x_0}}\notag\\
              &\leq C\Leb{\partial_cu(\sublevelset_s)},\label{eqn: aleksandrov final}
         \end{align}
         the last inequality resulting from~\eqref{eqn: cone subdifferential contained in u subdifferential} in Lemma~\ref{lem: map of cones}, along with~\eqref{Nondeg}.
         Noting that
         \begin{align*}
              d(p_0,\plane{-\e{i}}{})+d(p_0,\plane{\e{i}}{}) & =d(\plane{-\e{i}}{},\plane{\e{i}}{}),\\
              \min \{d(p_0,\plane{-\e{i}}{}),d(p_0,\plane{\e{i}}{})\} & = d(p_0,\plane{-\e{i}}{}\cup \plane{\e{i}}{})	
         \end{align*}
         we get the lower bound,
         \begin{align*}
              \left(\frac{1}{d(p_0, \plane{\e{1}}{})}+\frac{1}{d(p_0, \plane{-\e{1}}{})}\right) \prod \limits_{i=2}^n \frac{1}{d(\plane{-\e{i}}{},\plane{\e{i}}{})}	 & \geq \frac{d(\plane{-\e{1}}{},\plane{\e{1}}{})}{d(p_0,\plane{-\e{1}}{} \cup \plane{\e{1}}{})}\prod \limits_{i=1}^n \frac{1}{d(\plane{-\e{i}}{},\plane{\e{i}}{})}.
         \end{align*}
         Then combined with the bound~\eqref{eqn: S comparable to box} from Lemma \ref{lem: trapping planes}, we conclude that the left hand side in~\eqref{eqn: aleksandrov final} is no smaller than
         \begin{equation*}
         	  C^{-1}\lambda^n\frac{d(\plane{-\e{1}}{},\plane{\e{1}}{})}{d(p_0,\plane{-\e{1}}{}\cup \plane{\e{1}}{})}\supsegment{\subcoord}{\e{}} \Leb{\sublevelset}^{-1}.
         \end{equation*}
         Substituting this into \eqref{eqn: aleksandrov final} and rearranging the terms, the theorem is proved.
\end{proof}

%%%%%%%%%%%%%%%%%%%%%%%%%%%%%%%%%%%%%%%%%%%%%%%
%%%%%%%%%% LOCALIZATION PROPERTY %%%%%%%%%%%%%%
%%%%%%%%%%%%%%%%%%%%%%%%%%&%%%%%%%%%%%%%%%%%%%%
\section{Caffarelli's Localization Theorem revisited.}\label{section: localization}

    The purpose of this section is to prove a geometric property (cf. Theorem~\ref{thm: localization}) of Aleksandrov solutions of~\eqref{OT problem}. This property was first proven by Caffarelli in~\cite{C90}, in the special case $c(x,y)=-x\cdot y$ where $x,y\in\real^n$ and it has important consequences including strict convexity and $C^{1,\alpha}$ regularity for solutions of the classical Monge-Amp\`{e}re equation. Informally, it says that solutions cannot have ``ridges'' except when these extend all the way to the boundary of the domain.

It is here that we will combine the two estimates Lemma~\ref{lem: sharp growth estimate} and Theorem~\ref{thm: aleksandrov}. However, the main work of this section goes toward finding the appropriate sublevel sets to which we can apply these two estimates. Roughly speaking, starting with a contact set $\{u=\mountain_0\}$, we want to ``tilt'' the supporting $c$-function in such a way that the new $c$-functions will ``chop'' a small portion of the contact set, and converge back to this chopped portion as the tilting becomes smaller.

    We first recall some useful definitions from convex geometry (see for example,~\cite{Roc70}).

    \begin{DEF}\label{def: normal cones}
         Suppose that $\mathcal{\arbitrary}$ is a convex subset of an inner product space $V$ with inner product $\left(\cdot, \cdot\right)$, and $p_e\in \partial \mathcal{\arbitrary}$. Then, we define the \emph{strict normal cone of $\mathcal{\arbitrary}$ at $p_e$} and \emph{normal cone of $\mathcal{\arbitrary}$ at $p_e$} respectively by
         \begin{align*}
              \normal^0_{p_e}\left(\mathcal{\arbitrary}\right):&=\{q\in V\mid \left( q, p-p_e\right) <0,\ \forall p\neq p_e\in \mathcal{\arbitrary}\},\\
              \normal_{p_e}\left(\mathcal{\arbitrary}\right):&=\{q\in V\mid \left( q, p-p_e\right) \leq 0,\ \forall p\in \mathcal{\arbitrary}\}.
         \end{align*}
         If $\normal^0_{p_e}\left(\mathcal{\arbitrary}\right)\neq \emptyset$, $p_e$ is called an \emph{exposed point} of $\mathcal{\arbitrary}$.
    \end{DEF}

    \begin{rem}\label{rem: remark on normal cones}
         Recall that both $\normal_{p_e}\left(\mathcal{\arbitrary}\right)$ and $\normal^0_{p_e}\left(\mathcal{\arbitrary}\right)$ are convex and $1$-homogeneous, while $\normal_{p_e}\left(\mathcal{\arbitrary}\right)$ is closed, and contains $0$ and at least one nonzero vector for any $p_e\in\partial\mathcal{\arbitrary}$.
    \end{rem}

    \begin{DEF}\label{def: c-extremal points}
         Given $\xbar_0\in\outertarget$ and a set $\arbitrary\subset \outerdom$ which is $c$-convex with respect to $\xbar_0$, we say that a point $x_e\in \arbitrary$ is \emph{a $c$-extremal point ($c$-exposed point) of $\arbitrary$ with respect to $\xbar_0$} if $p_{(x_e, \xbar_0)}$ is an extremal point (exposed point) in the usual sense of the convex set $\coord{\arbitrary}{\xbar_0}$.
    \end{DEF}

    We begin by constructing a certain family of $c$-functions associated to each contact set, which will play a crucial role in ``chopping'' the contact sets near its $c$-exposed points.

    \begin{lem}\label{lem: localization claim}
         Let $u$ be a $c$-convex function on $\outerdom$, $\mountain_0$ be a $c$-function that is supporting to $u$ somewhere in $\outerdom$ with focus $\xbar_0\in \innertarget$, and define 
         \begin{align*}
              \sublevelset_0:&=\{u=\mountain_0\}.
         \end{align*}
         Also, suppose $x_e\in \partial \sublevelset_0$ is a $c$-exposed point of $\sublevelset_0$ with respect to $\xbar_0$, and $\e{0} \in \normal^0_{p_{(x_e,\xbar_0)}}\left(\subzerocoord\right)$ is a unit vector. 
         Then for any fixed $\delta>0$ there is a family of $c$-functions $\{\mountain^\delta_t\}$ depending on $\sublevelset_0$ and $\e{0}$, such that for all sufficiently small and positive $t$ we have  $\mountain^\delta_t \neq \mountain_0$ and
         \begin{align}
              \coord{\sublevelset_{\delta, t}}{\xbar_0}&\subset B_{\delta}(p_{(x_e, \xbar_0)}),\label{eqn: chopped sections close to exposed point}\\
              u(x_e)&<\mountain^\delta_t(x_e),\label{eqn: exposed point is inside chopped section}
         \end{align}
         where $\sublevelset_{\delta, t}:=\{u\leq \mountain^\delta_t\}$.
    \end{lem}

    \begin{proof}
         The proof is motivated by an idea used in \cite[Lemma 1]{C92}. Let us write 
         \begin{align*}
              p_e:=p_{(x_e, \xbar_0)},\ \pbar_0:=\pbar_{(x_e,\xbar_0)},
         \end{align*}
         and for any $\rho\in \real$, $x\in\outerdom$, $\xbar\in\outertarget$, and nonzero $\e{}\in\cotanspMbar{\xbar}$, define
    \begin{align*}
         \extremalhalfspace[x,\xbar,\e{}]{\rho}:&=\{ p\in\cotanspMbar{\xbar}\mid \innergbar[\xbar]{p-p_{(x,\xbar)}}{\tfrac{\e{}}{\gbarnorm[\xbar]{\e{}}}}\geq \rho\}.
    \end{align*}
         We first show the following auxiliary claim: there exists a constant $0<\tau_\delta<\delta$ depending on $\sublevelset_0$ and $\e{0}$ such that
\begin{align}\label{eqn: chopped section is small}
 \subzerocoord\cap \extremalhalfspace{-\tau_\delta}&\subset B_\delta(p_{(x_e, \xbar_0)}).
\end{align}
 Suppose the inclusion does not hold for any $\tau_\delta>0$. Then, there exists a sequence of points $p_k$ in $\subzerocoord$  such that 
         \begin{align*}
              \gbarnorm[\xbar_0]{p_k-p_e}\geq\delta
         \end{align*}
         while
         \begin{align*}
              \innergbar{p_k-p_e}{\e{0}}\geq -\tfrac{1}{k}
         \end{align*}
         for each $k$. By compactness, we can assume a subsequence $p_k\to p_\infty\in\subzerocoord$ as $k\to\infty$, thus $p_\infty\neq p_e$, but $\innergbar{p_\infty-p_e}{\e{0}}\geq 0$ which contradicts $\e{0}\in\normal^0_{p_e}\left(\subzerocoord\right)$. Hence, for some choice of $0<\tau_\delta\leq \delta$ depending on $\sublevelset_0$, we obtain~\eqref{eqn: chopped section is small}.
%\begin{equation*}
%\subzerocoord\cap \extremalhalfspace{-\tau_\delta}\subset B_\delta(p_e)
%\end{equation*}
%and by defining
%\begin{equation*}
%\mountain^\delta_s:=\mountain_{s, \tau_\delta}
%\end{equation*}
%for this choice of $\tau_\delta$ we obtain~\eqref{eqn: chopped section is small}.

         We will now define $\mountain^\delta_t$. Note that by~\eqref{Nondeg}, we have $\transforminv{x_e}{\xbar_0}\raisebar{\e{0}}\neq 0$ (see Definition \ref{def: music} for the definition of $\raisebar{\e{0}}$). 
%         Let $\ebar{0}\in\cotanspM{x_e}$ be defined by the relation
%         \begin{equation*}\label{eqn: dual ebar def}
%              \ebar{0}:=\transforminv{x_e}{\xbar_0}\raisebar{\e{0}},
%         \end{equation*}
%         by~\eqref{Nondeg} we have $\ebar{0}\neq 0$.
%\BlueComment{$\e{0} = \left (\DDbarinv{x_e}{\xbar_0}\ebar{0}\right )^{\gbar{\xbar_0}}	$}
         Since $\innertarget$ is compactly contained in $\outertarget$, for $t$ positive and small enough, $\pbar_0+t\transforminv{x_e}{\xbar_0}\raisebar{\e{0}}$ is still in $\outertargetcoord{x_e}$, hence 
         \begin{equation*}
              \xbar(t):=\cExp{x_e}{\pbar_0+t\transforminv{x_e}{\xbar_0}\raisebar{\e{0}}}
         \end{equation*}
         is a well-defined $c$-segment for $t$ small. Then, we will consider the $c$-functions,
         \begin{equation*}
              \mountain^\delta_t(x):=-c(x, \xbar(t))+c(x_e, \xbar(t))+\mountain_0(x_e)+\tau_\delta t
         \end{equation*}
         and note that if $\mountain_0(x)\leq\mountain^\delta_t(x)$, by a Taylor expansion in $t$ we have
         \begin{align*}%\label{eqn: taylor expansion calculation}
              -\tau_\delta t&\leq-c(x, \xbar(t))+c(x_e, \xbar(t))-(-c(x, \xbar_0)+c(x_e, \xbar_0))\notag\\
              &\leq t\inner{p_{(x, \xbar_0)}-p_e}{\xbardot(0)}+Ct^2\notag\\
              &=t\inner{p_{(x, \xbar_0)}-p_e}{ \transform{x_e}{\xbar_0}(\transforminv{x_e}{\xbar_0}\raisebar{\e{0}})}+Ct^2\notag\\
              &=t\innergbar{p_{(x, \xbar_0)}-p_e}{\e{0}}+Ct^2
         \end{align*}
         hence
         \begin{equation}\label{eqn: chopping set containment}
              \coord{\sublevelset_{\delta, t}}{\xbar_0}\subset\coord{\{\mountain_0\leq\mountain^\delta_t \}}{\xbar_0} \subset \extremalhalfspace{-Ct-\tau_\delta}
         \end{equation}
         for all $t>0$ sufficiently small. By a simple compactness argument, we easily obtain~\eqref{eqn: chopped sections close to exposed point}.

         Finally, since
         \begin{equation*}
              \frac{d}{dt}\mountain^\delta_t(x_e)=\tau_\delta>0,\;\;\forall\;t>0
         \end{equation*}
         we see 
         \begin{equation*}
              u(x_e)=\mountain_0(x_e)<\mountain^\delta_t(x_e)
         \end{equation*}
         for any $t>0$ sufficiently small, proving~\eqref{eqn: exposed point is inside chopped section}.
    \end{proof}

    For the next lemma, recall the notation $\supsegment{\subcoord}{\e{}}$ and $\plane{\pm \w{}}{\subcoord}$ introduced in Definition~\ref{def: notations for aleksandrov}.

    \begin{lem}\label{lem: chopping convergence}
         Let $u$, $\mountain_0$, and $\sublevelset_0$ satisfy the same conditions as in Lemma~\ref{lem: localization claim} above. Then, for each $\delta>0$, we can find (each of the following depending on $\sublevelset_0$) some $\epsilon_0>0$, a unit length vector $\e{0}\in\normal^0_{p_{(x_e, \xbar_0)}}\left(\subzerocoord\right)$, a family of $c$-functions $m_t$ with foci $\xbar_t\in \outertarget$, and a family of unit length $\e{t}\in\cotanspMbar{\xbar_t}$ all defined for $t>0$ sufficiently small, which satisfy the following:
         \begin{align}
	          \mountain_t(x_e)&>u(x_e),\label{eqn: chopped sections have interior}\\
	          \coord{\sublevelset_{t}}{\xbar_0}&\subset B_{\delta}(p_{(x_e, \xbar_0)}),\label{eqn: chopped sections close to exposed point 2}\\
              \min \left \{\frac{\mountain_t(x_e)-u(x_e)}{\sup \limits_{\sublevelset_t} (\mountain_t-u)}\;,\;\frac{\mountain_t(x_e)-u(x_e)}{\sup \limits_{\sublevelset_t} (\mountain_t-u)+u(x)-\mountain_t(x)} \right \} &\geq \epsilon_0 ,\label{eqn: lower bound on localization ratios}\\
              \supsegment{\coord{\sublevelset_t}{\xbar_t}}{\e{t}} &\geq \epsilon_0,\label{eqn: lower bound on segment}\\
              \lim \limits_{t\to 0^+} d(p_{(x_e,\xbar_t)}, \plane{\e{s}}{\coord{\sublevelset_t}{\xbar_t}}  \cup \plane{-\e{t}}{\coord{\sublevelset_t}{\xbar_t}} )  &= 0.\label{eqn: supporting plane collapses}
         \end{align}
         Here we have written $\sublevelset_t:=\{u\leq \mountain_t\}$. 
    \end{lem}

    \begin{proof}
         We will write $p_e:=p_{(x_e, \xbar_0)}$. First we apply Lemma~\ref{lem: good supporting normal} from the Appendix to $\subzerocoord$, to obtain a unit length $\e{0}\in\normal^0_{p_e}\left(\subzerocoord\right)$ and $\lambda_0>0$, such that $p_e-\lambda_0\e{0}\in\subzerocoord$.
       %Note that by~\eqref{Nondeg}, each $\e{s}$ is well-defined and nonzero since $\e{0}\neq 0$. 
         Next fix $\delta>0$, and let $\mountain_t:=\mountain^\delta_t$ with focus $\xbar_t$ be obtained by applying Lemma~\ref{lem: localization claim} with this $\delta$. Note that $\xbar_t$ is independent of $\delta$. Then~\eqref{eqn: exposed point is inside chopped section} immediately implies~\eqref{eqn: chopped sections have interior}, while~\eqref{eqn: chopped sections close to exposed point} implies~\eqref{eqn: chopped sections close to exposed point 2}.
               
         Now we will show~\eqref{eqn: lower bound on localization ratios}. Recalling that $\mountain_0\leq u$ while $x_e\in \sublevelset_0$, 
         \begin{align*}
              \frac{\mountain_t (x_e)-u(x_e)}{\sup\limits_{\sublevelset_t }{(\mountain_t -u)}}&\geq \frac{\mountain_t (x_e)-\mountain_0(x_e)}{\sup\limits_{\sublevelset_t }(\mountain_t -\mountain_0)}\\
              &=\frac{\tau_\delta t}{\sup\limits_{y\in\sublevelset_t }[\tau_\delta t -c(y, \xbar_t)+c(y, \xbar_0)-(-c(x_e, \xbar_t)+c(x_e, \xbar_0))]}\\
              &\geq \frac{\tau_\delta t}{\tau_\delta t+t\cdot\sup\limits_{y\in\sublevelset_t }{\innergbar{p_{(y, \xbar_0)}-p_e}{\e{0}}}+Ct^2}\\
              &\geq \frac{\tau_\delta}{\tau_\delta+C\diam{\left(\outerdom\right)}+Ct},
         \end{align*}
         for some universal $C>0$.  Next note that the denominator of the second expression in the minimum in~\eqref{eqn: lower bound on localization ratios} is always nonnegative. Then, since for any $x\in\sublevelset_0$,
         \begin{align*}
              \frac{\mountain_t (x_e)-u(x_e)}{\sup\limits_{\sublevelset_t }{(\mountain_t -u)}+u(x)-\mountain_t(x)}
              \geq \frac{\mountain_t (x_e)-\mountain_0(x_e)}{\sup\limits_{\sublevelset_t }(\mountain_t -\mountain_0)+\mountain_0(x)-\mountain_t(x)},
         \end{align*}
         by a nearly identical Taylor expansion argument as above, we obtain~\eqref{eqn: lower bound on localization ratios} as long as 
         \begin{align*}
         	  0<\epsilon_0<\min\left\{\frac{\tau_\delta}{C\diam{\left(\outerdom\right)}},\; \frac{\tau_\delta}{\tau_\delta+C\diam{\left(\outerdom\right)}}\right\}.
         \end{align*}
         Next we work towards showing~\eqref{eqn: lower bound on segment}. Let $\tau:=\min{\{\tau_\delta/2, \lambda_0\}}$ and define the following for all $s>0$ sufficiently small:
         \begin{align*}
              p^t_e:&=p_{(x_e, \xbar_t)},\\
              p_{cp}:&=p_e-\tau \e{0},\ p^t_{cp}:=p_{(x_{cp}, \xbar_t)},\\
              x_{cp}:&=\cExp{\xbar_0}{p_{cp}}\\
              v_t:&=\frac{p^t_{cp}-p^t_e}{\gbarnorm[\xbar_t]{p^t_{cp}-p^t_e}}
         \end{align*}
         (the subscript $cp$ stands for ``central point'').
%        Now, since $\subzerocoord$ is convex we see that there is a nontrivial portion of the line segment between $p_e$ and $p_0$ is contained in $\subzerocoord\cap \extremalhalfspace[x_e, \xbar_0, \e{0}]{-\tau_\delta/2}$, by relabeling points if necessary, let us assume that $p_0$ is contained in this portion. 
         We first claim that
         \begin{align}
              x_{cp}\in {\sublevelset}_{t},\qquad \forall\; t>0\text{ sufficiently small}.\label{eqn: other point in tilted section}
         \end{align}
         Indeed, by a Taylor expansion in $t$,
         \begin{align*}
              \mountain_t(x_{cp})-u(x_{cp})&=\mountain_t(x_{cp})-\mountain_0(x_{cp})\\
              &=-c(x_{cp}, \xbar(t))+c(x_e, \xbar(t))+\tau_\delta t-(-c(x_{cp}, \xbar_0)+c(x_e, \xbar_0))\\
              &\geq t\innergbar{p_{cp}-p_e}{\e{0}}-Ct^2+\tau_\delta t\\
              &\geq -\tfrac{\tau_\delta}{2}t-Ct^2+\tau_\delta t\\
              &\geq 0
         \end{align*}
         for $t\geq 0$ sufficiently small, proving~\eqref{eqn: other point in tilted section}. As a result, since $\coord{\sublevelset_t}{\xbar_t}$ is convex by Corollary~\ref{cor: c-convexity of sublevel sets}, we see that the entire line segment from $p^t_{cp}$ to $p^t_e$ is contained in $\coord{\sublevelset_t}{\xbar_t}$, and parallel to $\e{t}$, hence
         
         \begin{align*}
              \supsegment{\coord{\sublevelset_t}{\xbar_t}}{\e{t}}&\geq \gbarnorm[\xbar_t]{p^t_{cp}-p^t_e}\\
              &\to \gbarnorm[\xbar_0]{p_{cp}-p_e},\ (t\to 0)\\
              &=\tau.
         \end{align*}
         Thus, if $s>0$ is sufficiently small, we obtain~\eqref{eqn: lower bound on segment} for $0<\epsilon_0<\tau/2$.
%
%         By Lemma~\ref{lem: good supporting normal}, we can easily see that         
%\begin{align*}
% \supsegment{\subzerocoord}{\xbar_0}\geq \lambda_0>0.
%\end{align*}
%\BlueComment{Then, by the lower semi-continuity of $s\mapsto \supsegment{\coord{\sublevelset_s}{\xbar_s}}{\xbar_s}$, we obtain~\eqref{eqn: lower bound on segment} as long as $0<\epsilon_0<\lambda_0/2$ and $s>0$ is sufficiently small.}
        
         Finally, suppose that~\eqref{eqn: supporting plane collapses} fails. Then, recalling Remark~\ref{rem: beer bet}, there exists $\epsilon>0$, a sequence of positive numbers $t_k$ going to zero, and points $p_k\in \coord{\sublevelset_{k}}{\xbar_k}$ such that 
         \begin{align}
         \epsilon\leq d(p^k_e, \plane{\e{k}}{\coord{\sublevelset_k}{\xbar_k}})=\innergbar[\xbar_k]{p_k-p^k_e}{\e{k}}\label{eqn: collapsing support plane contradiction}%=\innergbar[\xbar_0]{p_{(x_n, \xbar_0)}-p_e}{\e{0}}\gbarnorm[\xbar_n]{\e{n}}^{-1}\leq C\innergbar[\xbar_0]{p_{(x_n, \xbar_0)}-p_e}{\e{0}},
         \end{align} 
         where for brevity, we have written $\sublevelset_k:=\sublevelset_{t_k}$, $\xbar_k:=\xbar_{t_k}$, $p^k_e:=p^{t_k}_e$, and $\e{k}:=\e{t_k}$. By compactness of $\outerdom^{\cl}$, passing to a subsequence we may assume that $\cExp{\xbar_k}{p_k}\to x_\infty$ as $k\to\infty$ for some $x_\infty\in\outerdom^{\cl}$, and we easily see $p_{(x_\infty, \xbar_0)}\in \subzerocoord$.  However, by passing to the limit in~\eqref{eqn: collapsing support plane contradiction} we would obtain
         \begin{align*}
              0< \epsilon\leq \innergbar[\xbar_0]{p_{(x_\infty, \xbar_0)}-p_e}{\e{0}},
         \end{align*}
         contradicting that $\e{0}\in\normal_{p_e}\left(\subzerocoord\right)$, and thus we have proved~\eqref{eqn: supporting plane collapses}.
    \end{proof}

%%%%%%%%%%%%%%%%%%%%%%%%%%%%%%%%%%
%%%%%%% LOCALIZATION %%%%%%%%%%%%%
%%%%%%%%%%%%%%%%%%%%%%%%%%%%%%%%%%
    At this point, we recall a classical result from convex analysis.
    \begin{thm}[Straszewicz's Theorem {\cite[Theorem 18.6]{Roc70}}]\label{thm: exposed points}
         The collection of exposed points of a convex set is dense in the collection of extremal points of the set.
    \end{thm}
    We are now ready to begin a sequence of proofs that will culminate in showing that any contact set between a $c$-convex solution $u$ and a supporting $c$-function that is not a singleton can only have $c$-extremal points in the boundary of $\outerdom$. In the following Section~\ref{section: strict convexity}, we will rule out this final case, thus any contact set must consist of only one point, implying strict $c$-convexity of solutions $u$. We proceed by systematically ruling out $c$-extremal points in the interior of $\innerdom$, then in the interior of $\outerdom\setminus\innerdom$, and then finally on the boundary of $\innerdom$. The first two cases closely mirror the proofs given by Caffarelli in~\cite{C92}. However, this third case requires the use of an idea borrowed from Figalli, Kim, and McCann (see~\cite[Theorem 7.1]{FKM11}). In the interest of highlighting these differences, we have split the three proofs up.
    \begin{thm}[Caffarelli Localization Theorem] \label{thm: localization}
         Suppose $u$ is a $c$-convex Aleksandrov solution of~\eqref{OT problem} and $\mountain_0$ is a $c$-function with focus $\xbar_0\in\innertarget$, supporting to $u$ at some $x_0\in(\innerdom)^{\interior}$, then the set 
         \begin{equation*}
              \sublevelset_0:=\{u=\mountain_0 \}
         \end{equation*}
         is either a single point, or it contains no $c$-extremal points with respect to $\xbar_0$ interior to $\innerdom$.
    \end{thm}

    \begin{proof}
         Suppose $\sublevelset_0$ is not a single point and it contains a $c$-extremal point with respect to $\xbar_0$ in $(\innerdom)^{\interior}$. By Theorem~\ref{thm: exposed points}, there exists a $c$-exposed point $x_e$ of $\sublevelset_0$ with respect to $\xbar_0$, which is also contained in $(\innerdom)^{\interior}$. We choose $0<\delta<d(p_{(x_e, \xbar_0)}, \partial \innerdomcoord{\xbar_0})$, and let $\sublevelset_t$, $\mountain_t$ and $\e{t}$ be as given by Lemma~\ref{lem: chopping convergence}.
   
         First apply Lemma~\ref{lem: sharp growth estimate} to  $\sublevelset_{t}$ and $\mountain_t$ for each $t>0$ small, with the choice of $\arbitrary=\ellipsoid_t$, where $\coord{\ellipsoid_t}{\xbar_t}$ is the dilation by $\tfrac{1}{2M}$ of the John ellipsoid of each $\coord{\sublevelset_t}{\xbar_t}$. We thus calculate 
         \begin{align*}
              \sup\limits_{\sublevelset_t} \left (\mountain_t-u \right )^n&\geq C\Leb{\sublevelset_t}\Leb{\partial_cu(\ellipsoid_t)}\\
              &\geq C\Leb{\sublevelset_t}\Leb{\ellipsoid_t\cap\innerdom}\\
              &\geq C\Leb{\sublevelset_t}^2,
%               \leq \epsilon_0^n \sup\limits_{\sublevelset_s} \left (\mountain_s-u \right )^n\\
%              &\leq   \\
%              &\leq C\frac{d(p_{(x_e,\xbar_s)}, \plane{\e{s}}{\coord{\sublevelset_s}{\xbar_s}} \cup \plane{-\e{s}} {\coord{\sublevelset_s}{\xbar_s} })}{\supsegment{\coord{\sublevelset_s}{\xbar_s}}{\e{s}}d(\plane{\e{}}{\subcoord}, \plane{-\e{}}{\subcoord})}\Leb{\sublevelset_s}^2\leq C\epsilon_0^{-2}d(p_{(x_e,\xbar_s)}, \plane{\e{s}}{\coord{\sublevelset_s}{\xbar_s}} \cup \plane{-\e{s}} {\coord{\sublevelset_s}{\xbar_s} })\Leb{\sublevelset_s}^2,
         \end{align*}
 we have used~\eqref{eqn: main} along with the facts that $\ellipsoid_t\subset \sublevelset_t\subset \innerdom$ and $\Leb{\ellipsoid_t}\sim\Leb{\sublevelset_t}$. At the same time since each $\sublevelset_t$ is compactly contained in $\innerdom$ by~\eqref{eqn: chopped sections close to exposed point 2}, by applying Theorem~\ref{thm: aleksandrov} with each $\e{t}$ direction, we obtain
         \begin{align*}
               (\mountain_t(x_e)-u(x_e))^n
              &\leq C\frac{d(p_{(x_e,\xbar_t)}, \plane{\e{t}}{\coord{\sublevelset_t}{\xbar_t}} \cup \plane{-\e{t}} {\coord{\sublevelset_t}{\xbar_t} })}{\supsegment{\coord{\sublevelset_t}{\xbar_t}}{\e{t}}d(\plane{\e{}}{\coord{\sublevelset_t}{\xbar_t}}, \plane{-\e{}}{\coord{\sublevelset_t}{\xbar_t}})}\Leb{\sublevelset_t}\Leb{\partial_cu(\sublevelset_t)}\\
              &\leq C\epsilon_0^{-2}d(p_{(x_e,\xbar_t)}, \plane{\e{t}}{\coord{\sublevelset_t}{\xbar_t}} \cup \plane{-\e{t}} {\coord{\sublevelset_t}{\xbar_t} })\Leb{\sublevelset_t}^2,
         \end{align*}
%         \begin{align*}
%              \sup\limits_{\sublevelset_s} \left (\mountain_s-u \right )^n\; & \geq C^{-1}\Leb{\sublevelset_s}^2\\
%              (\mountain_s(x_e)-u(x_e))^n & \leq C\frac{d(p_{(x_e,\xbar_s)}, \plane{\e{s}}{\coord{\sublevelset_s}{\xbar_s}} \cup \plane{-\e{s}} {\coord{\sublevelset_s}{\xbar_s} })}{\supsegment{\coord{\sublevelset_s}{\xbar_s}}{\e{s}}^2}\Leb{\sublevelset_s}^2
%         \end{align*} 
         here we have used~\eqref{eqn: main}, and that 
         \begin{align*}
        \supsegment{\coord{\sublevelset_t}{\xbar_t}}{\e{t}} d(\plane{\e{t}}{\coord{\sublevelset_t}{\xbar_t}}, \plane{-\e{t}}{\coord{\sublevelset_t}{\xbar_t}})\geq \supsegment{\coord{\sublevelset_t}{\xbar_t}}{\e{t}}^2\geq \epsilon^2_0
         \end{align*}
         by~\eqref{eqn: lower bound on segment} from Lemma~\ref{lem: chopping convergence}. Since $\Leb{\sublevelset_t }>0$ from property~\eqref{eqn: exposed point is inside chopped section}, we can divide the second inequality above by the first, rearrange terms, and use~\eqref{eqn: lower bound on localization ratios} to obtain
         \begin{align*}
              0< \epsilon_0^{n+2} &\leq C\epsilon_0^2\left(\frac{\mountain_t(x_e)-u(x_e)}{\sup \limits_{\sublevelset_t} (\mountain_t-u)}\right)^n\\
               &\leq 	Cd(p_{(x_e,\xbar_t)}, \plane{\e{t}}{\coord{\sublevelset_t}{\xbar_t}} \cup \plane{-\e{t}} {\coord{\sublevelset_t}{\xbar_t} }).
         \end{align*}
%         Noting that $\Leb{\sublevelset_s }>0$ from property~\eqref{eqn: exposed point is inside chopped section} and using the lower bound for $\supsegment{\coord{\sublevelset_s}{\xbar_s}}{\e{s}}$, we get
%         \begin{equation*}
%               0<\delta^3 C^{-1}\leq d(p_{(x_e,\xbar_s)}, \plane{\e{s}}{\coord{\sublevelset_s}{\xbar_s}} \cup \plane{-\e{s}} {\coord{\sublevelset_s}{\xbar_s} })
%         \end{equation*}
         However, this last expression approaches $0$ as $s\to 0$ by property~\eqref{eqn: supporting plane collapses} of Lemma~\ref{lem: chopping convergence}, hence by this contradiction we obtain the theorem.
\end{proof}

    %%%%%%%%%% NO EXTREMALS POINTS OUTSIDE OMEGA_0 %%%%%%%%%%%%%%%%
 Ruling out $c$-extremal points in $\outerdom^{\interior} \setminus \innerdom$ is simpler and only requires use of the upper bound Theorem~\ref{thm: aleksandrov}.
    \begin{lem}\label{lem: localization beyond innerdom}
        Let $u$ be a $c$-convex Aleksandrov solution of~\eqref{OT problem}. If $\mountain_0$ is a $c$-function with focus $\xbar_0\in\innertarget$ that is supporting to $u$ at a point in $(\innerdom)^{\interior}$, then the set
        \begin{equation*}
             \sublevelset_0:=\{ u =\mountain_0\}
        \end{equation*}
        contains no $c$-extremal points with respect to $\xbar_0$ in $\outerdom^{\interior} \setminus \innerdom$.         	
    \end{lem}
  
    \begin{proof}
	     Suppose the statement is not true. Again, by Theorem~\ref{thm: exposed points} we can assume there exists $x_e$, a $c$-exposed point of $\sublevelset_0$ with respect to $\xbar_0$ in $\outerdom^{\interior} \setminus \innerdom$. We can fix $0<\delta<d\left(p_{(x_e, \xbar_0)}, \partial \left(\outerdomcoord{\xbar_0}\setminus\innerdomcoord{\xbar_0}\right)\right)$ and apply Lemma~\ref{lem: chopping convergence} to $x_e$, this time we obtain $c$-functions $\mountain_s$ and $\sublevelset_s$ compactly contained in $\outerdom^{\interior} \setminus \innerdom$ by~\eqref{eqn: chopped sections close to exposed point 2} (in particular, $\sublevelset_s \cap \innerdom = \emptyset$ for all $s$).
%	      Suppose that the lemma does not hold, then we may apply the same construction using Lemma~\ref{lem: localization claim} at the beginning of the proof of Theorem~\ref{thm: localization} above. In particular, for $s>0$ sufficiently small we obtain a $c$-function $\mountain_s$ for which the sublevel set $\sublevelset_s:=\{u\leq \mountain_s\}$ is contained in a small neighborhood of $x_e$, i.e. we can guarantee that
%         \begin{align}\label{eqn: chopped section outside support}
%              \sublevelset\cap \innerdom=\emptyset
%         \end{align}
%         and that $\sublevelset$ is compactly contained in $\outerdom$. At the same time, by property~\eqref{eqn: exposed point is inside chopped section} we see $x_e\in\sublevelset^{\interior}$, and in particular $\Leb{\sublevelset}>0$.
         Then, applying Theorem~\ref{thm: aleksandrov} to $u$, $\sublevelset_s$, $\mountain_s$, and $x_e$, and using~\eqref{eqn: main} we obtain
         \begin{align*}
              (\mountain_s(x_e)- u(x_e))^n  &\leq C\Leb{\sublevelset_s}\Leb{\partial_cu(\sublevelset_s)}\\
              &\leq C\Leb{\sublevelset_s}\Leb{\sublevelset_s\cap \innerdom}=0.
         \end{align*}
         However, this is impossible since $\mountain_s(x_e)-u(x_e)>0$ by~\eqref{eqn: chopped sections have interior} of Lemma~\ref{lem: chopping convergence}, and the claim is proved.
%         However, this contradicts~\eqref{eqn: chopped section outside support} and the lemma is proved.
    \end{proof}

    %%%%%%%%% NO (INTERIOR) EXTREMAL POINTS IN OMEGA %%%%%%%%%%%%%%

    Finally, the third case is dealt following an idea of Figalli, Kim, and McCann \cite{FKM11}, it consists in applying the bounds from Theorem~\ref{thm: aleksandrov} and Lemma~\ref{lem: sharp growth estimate} to two different, but related sublevel sets of $u$. The issue is that when we use the Aleksandrov solution property of $u$, we require a fixed proportion of the set $\arbitrary$ in Lemma~\ref{lem: sharp growth estimate} to be contained in $\innerdom$.

    \begin{lem}\label{lem: localization in Omega}
         Let $u$ solve~\eqref{eqn: main}, if $\mountain_0$ is a $c$-function with focus $\xbar_0\in\innertarget$ that is supporting to $u$ at a point in $(\innerdom)^{\interior}$, then the set
         \begin{equation*}
              \sublevelset_0:=\{ u =\mountain_0\} 
         \end{equation*}
         contains no $c$-extremal points with respect to $\xbar_0$ in $\partial \innerdom$.
    \end{lem}

    \begin{proof}
         Suppose the lemma does not hold, by Theorem~\ref{thm: exposed points} again, we can assume there is a $c$-exposed point $x_e$ of $\sublevelset_0$ with respect to $\xbar_0$, near $\partial(\innerdom)$. Since exposed points are also extremal points and there exists a point $x_0\in\sublevelset_0\cap (\innerdom)^{\interior}$, we may apply Theorem~\ref{thm: localization} and Lemma~\ref{lem: localization beyond innerdom} to see that we must have $x_e\in\partial(\innerdom)$. Since $\innerdom$ is compactly contained in $\outerdom$, we can choose a $\delta>0$ small enough to obtain a family of $c$-functions $\mountain_t$ with focus $\xbar_t$ and unit vectors $\e{t} \in \cotanspMbar{\xbar_t}$ as provided to us by Lemma~\ref{lem: chopping convergence} all such that $\sublevelset_t := \{ u\leq \mountain_t\}$ are compactly contained in $\outerdom$. Additionally, by choosing $\delta$ sufficiently small, by property~\eqref{eqn: chopped sections close to exposed point} we may also assume that $x_0\not\in\sublevelset_t$ for any $t>0$.

	     We now use $x_0$ to define a related family of sublevel sets,	
         \begin{align*}
              \bigsublevelset_t:=\{u\leq \mountain_t+u(x_0)-\mountain_t(x_0)\}.
         \end{align*}
         Note here that $\sublevelset_t\subset\bigsublevelset_t$, and $\sublevelset_t$ is compactly contained in $\outerdom$ for every $t$, but $\bigsublevelset_t$ may intersect $\partial \outerdom$. By the same argument as~\cite[Lemma 3]{C92} (or see~\cite[Theorem 7.1]{FKM11} for a detailed proof of this claim), there exists some large constant $\Lambda_0>0$ (independent of $t$) such that $\ellipsoid_t\subset(\innerdom)^{\interior}\cap\bigsublevelset_t$, where $\coord{\ellipsoid_t}{\xbar_t}$ is some translation of the dilation by $\tfrac{1}{\Lambda_0}$ of the John ellipsoid of each $\coord{\bigsublevelset_t}{\xbar_t}$. Then, we can apply Lemma~\ref{lem: sharp growth estimate} to $\bigsublevelset_{t}$ and $\mountain_t+u(x_0)-\mountain_t(x_0)$ for each $t>0$ small, with the choice of $\arbitrary=\left(\tfrac{1}{2M}\ellipsoid_t\right)^{\xbar_t}$, to obtain
         \begin{align*}
              \sup\limits_{\bigsublevelset_t} \left (\mountain_t+u(x_0)-\mountain_t(x_0)-u \right )^n&\geq C\Leb{\bigsublevelset_t}\Leb{\partial_cu(\left(\tfrac{1}{2M}\ellipsoid_t\right)^{\xbar_t})}\\
              &\geq C\Leb{\bigsublevelset_t}\Leb{\left(\tfrac{1}{2M}\ellipsoid_t\right)^{\xbar_t}\cap\innerdom}\\
              &= C\Leb{\bigsublevelset_t}\Leb{\left(\tfrac{1}{2M}\ellipsoid_t\right)^{\xbar_t}}\\
              &\geq C\Lambda_0^{-n}\Leb{\bigsublevelset_t}^2.
         \end{align*}
	     Note that we do not require compact containment of $\bigsublevelset_t$ in $\outerdom$ in order to invoke Lemma~\ref{lem: sharp growth estimate}.
	
	      Since each $\sublevelset_t$ is compactly contained in $\innerdom$ by~\eqref{eqn: chopped sections close to exposed point 2}, we can apply Theorem~\ref{thm: aleksandrov} with each $\e{t}$ direction to these sets. Similar to the proof of Theorem~\ref{thm: localization} we arrive at the inequality:
	     \begin{align*}
              \left(\frac{\mountain_t(x_e)-u(x_e)}{\sup \limits_{\bigsublevelset_t} (\mountain_t+u(x_0)-\mountain_t(x_0)-u)}\right)^n
               &\leq 	Cd(p_{(x_e,\xbar_t)}, \plane{\e{t}}{\coord{\sublevelset_t}{\xbar_t}} \cup \plane{-\e{t}} {\coord{\sublevelset_t}{\xbar_t} })\left(\frac{\Lambda_0^n\Leb{\sublevelset_t}^2}{\Leb{\bigsublevelset_t}^2}\right).
         \end{align*}
	     This time, since $\sup \limits_{\bigsublevelset_t} {(\mountain_t+u(x_0)-\mountain_t(x_0)-u)}=\sup \limits_{\sublevelset_t} {(\mountain_t-u)}+u(x_0)-\mountain_t(x_0)$, by property~\eqref{eqn: lower bound on localization ratios} we see the quantity on the left side is bounded below by a uniform, positive constant. At the same time since $\Leb{\sublevelset_t}\leq \Leb{\bigsublevelset_t}$, the quantity on the right approaches $0$ as $t\to 0$ by property~\eqref{eqn: supporting plane collapses}. This contradiction completes the proof.
    \end{proof}

%%%%%%%%%%%%%%%%%%%%%%%%%%%%%%%%%%%%%%%%%%%%%%%%%%%%%%%%%%%%%%%%%%%%%%%%%%%%%%%%%%%%%%%%%%
%%%%%%%%%%%%%%%%%%%%%%%%%%%%%%%%%%%%%%%%%%%%%%%%%%%%%%%%%%%%%%%%%%%%%%%%%%%%%%%%%%%%%%%%%%
\section{Strict $c$-convexity}\label{section: strict convexity}
%%%%%%%%%%%%%%%%%%%%%%%%%%%%%%%%%%%%%%%%%%%%%%%%%%%%%%%%%%%%%%%%%%%%%%%%%%%%%%%%%%%%%%%%%%
%%%%%%%%%%%%%%%%%%%%%%%%%%%%%%%%%%%%%%%%%%%%%%%%%%%%%%%%%%%%%%%%%%%%%%%%%%%%%%%%%%%%%%%%%%

    In this section, we will use the results of Section \ref{section: localization} to show that a $c$-convex function $u$ solving \eqref{eqn: main} must actually be strictly $c$-convex. We note that the overall structure of the paper is somewhat reversed from the approach in~\cite{FKM11}. Figalli, Kim, and McCann first rule out $c$-extremal points of contact sets on $\partial \outerdom$, and then show that any contact set with more than one point must extend to $\partial\outerdom$. In our approach, we have first ruled out $c$-extremal points in $\outerdom^{\interior}$, and we will use this fact in an essential way to show that a contact set cannot stretch to $\partial \outerdom$. 
    
    Roughly speaking, we wish to construct a certain family of cones in $\innertarget$ depending on a parameter $r>0$, whose vertex is given by the focus of the $c$-function defining the contact set (this family will ``close down'' upon its axis as $r\to 0$). We then show that for $r>0$ sufficiently small, the preimages of these cones under the $c$-subdifferential map lie outside of $\innerdom$, obtaining a contradiction with the main equation~\eqref{eqn: main}, it is in obtaining this property that we use the results from Section~\ref{section: localization} (see Figures~\ref{figure: section 6a} and~\ref{figure: section 6b}).
    
    We work toward showing the following theorem.
    \begin{thm}\label{thm: strict c-convexity}
         Let $u$ be a $c$-convex Aleksandrov solution of~\eqref{OT problem} and suppose that $\mountain_0$ is a $c$-function that is supporting to $u$ from below with focus $\xbar_0\in\innertarget$. Define the contact set $\contact:=\{u=\mountain_0\}$, and suppose that $\contact$ contains a point $x_0\in(\innerdom)^{\interior}$. Then, 
         \begin{align*}
              \contact = \{x_0\}.
         \end{align*}
    \end{thm}
 
%    Before embarking on the proof of this proposition, we make a number of preliminary constructions. First, it will be convenient to introduce the modified cost function
%    \begin{align}\label{eqn: modified cost}
%         \ctil(p, \pbar):=c(\cExp{\xbar_0}{p}, \cExp{x_0}{\pbar})-c(\cExp{\xbar_0}{p}, \xbar_0)
%    \end{align}
%    for $p\in\outerdomcoord{\xbar_0}$ and $\pbar\in\outertargetcoord{x_0}$, and the corresponding modified solution $u$, given by 
%    \begin{align}\label{eqn: modified u}
%         \util(p):=u(\cExp{\xbar_0}{p})- \left (-c(\cExp{\xbar_0}{p}, \xbar_0)+c(x_0, \xbar_0)+u(x_0)\right ),
%    \end{align}
%    note that $\util$ is $\ctil$-convex. If we denote
%    \begin{align*}
%         p_0&:=p_{(x_0, \xbar_0)}\\
%         \pbar_0&:=\pbar_{(x_0, \xbar_0)}
%    \end{align*}
%    we then immediately find that 
%    \begin{align*}
%         \contactcoord=\{p\in\outerdomcoord{\xbar_0}\mid\util(p)=0\}
%    \end{align*}
%    and the $\ctil$-function $-\ctil(\cdot, \pbar_0)+\ctil(p_0, \pbar_0) \equiv 0 $ is supporting to $\util$ from below at $p_0$. Moreover, $\contactcoord$ is a convex set by \eqref{eqn: gLp}. 
    
    To this end, for the remainder of this section we will fix $\mountain_0$, $\xbar_0$, and $x_0$ as in the statement of the above theorem, and also write
    \begin{align*}
         p_0:=p_{(x_0, \xbar_0)}, \ 
         \pbar_0:=\pbar_{(x_0, \xbar_0)}.
    \end{align*}
    Additionally, in this section we will use the Riemannian inner product $\innerg[x_0]{\cdot}{\cdot}$ defined on $\cotanspM{x_0}$.  As will be seen in the following proofs, the actual choice of inner product on $\cotanspM{x_0}$ is irrelevant, we merely fix $\g{x_0}$ for concreteness.
% by
%    \begin{align*}
%         \innerg[x]{\pbar_1}{\pbar_2}:=\innergbar{\transform{x}{\xbar_0}(\pbar_1)}{\transform{x}{\xbar_0}(\pbar_2)},\qquad \forall\; \pbar_1, \pbar_2\in \cotanspM{x}.
%    \end{align*}
%\begin{rem}\label{rem: modified c-subdifferential}
% It is an elementary computation to verify that
%\begin{align*}
% \pbar\in\partial_{\ctil}\util(p) \iff \cExp{x_0}{\pbar}\in\partial_c u(\cExp{\xbar_0}{p})
%\end{align*}
%for any $p\in\outerdomcoord{\xbar_0}$ and $\pbar\in\outertargetcoord{x_0}$. Thus by~\eqref{Twist} and~\eqref{Nondeg}, we can also see that if $u$ satisfies~\eqref{eqn: main}, we must have
%\begin{align*}
% \Leb{\coord{\arbitrary}{\xbar_0}\cap\innerdomcoord{\xbar_0}}\sim\Leb{\arbitrary\cap\outerdom_0}\sim \Leb{\partial_cu(\arbitrary)}\sim\Leb{\partial_{\ctil}{\util\left(\coord{\arbitrary}{\xbar_0}\right)}}
%\end{align*}
%for any measurable $\arbitrary\subset \outerdom$, where the constants of proportionality are universal, i.e. $\util$ is an Aleksandrov solution with cost function $\ctil$ on the domain $\innerdomcoord{\xbar_0}$.
%\end{rem}
    
    We will also use the following result proven in~\cite[Theorem 2.2.9]{Sch93}.    
    \begin{lem}\label{lem: rob schneider's lemma}
     Suppose that $\vectsp$ is a $k$--dimensional normed vector space, and $\arbitrary\subset \vectsp$ is an $k$--dimensional convex subset. Then, the subset of the unit sphere in the dual space $\vectsp^*$ consisting of linear functions that do not attain a unique maximum over $\arbitrary$ has zero $(k-1)$--dimensional Hausdorff measure. In particular, the set of linear functions which do attain a unique maximum over $\arbitrary$ is dense in $\vectsp^*$.
     \end{lem}
    We are now ready to state and prove our first lemma, where we single out a certain direction to be used as the axis of our family of cones. Very roughly, we start with a direction that points inward to $\innertarget$, project this onto the $k$-dimensional affine hull of $\contactcoord$ to apply Lemma~\ref{lem: rob schneider's lemma} there, and show that the result of ``unprojecting'' the vector is the correct direction. Note that this lemma only utilizes convex geometry, and does not  require any properties of the solution $u$.
    \begin{lem}\label{lem: good direction a3w}
         Suppose that the conditions of Theorem~\ref{thm: strict c-convexity} hold and $\contact$ contains more than one point. Then there is some nonzero $\qbar_0 \in \cotanspMbar{x_0}$ such that
         \begin{align}\label{eqn: cone is in interior a3w}
              \left( \left(B_r(\pbar_0)\setminus B_{\frac{r}{2}}(\pbar_0)\right)\cap\cone{\qbar_0}{r}\right)\setminus\{ \pbar_0\}\subset \outertargetcoord{x_0}^{\interior}
         \end{align}
         for all sufficiently small and positive $r$. Here, $\cone{\qbar_0}{r}$ denotes the cone (see Figure~\ref{figure: section 6b})
         \begin{align}\label{eqn: cone def a3w}
              \cone{\qbar_0}{r}:=\left \{\pbar\in \outertargetcoord{x_0} \mid r\innerg[x_0]{\pbar-\pbar_0}{\frac{\qbar_0}{\gnorm[x_0]{\qbar_0}}}\geq \gnorm[x_0]{\proj{\qbar_0^\perp}{\pbar-\pbar_0}} \right\},
         \end{align}
         and $\proj{\qbar_0^\perp}{\pbar}$ is the projection of $\pbar$ onto the $(n-1)$--dimensional affine space containing $\pbar_0$, which is $\g{x_0}$-orthogonal to $\qbar_0$. Moreover, the linear function on $\outerdomcoord{\xbar_0}$ defined by
         \begin{align}\label{eqn: linear function def}
              \linear{}{p}:= \inner{\frac{\transform{x_0}{\xbar_0}(\qbar_0)}{\gnorm[x_0]{\qbar_0}}}{p}
         \end{align}
         attains a unique maximum on $\contactcoord$.
    \end{lem}

    \begin{proof}
         %First suppose that $\pbar_0\in\partial \outertargetcoord{x_0}$. By \RedComment{LOEPER}, the set $\outertargetcoord{x_0}$ is an $n$--dimensional convex subset of $\cotanspM{x_0}$ (hence with a nonempty interior), 
         Since $\outertargetcoord{x_0}$ has a nonempty interior, there exists a nonzero $\qbar_1\in \cotanspM{x_0}$ such that $\pbar_0+\qbar_1\in \outertargetcoord{x_0}^{\interior}$. By making a small perturbation of $\qbar_1$ and by~\eqref{Nondeg}, we may assume that the linear function on $\outerdomcoord{\xbar_0}$ which is defined by 
         \begin{align*}
              \linear[1]{}{\pbar}:=\inner{\transform{x_0}{\xbar_0}(\qbar_1)}{\pbar}
         \end{align*}
         is not identically constant on $\contactcoord$. Now, consider the linear subspace $\vectsp[k]:=\aff{\contactcoord}-p_0$ and the dual space of linear functions $\vectsp[k]^*\subset \tanspMbar{\xbar_0}$, endowed with the restriction of the inner product $\gbar{\xbar_0}$. By choosing an orthonormal basis $\{\e{i}\}_{i=1}^k$ of $\vectsp[k]$, we may define the projection of $\linear[1]{}{\cdot}$ by
         \begin{align*}
              \linear[1]{k}{p}:=\sum_{i=1}^k{\linear[1]{}{\e{i}}\innergbar{\e{i}}{p}}\in \vectsp[k]^*.
         \end{align*}
         Note that if $p\in\aff{\contactcoord}$, 
         \begin{align*}
              \linear[1]{}{p}- \linear[1]{k}{p}-\linear[1]{}{p_0}+\linear[1]{k}{p_0}&=\linear[1]{}{p-p_0-\sum_{i=1}^k{(\innergbar{\e{i}}{p-p_0}\e{i})}}=0,
         \end{align*}
         in other words $\linear[1]{}{\cdot}-\linear[1]{K}{\cdot}$ is identically constant on $\contactcoord$, hence by our choice of $\qbar_1$ we must have that $\linear[1]{k}{p}$ is not identically zero. Then by Lemma~\ref{lem: rob schneider's lemma}, there exists a $\vectn\in \vectsp[k]^*$ of arbitrarily small $\gbar{\xbar_0}$ norm such that $\linear[1]{K}{\cdot}+\inner{\vectn}{\cdot}$ attains a unique maximum over $\contactcoord.$ Thus, we may define 
         \begin{align*}
              \qbar_0:=\qbar_1+\transforminv{x_0}{\xbar_0}(\vectn)
         \end{align*}
          for which $\pbar_0+\qbar_0\in \outertargetcoord{x_0}^{\interior}$ and $\qbar_0\neq 0$, by taking the norm of $\vectn$ sufficiently small. Then if $\linear[]{}{\cdot}$ is the linear function defined by~\eqref{eqn: linear function def} we find that for $p\in \contactcoord$,
         \begin{align*}
              \linear[]{}{p}&=\linear[1]{}{p}+\inner{\vectn}{p}\\
              &=\linear[1]{k}{p}+\inner{\vectn}{p}+\linear[1]{}{p_0}-\linear[1]{k}{p_0},
         \end{align*}
         hence $\linear[]{}{\cdot}$ attains a unique maximum on $\contactcoord$.
% By~\eqref{Nondeg} we have $\transform{x_0}{\xbar_0}(\qbar_1)\neq 0$, hence we can easily adapt~\cite[Theorem 2.2.9]{Sch93} \BlueComment{(maybe we want to state it as a previous proposition?) } to see that there exists a nonzero $\qbar_0\in \cotanspM{x_0}$ such that $\linear{}{\cdot}$ attains a unique maximum over the convex set $\contactcoord$, 
%         \begin{align*}
%              \pbar_0+\qbar_0\in\outertargetcoord{x_0}^{\interior}.
%         \end{align*}

         We will now show the inclusion~\eqref{eqn: cone is in interior a3w}. If $\pbar_0\in\outertargetcoord{x_0}^{\interior}$, this inclusion is immediate for $r>0$ sufficiently small, so assume that $\pbar_0\in\partial\outertargetcoord{x_0}$. Since $\pbar_0+\qbar_0 \in \outertargetcoord{x_0}^{\interior}$, there exists some $r_0>0$ such that $B_{r_0}(\pbar_0+\qbar_0)\subset \outertargetcoord{x_0}^{\interior}$. Then by considering the convex hull of $B_{r_0}(\pbar_0+\qbar_0)$ and the point $\pbar_0$, and since $\outertargetcoord{x_0}$ is convex, we can easily obtain the inclusion~\eqref{eqn: cone is in interior a3w} for all $r>0$ sufficiently small.

    \end{proof}

         %there exists a non-negative Lipschitz function $\gamma$ defined on $\proj{\qbar_0^\perp}{\cotanspM{x_0}}$ with $\gamma(\pbar_0) =0$, such that for some $\rho_0>0$ we have
         %\begin{align*}
              %\outertargetcoord{x_0}^{\interior}\cap B_{\rho_0}(\pbar_0)=\{\pbar\in B_{\rho_0}(\pbar_0) \mid \innerg{\pbar-\pbar_0}{\qbar_0}>        \gamma(\proj{\qbar_0^\perp}{\pbar-\pbar_0})\}.
         %\end{align*}
         %Thus, if $r<\min{\{\rho_0,\ \lVert \gamma\rVert_{C^{0, 1}}^{-1}\}}$, we obtain the inclusion~\eqref{eqn: cone is in interior a3w}.
%         \begin{center}
%     \includegraphics[height=.35\textwidth]{Figures/Section6fig1.png}
%\end{center}
\begin{figure}[H]
  \centering
    \includegraphics[height=.35\textwidth]{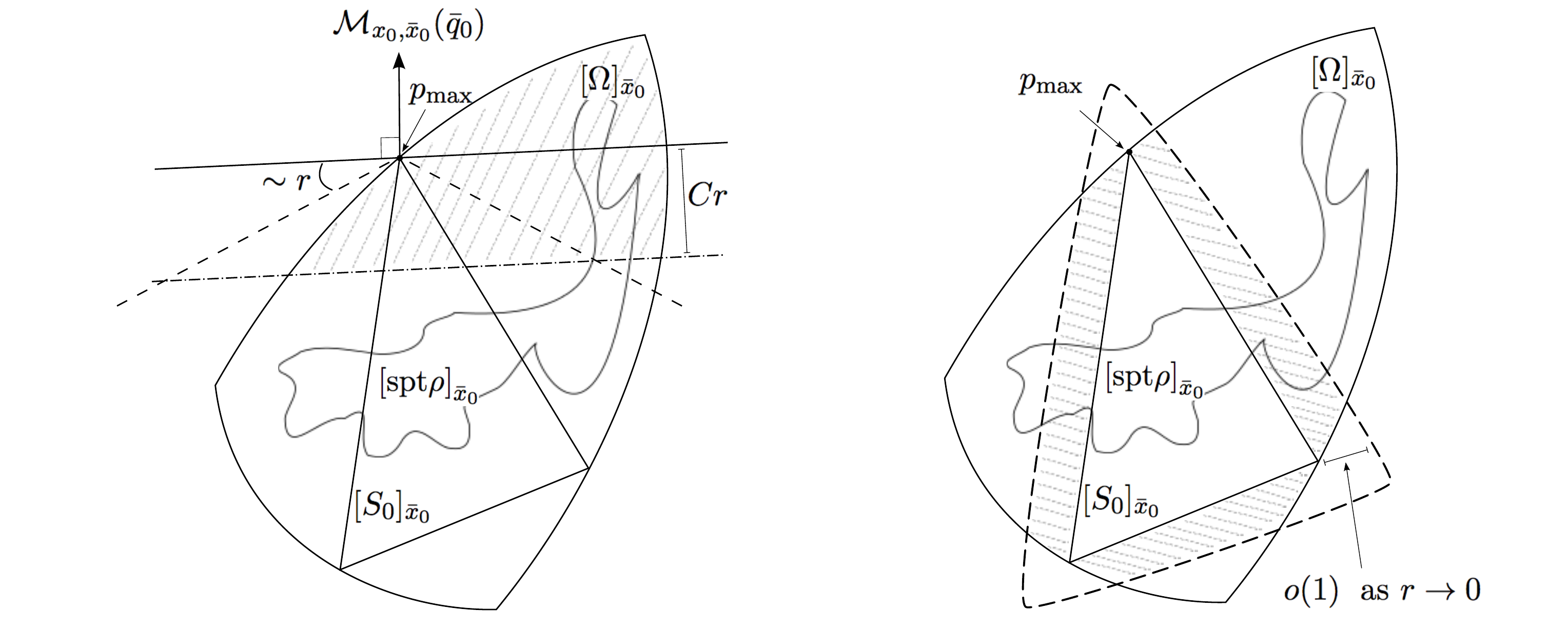}
     \caption{We first trap the preimage of the family of cones in a ``dual cone'' that flattens as $r\to 0$ (left diagram). Then, we show it must converge to the contact set $\sublevelset_0$ as $r\to 0$ (right diagram). Note that each diagram on its own is not enough to conclude that $\innerdom$ does not intersect the preimage of the family of cones we have constructed, even for $r>0$ small, and we must combine the two as in Figure~\ref{figure: section 6b}.}\label{figure: section 6a}
\end{figure}
In the next two results, we will show that the preimage under $\partial_cu$ of the family of cones constructed above must be close to the point $\xmax$. In the first lemma, we show the image is contained in some ``dual cone'' in $\outerdom$ (left image in Figure~\ref{figure: section 6a} above). We utilize the results from Section~\ref{section: localization} here.

    \begin{lem}\label{lem: halfspace inequality}
         Suppose $\qbar_0\in \outertargetcoord{x_0}$ is chosen as in Lemma~\ref{lem: good direction a3w} above, $\linear{}{p}$ is defined by~\eqref{eqn: linear function def}, and $\contact$ contains more than one point. Then if $\pmax\in\contactcoord$ is the unique point where $\linear{}{\cdot}$ attains its maximum over $\contactcoord$, we have
         \begin{align}\label{eqn: max at boundary a3w}
              \pmax\in \contactcoord\cap \partial \outerdomcoord{\xbar_0}.
         \end{align} 
         Additionally, for all sufficiently small positive $r$ and some constant $C>0$ independent of $r$, if $x\not \in\contact$ and
         \begin{align*}
              \coord{\partial_{c} u(x)}{x_0}\cap\cone{\qbar_0}{r}\cap \left(B_r(\pbar_0)\setminus B_{\frac{r}{2}}(\pbar_0)\right)\neq \emptyset,
         \end{align*}
         then we have the inequality
         \begin{align}\label{eqn: halfspace inequality a3w}
              \linear{}{p_{(x, \xbar_0)}}> \linear{}{\pmax}-Cr.
         \end{align}
    \end{lem}

    \begin{proof}
         Since the extrema of a linear function on a convex set must be attained at at least one of its extremal points, $\pmax$ must be an extremal point of $\contactcoord$. However, since $\contact$ contains more than one point by assumption, we may apply Theorem~\ref{thm: localization}, and Lemmas~\ref{lem: localization beyond innerdom} and~\ref{lem: localization in Omega} to conclude that $\pmax$ cannot be in $\outerdomcoord{\xbar_0}^{\interior}$, this proves~\eqref{eqn: max at boundary a3w}.
        
                We now work towards the inequality~\eqref{eqn: halfspace inequality a3w}. Fix some $r>0$ and $x\not \in \contact$, and write $p:=p_{(x, \xbar_0)}$. Then if $\pbar_r\in \coord{\partial_{c} u(x)}{x_0}\cap \cone{\qbar_0}{r}\cap \left(B_r(\pbar_0)\setminus B_{\frac{r}{2}}(\pbar_0)\right)$ we must have (writing $\xbar_r:=\cExp{x_0}{\pbar_r}$)
         \begin{align}\label{eqn: supporting inequality a3w}
              u(y)\geq -c(y, \xbar_r)+c(x, \xbar_r)+u(x)>-c(y, \xbar_r)+c(x, \xbar_r)+\mountain_0(x)
         \end{align}
         for any $y\in\outerdom$. Now, we define the $c$-segment
         \begin{align*}
              \xbar(t):=\cExp{x_0}{(1-t)\pbar_r+t\pbar_0},
         \end{align*}
%         where
%         \begin{align*}
%              p(t):=(1-t)p_r+t\pmax,
%         \end{align*}
         and let us also write
         \begin{align*}
              \xmax:=\cExp{\xbar_0}{\pmax}
         \end{align*}
         for ease of notation. Then, by taking $y=\xmax$ in~\eqref{eqn: supporting inequality a3w} and applying Taylor's theorem in $t$, we calculate
         \begin{align}
         0&=u(\xmax)-\mountain_0(\xmax)\notag\\
         &>-c(\xmax, \xbar_r)+c(x, \xbar_r)+\mountain_0(x)-\mountain_0(\xmax)\notag\\
         &=-c(\xmax, \xbar_r)+c(\xmax, \xbar_0)-(-c(x, \xbar_r)+c(x, \xbar_0))\notag\\
              &\geq \inner{-\Dbar c(\xmax, \xbar_0)+\Dbar c(x, \xbar_0)}{\xbardot(0)}-C\sup_{t\in[0,1]}{\lvert \frac{d^2}{dt^2}\left[c(\xmax, \xbar(t))-c(x, \xbar(t))\right]\rvert}\label{eqn: first inequality pt1}.
              \end{align}
              Now note that
\begin{align*}
 \norm{\frac{d^2}{dt^2}\left[-c(\xmax, \xbar(t))+c(x, \xbar(t))\right]}&=\norm{\frac{d}{dt}\inner{-\Dbar c(\xmax, \xbar(t))+\Dbar c(x, \xbar(t))}{\xbardot(t)}}\\
 &=\norm{\frac{d}{dt}\inner{\transformadj{x_0}{\xbar(t)}(-\Dbar c(\xmax, \xbar(t))+\Dbar c(x, \xbar(t)))}{\pbar_r-\pbar_0}}\\
 &\leq C\gnorm[x_0]{\pbar_r-\pbar_0}^2
\end{align*}
for some universal $C>0$, by applying the chain rule along with~\eqref{Nondeg}. Thus combining this with~\eqref{eqn: first inequality pt1} we continue calculating,
              \begin{align}
              %&\geq \inner{(\pbar_r-\pbar_0)}{\left[-D \Dbar c(x_r, \xbar_0)\right]^{-1}(\pmax-p_r)}-C\gnorm[x_0, \xbar_0]{\pbar_r-\pbar_0}\notag\\
              0&> \inner{\pmax-p}{\DDbarinv{x_0}{\xbar_0}(\pbar_r-\pbar_0)}-C\gnorm[x_0]{\pbar_r-\pbar_0}^2\notag\\
              &\geq \inner{\transform{x_0}{\xbar_0}(\pbar_r-\pbar_0)}{\pmax-p}-Cr^2\notag\\
              &=\inner{\transform{x_0}{\xbar_0}\left(\innerg[x_0]{\pbar_r-\pbar_0}{\frac{\qbar_0}{\gnorm[x_0]{\qbar_0}}}\frac{\qbar_0}{\gnorm[x_0]{\qbar_0}}+\proj{\qbar_0^\perp}{\pbar_r-\pbar_0}\right)}{\pmax-p}-Cr^2\notag\\
              &=\innerg[x_0]{\pbar_r-\pbar_0}{\frac{\qbar_0}{\gnorm[x_0]{\qbar_0}}} \inner{\frac{\transform{x_0}{\xbar_0}(\qbar_0)}{\gnorm[x_0]{\qbar_0}}}{\pmax-p}+\inner{\proj{\qbar_0^\perp}{\pbar_r-\pbar_0}}{\transformadj{x_0}{\xbar_0}(\pmax-p)}-Cr^2.\label{eqn: first inequality pt1}
         \end{align}
         Now note that by the definition of $\cone{\qbar_0}{r}$,  if $\innerg[x_0]{\pbar_r-\pbar_0}{\qbar_0}=0$ we would have $\pbar_r=\pbar_0$, which would contradict inequality~\eqref{eqn: supporting inequality a3w}. Hence we must actually have 
         \begin{align*}
              \innerg[x_0]{\pbar_r-\pbar_0}{\qbar_0}> 0.
         \end{align*} 
         At the same time since $\cone{\qbar_0}{r}\setminus B_{\frac{r}{2}}(\pbar_0)$,
\begin{align*}
 \frac{r^2}{4}&\leq\gnorm[x_0]{\pbar_r-\pbar_0}^2\\
 &=\innerg[x_0]{\pbar_r-\pbar_0}{\frac{\qbar_0}{\gnorm[x_0]{\qbar_0}}}^2+\gnorm[x_0]{\proj{\qbar_0^\perp}{\pbar_r-\pbar_0}}^2\\
 &\leq (1+r^2)\innerg[x_0]{\pbar_r-\pbar_0}{\frac{\qbar_0}{\gnorm[x_0]{\qbar_0}}}^2.
\end{align*}
Thus if $r$ is sufficiently small, we have 
\begin{align*}
 \innerg[x_0]{\pbar_r-\pbar_0}{\frac{\qbar_0}{\gnorm[x_0]{\qbar_0}}}\geq C^{-1}r
\end{align*}
for some universal constant $C>0$. Dividing~\eqref{eqn: first inequality pt1} by $\innerg[x_0]{\pbar_r-\pbar_0}{\frac{\qbar_0}{\gnorm[x_0]{\qbar_0}}}$, rearranging, and using that $\pbar_r\in \cone{\qbar_0}{r}$ along with~\eqref{Nondeg}, we obtain
         \begin{align*}
              \linear{}{p}&=\inner{\frac{\transform{x_0}{\xbar_0}(\qbar_0)}{\gnorm[x_0]{\qbar_0}}}{p}\\
              &>\inner{\frac{\transform{x_0}{\xbar_0}(\qbar_0)}{\gnorm[x_0]{\qbar_0}}}{\pmax}-C\lVert \transformadj{x_0}{\xbar_0}\rVert\gbarnorm[x_0]{\pmax-p}\left(\frac{\gnorm[x_0]                                     {\proj{\qbar_0^\perp}{\pbar_r-\pbar_0}}}{\innerg[x_0]{\pbar_r-\pbar_0}{\frac{\qbar_0}{\gnorm[x_0]{\qbar_0}}}}\right)\\
              &\qquad-C\frac{r^2}{\innerg[x_0]{\pbar_r-\pbar_0}{\frac{\qbar_0}{\gnorm[x_0]{\qbar_0}}}}\\
              &\geq \linear{}{\pmax}-Cr,
         \end{align*}
         hence~\eqref{eqn: halfspace inequality a3w} is proven.
    \end{proof}
\begin{figure}[H]
  \centering
    \includegraphics[height=.4\textwidth]{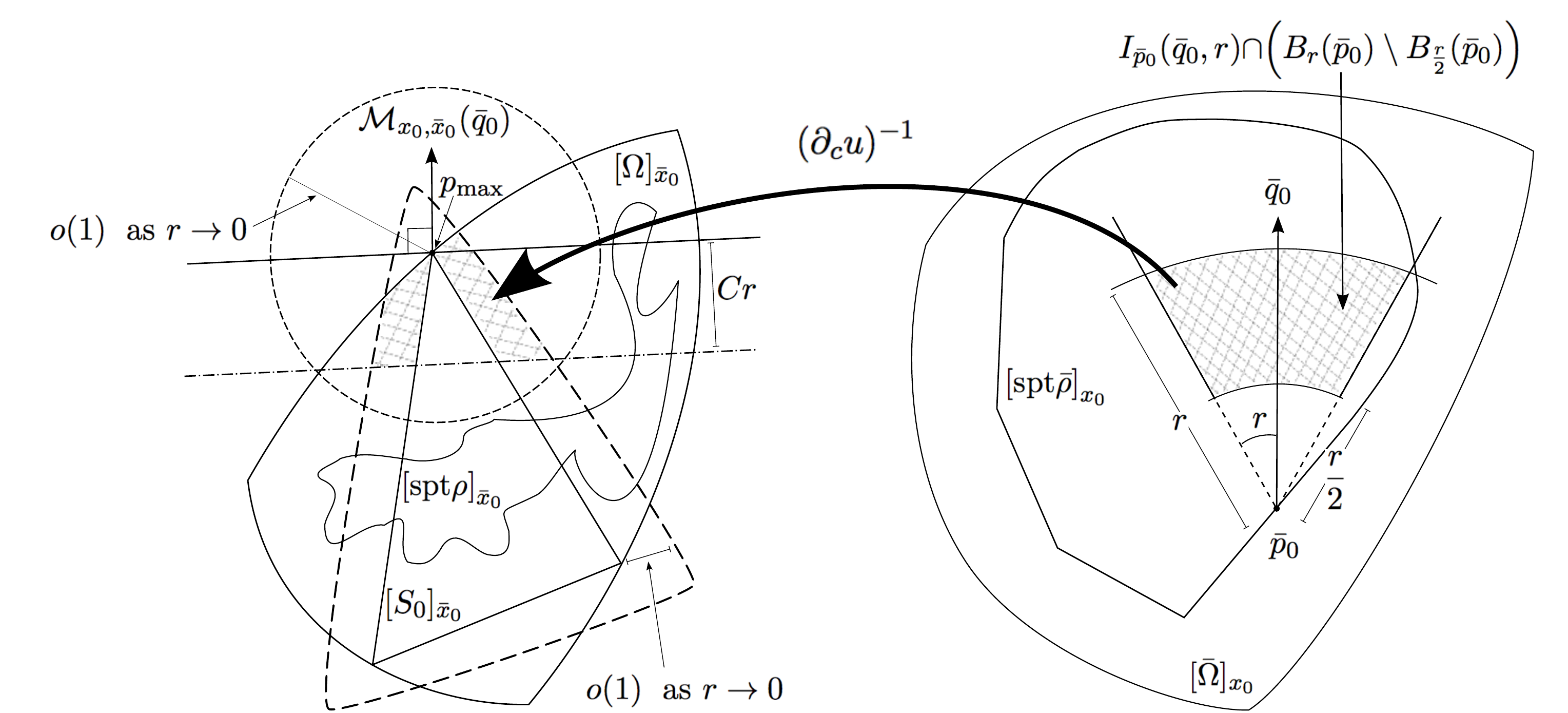}
     \caption{By combining Lemma~\ref{lem: halfspace inequality} with Corollary~\ref{cor: preimage trapping a3w}, we can show the family of cones (right diagram) has preimage approaching the point $\xmax$ as $r\to 0$ (left diagram). Note the left diagram combines the two illustrations in Figure~\ref{figure: section 6a}.}\label{figure: section 6b}
\end{figure}
    Next, we prove that as $r\to 0$, the inverse image under $\partial_cu$ of our family of cones must also be close to the contact set $\contact$ (right image in Figure~\ref{figure: section 6a}). By combining this with the above lemma, we can conclude that the inverse images must approach the point $\xmax$, or be contained in the contact set $\contact$ (see Figure~\ref{figure: section 6b} above).
 
    \begin{cor}\label{cor: preimage trapping a3w}
         Suppose that the conditions of Lemma~\ref{lem: good direction a3w} hold. Let $\qbar_0\in \cotanspM{x_0}$ and $\pmax\in \contactcoord\cap\partial\outerdomcoord{\xbar_0}$ satisfy the conclusions of Lemma~\ref{lem: good direction a3w}, and let $\cone{\qbar_0}{r}$ be as defined by~\eqref{eqn: cone def a3w}. Then given any $\epsilon>0$, there exists $r_\epsilon>0$ such that for any $x\in\outerdom^{\cl}\setminus \contact$ satisfying
         \begin{align*}
              \coord{\partial_{c}u(x)}{x_0}\cap \cone{\qbar_0}{r_\epsilon}\cap \left(B_{r_\epsilon}(\pbar_0)\setminus B_{\frac{r_\epsilon}{2}}(\pbar_0)\right)\neq \emptyset,
         \end{align*}
         we must have
         \begin{align*}%\label{eqn: close to max point}
              \gbarnorm[\xbar_0]{p_{(x, \xbar_0)}-\pmax}<\epsilon.
         \end{align*} 
    \end{cor}

    \begin{proof}
         Let $\linear{}{\cdot}$ be defined by~\eqref{eqn: linear function def}. Suppose that the corollary does not hold, then there exists some $\epsilon_0>0$, a sequence of positive numbers $r_k$ decreasing to $0$ as $k\to \infty$, and sequences of points $\{x_k\}_{k=1}^\infty\subset \outerdom^{\cl}\setminus \contact$, and $\pbar_k\in \cone{\qbar_0}{r_k}\cap \left(B_{r_k}(\pbar_0)\setminus B_{\frac{r_k}{2}}(\pbar_0)\right)$ such that 
         \begin{align}
              \pbar_k&\in\coord{\partial_{c} u(x_k)}{x_0},\notag\\
              \gbarnorm[\xbar_0]{p_{(x_k, \xbar_0)}-\pmax}&\geq \epsilon_0\label{eqn: contradiction assumption a3w}
         \end{align}
         for all $k$. It is clear that $\pbar_k\to\pbar_0$ as $k\to\infty$, and by the compactness of $\outerdomcoord{\xbar_0}^{\cl}$ we may extract a subsequence to assume that $x_k\to x_\infty$ for some $x_\infty\in \outerdom^{\cl}$. Let us write         
\begin{align*}
 \xmax:=\cExp{\xbar_0}{\pmax},\ \xbar_k:=\cExp{x_0}{\pbar_k}.
\end{align*}
 Since $\xmax\in \contact$, we can calculate that
         \begin{align*}
              \mountain_0(\xmax)&=u(\xmax)\\
              &\geq -c(\xmax, \xbar_k)+c(x_k, \xbar_k)+u(x_k)\\
              &\to -c(\xmax, \xbar_0)+c(x_\infty, \xbar_0)+u(x_\infty)\\
              &=\mountain_0(\xmax)-\mountain_0(x_\infty)+u(x_\infty),
% \\
% &=\inner{\pbar_k}{q_0}\inner{\pmax-p_k}{q_0}+\inner{\pmax-p_k}{\proj{q_0^\perp}{\pbar_k}}+u(p_k)\\
% &\geq \inner{\pbar_k}{q_0}\inner{\pmax-p_k}{q_0}-\diam{\outerdom}\lvert \proj{q_0^\perp}{\pbar_k}\rvert+u(p_k)\\
% &\geq \inner{\pbar_k}{q_0}\inner{\pmax-p_k}{q_0}-r_k \inner{\pbar_k}{q_0}+u(p_k)\\
%& \to \inner{\pbar_\infty}{q_0}\inner{\pmax-p_\infty}{q_0}+u(p_\infty)
         \end{align*}
         as $k\to \infty$, in other words, $x_\infty\in \contact$. At the same time, since each $x_k$ satisfies inequality~\eqref{eqn: halfspace inequality a3w} with $r=r_k$, by taking $k\to\infty$ we would obtain that $\linear{}{p_{(x_\infty, \xbar_0)}}\geq \linear{}{\pmax}$. However, since~\eqref{eqn: contradiction assumption a3w} implies that $p_{(x_\infty, \xbar_0)}\neq \pmax$, by the uniqueness of $\pmax\in \contactcoord$ as the point achieving the maximum value of $\linear{}{\cdot}$, we must have that $p_{(x_\infty, \xbar_0)}\not \in \contactcoord$, this contradiction completes the proof.
    \end{proof}
With this final result in hand, we can finally obtain a contradiction with the main equation~\eqref{eqn: main}, proving the desired result of strict $c$-convexity.
    \begin{proof}[Proof of Theorem~\ref{thm: strict c-convexity}]
         Suppose that $u$ fails to be strictly $c$-convex, thus the contact set $\contact$ contains more than one point. Since $\innerdom$ is assumed to be compactly contained in $\outerdom$, we may fix  
         \begin{align*}
              0<\epsilon <\dist{(\innerdomcoord{\xbar_0}, \partial \outerdomcoord{\xbar_0})}.
         \end{align*}
         We can now find some $\qbar_0\in \cotanspM{x_0}$ satisfying the conclusions of Lemma~\ref{lem: good direction a3w}, and by applying Corollary~\ref{cor: preimage trapping a3w} above we see that
         \begin{equation*}
              (\partial_{c} u)^{-1}(\cExp{x_0}{\cone{\qbar_0}{r_\epsilon}\cap \left(B_{r_\epsilon}(\pbar_0)\setminus B_{\frac{r_\epsilon}{2}}(\pbar_0)\right)}) \subset \contact\cup \left(\outerdom^{\cl}\setminus\innerdom\right)
         \end{equation*}
         for some $r_\epsilon>0$. Then by~\eqref{eqn: main} we can see that
         \begin{align*}
    	      \Leb{\cone{\qbar_0}{r_\epsilon}\cap B_{r_\epsilon}(\pbar_0)} &\leq \Leb{\partial_{c} u (\contact\cup \left(\outerdom^{\cl}\setminus\innerdom\right))} \\
	          &\leq C\Leb{\left(\contact\cup \left(\outerdom^{\cl}\setminus\innerdom\right)\right)\cap \innerdom}\\
	          &=0
         \end{align*}
         (we have also used here that $\Leb{\contact}=0$).
         However, by~\eqref{eqn: cone is in interior a3w} we see that $\Leb{\cone{\qbar_0}{r_\epsilon}\cap \left(B_{r_\epsilon}(\pbar_0)\setminus B_{\frac{r_\epsilon}{2}}(\pbar_0)\right)}$ has strictly positive measure, which leads to a contradiction, proving the proposition.
    \end{proof}
    With this result in hand, we may finally prove the main theorem.
    \begin{proof}[Proof of Theorem~\ref{thm: main}]
    Suppose $u$ is a $c$-convex Aleksandrov solution to~\eqref{OT problem}, let $x_0\in(\innerdom)^{\interior}$ and suppose $\mountain_0$ is a supporting $c$-function to $u$ at $x_0$ with focus $\xbar_0$. We must also have $\xbar_0\in\innertarget$, and by Theorem~\ref{thm: strict c-convexity} this implies that $u(x)>\mountain_0(x)$ for any $x\neq x_0$, i.e. $u$ is strictly $c$-convex at $x_0$.
\end{proof}

\appendix
\section{Inward pointing normals of convex sets}\label{appendix}
%We first recall some useful definitions from convex geometry (refer to rockafellar?)
%\begin{def}\label{def: normal cones}
%Suppose that $\arbitrary$ is a convex subset of an $n$-dimensional inner product space $V$ with inner product $\left(\cdot, \cdot\right)$, and $x_0\in \partial \arbitrary$. Then, we define the \emph{(strict) normal cone of $\arbitrary$ at $x_0$} by
%\begin{align*}
%\normal^0_{x_0}\left(\arbitrary\right):&=\{q\in V\mid \left( q, p\right) <0,\ \forall p\in \arbitrary\}\\
%\normal_{x_0}\left(\arbitrary\right):&=\{q\in V\mid \left( q, p\right) \leq 0,\ \forall p\in \arbitrary\}.
%\end{align*}
%If $\normal^0_{x_0}\left(\arbitrary\right)\neq \empty$, $x_0$ is called an \emph{exposed point} of $\arbitrary$.
% \end{def}
%\begin{rem}\label{rem: remark on normal cones}
%Note that both $\normal_{x_0}\left(\arbitrary\right)$ and $\normal^0_{x_0}\left(\arbitrary\right)$ are convex, while $\normal_{x_0}\left(\arbitrary\right)$ is always nonempty if $\arbitrary$ is convex, is closed and $1$-homogeneous, and contains $0$.
%\end{rem}
The results in this appendix are necessary to obtain the lower bound~\eqref{eqn: supporting plane collapses} on the family of line segments in Lemma~\ref{lem: chopping convergence}. In turn, this bound is needed to apply the Aleksandrov estimate Theorem~\ref{thm: aleksandrov} in the proofs of Theorem~\ref{thm: localization} and Lemma~\ref{lem: localization in Omega}. 

The idea is the following. We have freedom in choosing the direction $\e{0}$ to apply Lemma~\ref{lem: chopping convergence}. However, in order to obtain a strictly positive lower bound~\eqref{eqn: lower bound on segment}, we must be careful to select a $\e{0}$ for which the negative actually points into the sublevel set $\sublevelset_0$. Since it is not a priori obvious that such a choice of direction exists, this is what we aim to show. We note here that the main result of~\cite{FKM11a} by Figalli, Kim, McCann plays an analogous role in their paper~\cite{FKM11}. 

We begin by stating a well-known result in convex analysis, the Fenchel-Rockafellar Duality Theorem. Throughout the section, we will fix an $n$-dimensional inner product space $V$ with an inner product $(\cdot, \cdot)$.
\begin{thm}[Fenchel-Rockafellar Duality Theorem, {\cite{Roc66}}]\label{thm: fenchel-rockafellar}
 If $f$ and $g$ are convex functions on $V$ such that one of the functions is continuous at some point in $\{p\in V\mid f(p)+g(p)<+\infty\}$, then
\begin{align*}
 \inf_{p\in V}{(f(p)+g(p))}=\max_{p^*\in V}{(-f^*(-p^*)-g^*(p^*))},
\end{align*}
where $f^*$ is the usual Legendre-Fenchel transform,
\begin{align*}
 f^*(p^*):=\sup_{p\in V}{[\left(p^*,p\right)-f(p) ]}.
\end{align*}
\end{thm}
We will also need the concept of the indicator function of a set.
\begin{DEF}
If $G$ is a set, the \emph{indicator function} of $G$ is defined by
\begin{align*}
 \delta_G(p):=
 \begin{cases}
 0,&p\in G\\
 +\infty,&p\not\in G.
 \end{cases}
\end{align*}
If $G$ is convex and nonempty, $\delta_G$ is a proper, convex function, and if $G$ is closed $\delta_G$ is lower semi-continuous.
\end{DEF}
We first show essentially the desired result, but applied to the strict normal cone of a convex set with nonempty interior. The added structure of the cone allows us to obtain the result more easily.
\begin{lem}\label{lem: inward vector to normal cone}
Suppose that $\mathcal{\arbitrary}$ is a convex subset of an $V$ with nonempty interior, and that $p_e\in \partial \mathcal{\arbitrary}$ is an exposed point of $\mathcal{\arbitrary}$ (recall Definition~\ref{def: normal cones}). Then, there exists some $\w{0}\in \normal_{p_e}\left(\mathcal{\arbitrary}\right)\cap \S^{n-1}$ such that $\left( \w{0}, p\right) >0$ for all $p\in \normal_{p_e}\left(\mathcal{\arbitrary}\right)\cap \S^{n-1}$. Here $\S^{n-1}$ is the unit sphere in $V$.
\end{lem}
\begin{proof}
We may make a translation to assume that $p_e=0$. First we will show that $\normal_0\left(\mathcal{\arbitrary}\right)$ can be generated by its intersection with some plane that does not intersect the origin. Since $\mathcal{\arbitrary}^{\interior}\neq \emptyset$, we may assume that for some radius $r_0>0$ and center $p_0\neq 0$, there exists a ball $B_{r_0}(p_0)\subset \mathcal{\arbitrary}$. Now let $K$ be the cone generated by this ball, with $0$ as the vertex, i.e.
\begin{align*}
 K:=\{\lambda p\mid \lambda \geq 0,\ p\in B_{r_0}(p_0)\}.
\end{align*}
Since clearly $K\subset \mathcal{\arbitrary}$ and $0\in\partial K$, we immediately see that $\normal_0\left(\mathcal{\arbitrary}\right)\subset \normal_0\left(K\right)$, and the normal cone $\normal_0\left(K\right)$ is a cone with vertex $0$ and axial direction $\e{0}:=-\frac{p_0}{\norm{p_0}}$. Any $(n-1)$-dimensional hyperplane through $0$ that has normal vector orthogonal to $\e{0}$ would contain the point $p_0$, hence cannot be supporting to $K$. In other words, $\left (\e{0}, \w{}\right)>0$ for any $\w{}\in \normal_0\left(K\right)\setminus \{0\}$, and in particular, for any $\w{}\in \normal_0\left(\mathcal{\arbitrary}\right)\setminus\{0\}$. Thus by the homogeneity of normal cones, we see that the intersection $G:=\normal_0\left(\mathcal{\arbitrary}\right)\cap\{p\in V\mid \left (p, \e{0}\right)= 1\}$ generates $\normal_0\left(\mathcal{\arbitrary}\right)$. Moreover, it is easy to see that $G$ is compact and convex.

We will now obtain the desired $\w{0}$. Define the concave function $h(p^*):=\inf_{p\in G}{\left (p^*, p\right)}$. We wish to choose $f$ and $g$ in Theorem~\ref{thm: fenchel-rockafellar} so that $g^*$ is the indicator function of $G$, while $-f^*(-p^*)=h(p^*)$. To this end, define
\begin{align*}
 f(p):&=\delta_G(p),\\
 g(p):&=\sup_{p^*\in G}{\left (p^*, p\right)}=(\delta_G)^*(p).
\end{align*}
Since $G$ is convex and nonempty, both $f$ and $g$ are proper convex functions. Moreover, since $G$ is compact we see that $g$ is continuous and finite everywhere, thus we may apply Theorem~\ref{thm: fenchel-rockafellar} to $f$ and $g$. We can also calculate that indeed, $-f^*(-p^*)=h(p^*)$, while  
\begin{align*}
 g^*(p^*)=(\delta_G)^{**}(p^*)=\delta_G(p^*),
\end{align*}
since $\delta_G$ is lower semi-continuous by the closedness of $G$. Hence, we find
\begin{align*}
 \max_{p^*\in G}{h(p^*)}&=\max_{p^*\in V}{(-f^*(-p^*)-g^*(p^*))}\\
 &= \inf_{p\in V}{(f(p)+g(p))}\\
&=\inf_{p\in G}{\sup_{p^*\in G}{\left (p^*, p\right)}}\geq 1.
\end{align*}
By letting $\w{0}$ be the vector achieving the maximum in the expression on the left, normalized to unit length, we obtain the claimed properties.
\end{proof}
Finally, we can use a separation theorem to translate the above lemma into our main result.
\begin{lem}\label{lem: good supporting normal}
Suppose that $\mathcal{\arbitrary}\subset V$ is convex and contains more than one point, and $p_e\in \partial \mathcal{\arbitrary}$ is an exposed point of $\mathcal{\arbitrary}$. Then, there exists some $\e{0}\in \normal^0_{p_e}\left(\mathcal{\arbitrary}\right)\cap \S^{n-1}$ and $\lambda_0>0$ such that $p_e-\lambda \e{0}\in \mathcal{\arbitrary}$ for any $\lambda \in (0, \lambda_0]$.
\end{lem}
\begin{proof}
Again, assume that $p_e=0$. Since $\mathcal{\arbitrary}$ contains more than one point, its affine dimension must be strictly bigger than $0$. If the affine dimension of $\mathcal{\arbitrary}$ is strictly less than $n$, we may consider the orthogonal projection of $\mathcal{\arbitrary}$ onto its affine hull for the following proof, so without loss of generality assume that $\mathcal{\arbitrary}$ has affine dimension $n$. In particular, $\mathcal{\arbitrary}^{\interior}\neq \emptyset$ and we may choose an associated $\w{0}\in \normal_{0}\left(\mathcal{\arbitrary}\right)\cap \S^{n-1}$ with the property described in Lemma~\ref{lem: inward vector to normal cone} above. Our claim will be proven if we can show that $-\normal^0_{0}\left(\mathcal{\arbitrary}\right)\cap \mathcal{\arbitrary}^{\interior}\neq \emptyset$. Suppose this does not hold. Then by applying the separation theorems~\cite[Theorem 11.3 and 11.7]{Roc70} to $-\normal^0_{0}\left(\mathcal{\arbitrary}\right)$ and $\mathcal{\arbitrary}$, we obtain a unit length $\e{0}\in \S^{n-1}$ such that
\begin{align*}
\left(\e{0}, \w{}\right)&\geq 0,\qquad\forall \w{}\in -\normal^0_{0}\left(\mathcal{\arbitrary}\right),\\%\label{eqn: negative normal cone support}\\
\left(p, \e{0}\right)&\leq 0,\qquad \forall p\in \mathcal{\arbitrary}.%\label{eqn: arbitrary set support},
\end{align*}
Since $0$ is an exposed point of $\mathcal{\arbitrary}$ we have $\normal^0_{0}\left(\mathcal{\arbitrary}\right)\neq \emptyset$, and hence it can be seen that $\left(\normal^0_{0}\left(\mathcal{\arbitrary}\right)\right)^{\cl}=\normal_{0}\left(\mathcal{\arbitrary}\right)$. Thus as a result of the first inequality above, $\left(\e{0}, \w{0}\right)\leq 0$. However, the second inequality implies that $\e{0}\in\normal_0\left(\mathcal{\arbitrary}\right)$, which contradicts the choice of $\w{0}$, and we obtain the lemma.

\end{proof}
\nocite{B91,GM96,L09,MTW05}
\bibliography{transport}
\bibliographystyle{plain}
%%%%%%%%%%%%%%%%%%%%%%%%%%%%%%%%%%%%%%%%%%%%%%
\end{document}